\documentclass[10pt]{amsart}
\usepackage{etex}
\usepackage[margin=1.3in,marginparwidth=1.1in,marginparsep=0.1in]{geometry}
\usepackage{import}

\usepackage{graphicx}
\makeatletter
\DeclareRobustCommand{\cev}[1]{%
  {\mathpalette\do@cev{#1}}%
}
\newcommand{\do@cev}[2]{%
  \vbox{\offinterlineskip
    \sbox\z@{$\m@th#1 x$}%
    \ialign{##\cr
      \hidewidth\reflectbox{$\m@th#1\vec{}\mkern4mu$}\hidewidth\cr
      \noalign{\kern-\ht\z@}
      $\m@th#1#2$\cr
    }%
  }%
}
\makeatother

\usepackage{marginnote,color}

\usepackage{bm}
\usepackage{tikz-cd}
\usepackage[toc,title,page]{appendix}
\usepackage{yhmath}
\usepackage{tikz}\usepackage{float}
\usepackage{fancyhdr}
\usepackage{blkarray, bigstrut}
\usepackage[hidelinks]{hyperref}
\usepackage{setspace}
\usepackage{amstext} 
\usepackage{array}
\usetikzlibrary{calc}
\allowdisplaybreaks
\theoremstyle{plain}
\newtheorem{thm}{Theorem}[section]
\newtheorem{lem}[thm]{Lemma}

\newtheorem{prop}[thm]{Proposition}

\newtheorem{cor}[thm]{Corollary}
\theoremstyle{definition}
\newtheorem{defn}[thm]{Definition}

\theoremstyle{remark}
\newtheorem{rem}[thm]{Remark}
\theoremstyle{definition}
\newtheorem{ex}[thm]{Example}

\newtheorem{nt}[thm]{Notation}

\newcommand{\ZZ}{\mathbb{Z}}
\renewcommand{\SS}{\mathbb{S}}
\newcommand{\lag}{\langle}
\newcommand{\rag}{\rangle}

\newcommand{\hh}{,\hdots,}

\newcommand{\fa}{\forall}

\newcommand{\mM}{\mathscr{M}}
\newcommand{\cG}{\mathcal{G}}

\newcommand*{\shom}{\mathscr{H}\kern -.5pt om}

\newcommand{\cF}{\mathcal{F}}

\newcommand{\op}{\oplus}

\newcommand{\cd}{\cdot}

\let \op=\oplus
\let \ov=\overline

\let \:=\colon
\numberwithin{equation}{section}
\usepackage{amsfonts}

\usepackage{pdflscape}
\usepackage{tikz-cd}
\usepackage{amsmath,amssymb,latexsym,amsthm}
\usepackage{mathtools}
\usepackage{graphicx}
\usepackage{mathrsfs}
\usepackage{amscd}
\usepackage{color}

\DeclareMathOperator\rk{rank}
\DeclareMathOperator\Gr{Gr}
\DeclareMathOperator\GL{GL}

\DeclareMathOperator\PGL{PGL}

\DeclareMathOperator\id{id}

\DeclareMathOperator\End{End}

\DeclareMathOperator\aut{Aut}

\DeclareMathOperator\s{span}
\DeclareMathOperator\spec{Spec}

\DeclareMathOperator\SL{SL}

\DeclareMathOperator\h{Hom}

\input xy
\xyoption{all}
\DeclareRobustCommand{\SkipTocEntry}[5]{}
\begin{document}
\title{Degenerations of Grassmannians via lattice configurations}
\author{Xiang He, Naizhen Zhang}
\begin{abstract}
We study degenerations of Grassmannians constructed using convex lattice configurations in Bruhat-Tits buildings. Using techniques from quiver representations, we analyze their special fibers, which are explicitly described as quiver Grassmannians. For a class of lattice configurations, called the locally linearly independent configurations, we show that our construction coincide with Mustafin degenerations, thus generalizing a result of Faltings. In such cases, our analysis of special fibers also generalizes results of Cartwright et al. As an application, we prove a smoothing criterion for limit linear series on arbitrary reducible nodal curves.
\end{abstract}
\address[Xiang He]{Einstein Institute of Mathematics, Edmond J. Safra Campus, The Hebrew University of Jerusalem,
Givat Ram. Jerusalem, 9190401, Israel}
\email{xiang.he@mail.huji.ac.il}
\thanks{Xiang He is supported by the ERC Consolidator Grant 770922 - BirNonArchGeom. Naizhen Zhang is supported by the DFG Priority Programme 2026 ``Geometry at infinity" and was supported by the Methusalem Project Pure 
Mathematics at KU Leuven during the preparation of an earlier version of this work.}
\address[Naizhen Zhang]{Institut f\"ur Differentialgeometrie, Leibniz Universit\"at Hannover, Welfengarten 1, 30167 Hannover, Germany}
\email{naizhen.zhang@math.uni-hannover.de}
\maketitle
\tableofcontents

\section{Introduction}
In this paper, we introduce a class of projective schemes called \textit{linked flag schemes} (Definition \ref{defn:linked flag schemes}). They are projective schemes representing some moduli functors defined using convex lattice configurations in Bruhat-Tits buildings and produce flat degenerations in various situations. The idea to consider such constructions dates back to Mustafin (\cite{mustafin1978}) who, based on an earlier work of Mumford (\cite{mumford1972analytic}), studied certain flat degenerations of projective spaces \textemdash\  now known as Mustafin varieties (see \cite{cartwright2011mustafin}), in order to study non-archimedean uniformization of higher-dimensional varieties. Following a result of Faltings' (\cite{faltings2001toroidal}), Mustafin varieties are special cases of linked flag schemes, and this description gives a moduli interpretation for them. Much later on, H\"abich (\cite{habich2014mustafin}) generalizes the original definition of Mustafin in order to produce flat degenerations of classical flag varieties. However, it was not known whether Faltings' moduli interpretation for Mustafin varieties generalizes to this generality. It is thus very natural for us to study the problem from the opposite direction, i.e. if we consider schemes representing the natural generalization of Faltings' functors, when do they exhibit nice properties (eg. flatness, integrality, Cohen-Macaulayness)? 

\medskip 
The current paper focuses on a special case of linked flag schemes, namely \textit{linked Grassmannians}, which will be denoted by $LG_r(\Gamma)$, where $r$ is the dimension of the subspace and $\Gamma$ is a convex lattice configuration. We restrict our base schemes to be spectra of discrete valuation rings. 

\medskip
Historically, Osserman first came up with the notion of linked Grassmannians in \cite{Olls}, where his main motivation was to provide a more functorial construction of the moduli of degenerations of linear series on smooth projective curves, a.k.a. the moduli of \textit{limit linear series}. Although formulated differently, Osserman's linked Grassmannians are relative quiver Grassmannians over a base scheme, which is allowed to be an arbitrary integral, Cohen-Macaulay scheme, with respect to some desired ambient representations of the double quiver of a type-$A_n$ quiver. Later in \cite{Ohrk}, Osserman generalized this notion by allowing arbitrary underlying connected quivers, and called the new objects \textit{pre-linked Grassmannian}. However, unlike in the original version, these objects are less well-behaved in general; for example, they are not always flat over the base.\footnote{It is also not known how to impose further restrictions on the ambient representations so that the resulting pre-linked Grassmannians become flat.} Throughout this paper, we shall refer to Osserman's original definition of linked Grassmannian as \textit{Osserman's linked Grassmannian}; the terminology \textit{pre-linked Grassmannian} will also be reserved for Osserman's construction to avoid confusion.

\medskip 
Inspired by an earlier work of Hahn and Li (\cite{hahn2020mustafin}), which connects the study of Osserman's linked Grassmannian with classical Mustafin varieties and local models of Shimura varieties, our current formulation of linked Grassmannians has a couple of advantages. First of all, one may apply ingredients in Bruhat-Tits theory and tropical geometry to talk about structures which are otherwise harder to describe. Secondly, it shows a natural connection with quiver representations and thus allows for new techniques which weren't considered by Osserman or Hahn and Li. 
More precisely, for each convex lattice configuration $\Gamma$, the special fiber $LG_r(\Gamma)_0$ of the linked Grassmannian $LG_r(\Gamma)$ is a quiver Grassmannian associated to a finite quiver $(Q(\Gamma),J_{\Gamma})$ with relations induced from $\Gamma$, which parametrizes subrepresentations of a representation $M_{\Gamma}$ of $(Q(\Gamma),J_\Gamma)$. This allows us to translate many problems about the geometry of $LG_r(\Gamma)_0$ into constructing or classifying subrepresentations of $M_{\Gamma}$, eventually generalizing various results in the literature. Our main theorem is the following:
\medskip

\begin{thm}[Proposition \ref{prop:linked degeneration stratification}, Theorem \ref{thm:linked degeneration stratification}, \ref{thm:flatness}]\label{thm:intro 1}
Let $\Gamma$ be a locally linearly independent lattice configuration (Definition \ref{defn:locally linearly independent}).  
We obtain a stratification of $LG_r(\Gamma)_0$ indexed by subrepresentations of $M_{\Gamma}$ (which we completely classify). We describe concretely the closure of each stratum, and hence classify the irreducible components of $LG_r(\Gamma)_0$.

As a result, we show that $LG_r(\Gamma)$ is Cohen-Macaulay, integral, and flat over the base with reduced fibers. In particular, it 
coincides with the Mustafin degeneration.
\end{thm}
Our theorem generalizes Faltings' moduli interpretation for classical Mustafin varieties to a larger class of Mustafin degenerations. Moreover, our stratification of the special fiber and classification of its irreducible components extends the corresponding results in \cite{cartwright2011mustafin} for $r=1$. As a byproduct, we also obtained a classification of the irreducible components of Osserman's linked Grassmannian over an algebraically closed field (\ref{ex:olink_irr}), thus answering an open question in \cite{Olls}. As already pointed out by Hahn and Li (\cite[\textsection 2.6]{hahn2020mustafin}), the standard local model of Shimura varieties in \cite{gortz2001flatness} can be interpreted as linked Grassmannians for some specific convex lattice configurations, which are in general not locally linearly independent. 
In \cite{gortz2001flatness}, G\"ortz showed that such linked Grassmannians are flat over the base, with reduced fibers, and the irreducible components are normal with rational singularities. Although considering different lattice configurations and proved using distinct methods, our main theorem above can be seen as analogous to the main results of \textit{loc.cit.} 
\medskip

Besides the theoretical interest of its own right, our construction has direct application to the study of limit linear series. The theory of limit linear series was invented by Eisenbud and Harris \cite{EHL} to prove several results in Brill-Noether theory, while providing a simplified proof of the celebrated Brill-Noether Theorem. It turns out to be a powerful tool for studying the geometry of algebraic curves. For example, Eisenbud and Harris used it to establish the birational type of moduli spaces of curves $\mathcal M_g$ ($g\ge 24$) (\cite{harris1982kodaira,EHgen}). It was originally defined for curves of compact type, i.e. curves whose dual graph is a tree, and later generalized by Osserman (\cite{osserman2019limit}) to all nodal curves. Just as the moduli of linear series on a smooth projective curve are closed subschemes of Grassmannians, the construction of moduli of Osserman-style limit linear series involves linked Grassmannians. More precisely, Osserman defines auxiliary objects called \textit{linked linear series} (\cite{Ohrk}), whose moduli schemes we now know are (DVR-locally) closed subschemes of linked Grassmannians (Theorem \ref{thm:intro2}) and admit natural forgetful maps to moduli of limit linear series. 

\medskip
The classic smoothing theorem of Eisenbud and Harris (Theorem 3.4, \cite{EHL}) states that if the moduli space of limit linear series on a compact-type curve is dimensionally proper, then every limit linear series can be realized as the limit of linear series on nearby smooth curves. This was later extended by Osserman using his theory of linked Grassmannian, to curves of \textit{pseudo-compact type}, i.e. curves whose dual graphs  contain no cycle of length at least 3. 

\medskip
However, to generalize Osserman's result to other non-compact type curves, one has to consider more general linked Grassmannians. Very often, knowing the geometric properties (e.g. irreducibility) of the involved linked Grassmannians is crucial. 
In that regard, we have\footnote{For technical precision, see Section \ref{subsec: limit linear series}.}:
\begin{thm}[Proposition \ref{prop:moduli of lls and linked flags}, Theorem \ref{thm:smoothing of lls}]\label{thm:intro2}
For an arbitrary nodal curve, the moduli of linked linear series is a closed subscheme of a scheme projective over the base which is DVR-locally a linked Grassmannian associated to a convex lattice configuration.  

Moreover, when the aforementioned linked Grassmannians satisfy certain conditions, (e.g. when they are all locally linearly independent as in Theorem \ref{thm:intro 1}), and the moduli of limit linear series has correct dimension, all the limit linear series smooth out.  
\end{thm}
It is worth-mentioning that linked linear series have proved to be useful in studying more difficult question in Brill-Noether theory, such as the strong maximal rank conjecture (see \cite{FSMRC}). In \cite{liu2018strong}, it played an essential role in verifying the conjecture in certain cases which are recently shown to imply that the moduli spaces of curves $\mathcal M_{22}$ and $\mathcal M_{23}$ are of general type (\cite{farkas2020kodaira}). Since the study of linked linear series is very closely related to the study of linked Grassmannians, it provides extra motivation for further development of the current project.  

\medskip
On the other hand, recently the authors (joint with E. Cotterill) constructed the moduli of inclusion of limit linear series (\cite{cotterill2020secant}), which roughly considers degenerations of two-term flags of linear series on curves where the sub-series has fixed amount of base-points. (Equivalently, they can be seen as linear series on curves, whose images are exceptional with respect to their secant planes.) The notion of \textit{linked chains of flags} naturally arose in the main construction of \textit{loc. cit.} We believe that an in-depth study of (two-term) linked flag schemes will both simplify and generalize the existing constructions. 

\medskip 
In another related direction, our theory also provide potential tools for lifting divisors on the graph $G$ associated to a regular smoothing family $X$ (Definition \ref{defn:smoothing family}) to divisors on the generic fiber $X_\eta$ of the same rank. More precisely, this can be achieved by showing that the moduli space of limit linear series has expected dimension (\cite{osserman2016dimension,liu2018limit,he2018brill}),  lifting the divisor on $G$ to a limit linear series on the special fiber $X_0$, and applying the corresponding smoothing theorem of limit linear series. A similar approach can be found in \cite[\S 5]{he2019smoothing}. See also \cite[\S 10]{baker2016degeneration} for a survey on this problem, and \cite{he2018lifting} for results of lifting divisors while preserving both the rank and ramifications. We expect to extend the lifting results to graphs $G$ beyond those mentioned in \cite{baker2016degeneration} by applying our new smoothing criterion.



\subsection*{Roadmap}\addtocontents{toc}{\protect\setcounter{tocdepth}{1}}The plan for the remainder of this paper is as follows. In Section \ref{sec:mus} we first give the functor of points (\ref{defn:mustafin degeneration functor}) of the linked Grassmannian $LG_r(\Gamma)$ over a DVR associated to any finite convex lattice configuration $\Gamma$ (the notion also generalizes to linked flag schemes). We then establish the quiver representation-theoretic foundation for analyzing the topology and geometry of $LG_r(\Gamma)$. More concretely, we associate to $\Gamma$ a finite quiver with relation $(Q(\Gamma),J_{\Gamma})$ (\ref{defn:associated quiver}, \ref{lem:well-defined}) and describe its bound quiver algebra $A_{\Gamma}$ (\ref{prop:ambient representation 2}). Moreover, we obtain a representation $M_{\Gamma}$ of $A_{\Gamma}$ (\ref{prop:quiver representation and special fiber}, \ref{prop:ambient representation}, \ref{prop:ambient representation 2}), encoding the topology and geometry of the special fiber of $LG_r(\Gamma)$. The definition of \textit{locally linearly independent configurations} is given in \ref{defn:locally linearly independent}. We conclude Section \ref{sec:mus} with motivating examples including Osserman's linked Grassmannian (\ref{subsubsec:linked grassmannian}), the standard local model of certain Shimura varieties (\ref{subsubsec:local model}) and the construction of moduli of limit linear series (\ref{subsubsec:degenerate linear series}). 

\medskip
In Section \ref{sec:geometry of linked flag schemes}, we prove that when $\Gamma$ is locally linearly independent, $LG_r(\Gamma)$ is integral, flat over the base, and Cohen-Macaulay  with reduced fibers (\ref{thm:flatness}). As a result, it coincides with the Mustafin degeneration $M_r(\Gamma)$, for which the Cohen Macaulayness and reducedness of the special fiber was unknown. The main effort was devoted to proving the irreducibility, which involves classifying subrepresentations of $M_{\Gamma}$ (\ref{lem:decomp},  \ref{prop:indecomp of loc linear indep}), as well as the description of the irreducible components of the special fiber $LG_r(\Gamma)_0$ and their general points (\ref{thm:linked degeneration stratification}).  As a byproduct, we give a concrete stratification of $LG_r(\Gamma)_0$ (\ref{thm:linked degeneration stratification}), and provide a way to  calculate the dimensions of the strata using standard techniques from quiver representations (\ref{prop:quiver stratification}, \ref{lem:hom_dim}). The proof of the rest of our main theorem is by a similar argument as in the case of Osserman's linked Grassmannian (\cite{helm2008flatness}).



\medskip

In Section \ref{sec:connection to lls}, we establish a smoothing criterion for limit linear series on arbitrary reducible nodal curves (\ref{thm:smoothing of lls}). This is done in the following way: the moduli scheme $\mathcal G$ of such limit linear series on a family of curves is known to be a determinantal locus inside a particular projective scheme $\cG^2$. Since the data of multidegrees of a limit linear series give a tropically convex set (\ref {prop:tropical convexity of vgw0}), $\mathcal G^2$ admits a natural morphism $\tilde \pi$ from a projective scheme which is DVR-locally a linked Grassmannian (\ref{prop:moduli of lls and linked flags}). 
When $\tilde \pi$ is surjective and those linked Grassmannians are all irreducible, one can conclude that $\mathcal G$ has relative dimension no less than expected over the base of the family, and hence conclude the smoothing theorem as classically stated. We end by giving two examples (\ref{subsubsec:two component case}, \ref{subsubsec:cyclic curve}) where the surjectivity and irreducibility are verified to show the validity of our smoothing criterion, the second of which considers curves not of pseudo-compact type. 


\medskip
After the appearance of the first version of the paper on arxiv, we were kindly informed by Ulrich G\"ortz that the equational description of Mustafin degenerations we cited from \cite{habich2014mustafin} is flawed. This called for a significant re-writing of the paper, which led to the present version. We find it helpful to include G\"ortz's counterexample to the aforementioned description so that the curious readers may be aware of the issue. 
\subsection*{Notations and conventions}\addtocontents{toc}{\protect\setcounter{tocdepth}{1}}
\begin{nt}
Throughout the paper, $\kappa$ will always be an algebraically closed field, $R$ will always be a discrete valuation ring with residue field $\kappa$ and fraction field $K$, and $\pi$ will always be a uniformizer of $R$. We do not assume $\kappa$ is of characteristic zero unless otherwise stated. In Section \ref{sec:connection to lls}, we will assume $R$ to be also complete, but we do not make extra assumption on $R$ elsewhere. 
\end{nt}

\begin{nt}
We use $V$ to denote a $d$ dimensional vector space over $K$. We denote $\mathfrak B_d$ for the Bruhat-Tits building associated to $\PGL(V)$ (\cite[\S 6.9]{abramenko2008buildings}. See also \cite[\S 2]{cartwright2011mustafin}.). We also denote $\mathfrak B^0_d$ to be the set of homothety classes of lattices in $V$.\footnote{Recall that the affine building for $\PGL_n$ is also the affine building for $\SL_n$. } 
\end{nt}

\begin{nt}
For a graph $G$, we denote by $V(G)$ and $E(G)$ the set of vertices and edges of $G$, respectively. If $\ell$ is a path (e.g. a directed edge) in $G$, we denote $s(\ell)$, $t(\ell)$ for its source and target respectively. By a cycle we mean a path $\ell$ such that $s(\ell)=t(\ell)$ is the only repeating vertex. 
\end{nt}

\begin{nt}
Let $n$ be any positive integer. We denote $\mathbf n$ to be the dimension vector $(n\hh n)$ of representations of a given quiver $Q$.  We denote by $[n]$ the set $\{1,...,n\}$.
\end{nt}



\subsection*{Convention on quiver representations}\addtocontents{toc}{\protect\setcounter{tocdepth}{1}}
We refer the readers to \cite{assem2006elements} for basic theory of quiver representations. Throughout the paper, quiver Grassmannians are defined over 
 $\kappa$. Assume that we are given a representation $M$ of a quiver $Q=(Q_0,Q_1)$ with the data $(f_{\ell})_{\ell\in Q_1}$ of linear maps between the underlying vector spaces $(M_i)_{i\in Q_0}$. 
 A \textit{subrepresentation} of $M$ is represented by a tuple of vector spaces $(U_i)_{i\in Q_0}$ such that $U_i\subset M_i$ and $f_{\ell}(U_{s(\ell)})\subset U_{t(\ell)}$, $\fa \ell\in Q_1$. For a path $\ell'=\ell_n\cdots \ell_1$ in $Q$ where $\ell_i\in Q_1$ we denote $f_{\ell'}$ the compositions of all $f_{\ell_i}$. 
 
\subsection*{Convention on the path algebra}\addtocontents{toc}{\protect\setcounter{tocdepth}{1}}
Let $\kappa Q$ be the path algebra of a quiver $Q=(Q_0,Q_1)$. We adopt the right-to-left convention when describing multiplication in $\kappa Q$: let $\ell_1,\ell_2$ be two paths in $Q$. $\ell_2\cd \ell_1$ will be the element in $\kappa Q$ corresponding to the concatenation of the two paths if $t(\ell_1)=s(\ell_2)$, and 0 otherwise. Note that this is opposite to the convention in \cite{assem2006elements}. In particular, our path algebra $\kappa Q$ is the opposite ring of the corresponding definition in \textit{loc. cit.}, and therefore when citing results thereof, ``right modules'' are replaced by ``left modules''. We adopt this convention in consistency with our convention for composition of morphisms. Moreover, we shall denote $\epsilon_i$ for the idempotent element corresponding to a vertex $i\in Q_0$ and $R_Q$ to be the ideal of $\kappa Q$ generated by all the elements in $Q_1$, i.e. the \textit{arrow ideal} of $\kappa Q$ (\cite[Definition II.1.9]{assem2006elements}).

\subsection*{Acknowledgement}\addtocontents{toc}{\protect\setcounter{tocdepth}{1}}
The authors would like to thank Brian Osserman for useful conversations on the subject of the current paper as well as tireless instructions during their PhD studies. Special thanks to Ulrich G\"ortz for pointing out a mistake in a cited result in the first version of this paper and providing us with a counter-example to that result (see Appendix \ref{app:counter example}) as well as useful suggestions via personal correspondence. We would also like to thank the anonymous referees, whose suggestions both helped improve the quality of exposition and pointed us to a related research article which we were previously unaware of.

\section{Linked Grassmannian: Definition and Basic Properties}\label{sec:mus}
In this section, we start by giving the general definition of a linked flag scheme via the moduli functor it represents. Afterwards and throughout the rest of the paper, we shall focus on the special case of linked Grassmannians, and mostly on the cases corresponding to locally linearly independent lattice configurations (Definition \ref{defn:locally linearly independent}). This section also contains the representation-theoretic preliminaries needed in the rest of the paper. 



\subsection{Definition and basic properties}\addtocontents{toc}{\protect\setcounter{tocdepth}{2}}\label{subsec:defn and prop}
Recall that $R$ is a discrete valuation ring with fraction field $K$ and algebraically closed residue field $\kappa$. Fix a uniformizer $\pi$ of $R$. Let $V$ be a vector space of dimension $d$ over $K$. Let $\Gamma=\{[L_i]\}_{i\in I}\subset\mathfrak B^0_d$ be a \textit{convex} collection of homothety classes of lattices in $V$. This means that for any two lattices $L_1,L_2$ such that $[L_1],[L_2]\in\Gamma$, we have $[L_1\cap L_2]\in\Gamma$. Let $\underline d=(d_1,...,d_m)$ where $0<d_m<\cdots<d_1<d$ are positive integers.

\begin{nt}\label{nt:morphisms in a configuration}
Fix a set of representatives $\{L_i\}_{i\in I}$ of $\Gamma$. For each pair $(i,j)\in I^2$ let $n_{i,j}$ be the minimal integer such that $\pi^{n_{i,j}}L_i\subset L_j$. Denote by $F_{i,j}$ the map from $L_i$ to $L_j$ induced by multiplying with $\pi^{n_{i,j}}$. For each $i,j\in I$ denote $\ov L_i=L_i/\pi L_i$ and $f_{i,j}\colon \ov L_i\rightarrow \ov L_j$ the map induced by $F_{i,j}$.
\end{nt}

\begin{defn}\label{defn:mustafin degeneration functor}
Let $\mathcal{LF}_{\underline d}(\Gamma)$ be the functor on $R$-schemes $T$ such that a $T$-valued point of $\mathcal{LF}_{{\underline d}}(\Gamma)$ is a collection of rank-${\underline d}$ subbundles $\mathscr E^m_i\hookrightarrow \cdots \hookrightarrow\mathscr E^1_i$ of $L_i\otimes \mathcal O_T$, one for each $i\in I$, such that 
for each possible inclusion $\pi^k\colon L_a\hookrightarrow L_b$ where $k\in\ZZ$, the induced morphism $L_a\otimes \mathcal O_T\rightarrow L_b\otimes \mathcal O_T$ maps  $\mathscr E^j_a$ to $\mathscr E^j_b$ for all $j$. If $m=1$, we denote the functor by $\mathcal{LG}_{d_1}(\Gamma)$ instead.
\end{defn}


\begin{prop}\label{prop:mustafin degeneration functor representable}
The functor $\mathcal{LF}_{{\underline d}}(\Gamma)$ is represented by a scheme $LF_{\underline d}(\Gamma)$ projective over $R$ which is independent of the choice of representatives $L_i$.
\end{prop}
\begin{proof}
Let $\mathrm{Flag}_{\underline d}(L_i)$ be the flag scheme of $L_i$ over $\spec(R)$ with universal flag 
$$\mathcal E^m_i\hookrightarrow \cdots \hookrightarrow\mathcal E^1_i\hookrightarrow\mathcal O_i:= L_i\otimes \mathcal O_{\mathrm{Flag}_{\underline d}(L_i)}.$$
Then $LF_{\underline d}(\Gamma)$ is the closed subscheme of the  $R$-fiber product $\prod_{i\in I}\mathrm{Flag}_{\underline d}(L_i)$ which is the intersection of the vanishing loci of the composition of the morphisms
$$\mathcal E^j_{i_1}\hookrightarrow\mathcal O_{i_1}\xrightarrow{F_{i_1,i_2}}\mathcal O_{i_2}\rightarrow \mathcal O_{i_2}/\mathcal E^j_{i_2}$$
for all $1\leq j\leq m$ and $(i_1,i_2)\in I^2$. Hence $LF_{\underline d}(\Gamma)$ is projective.  Moreover, scaling $L_i$ gives an isomorphism between $LF_{\underline d}(\Gamma)$s with respect to different choices of representatives of lattice classes in $\Gamma$.
\end{proof}

\begin{defn}\label{defn:linked flag schemes}
We shall call $LF_{\underline d}(\Gamma)$ the \textit{linked flag scheme of index} ${\underline d}$ associated to $\Gamma$. We shall call it a \textit{linked Grassmannian} and denote by $LG_{d_1}(\Gamma)$ whenever $m=1$. 
\end{defn}
By definition, $LF_{\underline d}(\Gamma)$ is a closed subscheme of the $R$-fibered product  $ \prod_{i\in I}\mathrm{Flag}_{\underline d}(L_i)$
whose generic fiber is the ``diagonal" $\mathrm{Flag}_{\underline d}(V)$. A closely related notion is the \textit{Mustafin degenerations} introduced by H\"abich  \cite{habich2014mustafin}, which generalizes the notion of Mustafin varieties (\cite{mustafin1978}) and provides flat degenerations of flag varieties.

\begin{defn}{\cite[Definition 2.1]{habich2014mustafin}\footnote{In \cite{habich2014mustafin} the author gave an ``equational description" of the Mustafin degenerations below Definition 2.1. This is not necessarily true due to a counter-example provided to us by Ulrich G\"oertz. See Appendix \ref{app:counter example}.}}
The \textit{Mustafin degeneration} associated to $\Gamma$ is the scheme theoretic image of the natural morphism 
$\mathrm{Flag}_{\underline d}(V)\to \prod_{i\in I}\mathrm{Flag}_{\underline d}(L_i)$, where the product on the right is fibered over $R$.
We shall denote it by $M_{\underline d} (\Gamma)$ in general and $M_{d_1}(\Gamma)$ when $m=1$. 
\end{defn}

Note that Mustafin varieties are just the special case of Mustafin degenerations where $m=d_1=1$. By construction, we have a natural inclusion $\iota:M_{{\underline d}}(\Gamma)\to LF_{\underline d}(\Gamma)$ realizing the Mustafin degeneration as a closed subscheme of the corresponding linked flag scheme. When $m=d_1=1$, it is the well-known that $\iota$ is an isomorphism (cf. \cite{hahn2020mustafin}), and this is proved by Faltings \cite[\S  5]{faltings2001toroidal}:

\begin{thm}\label{thm:Faltings}\label{thm:mustafin variety and linked flags scheme}
We have $LG_1(\Gamma)=M_1(\Gamma)$ as schemes.
\end{thm}

From now on, we shall focus on the case $m=1$, i.e., linked Grassmannians, and replace $d_1$ with $r$.
We shall see that for certain class of $\Gamma$, $\iota$ is also an isomorphism of schemes in this case (Section \ref{sec:geometry of linked flag schemes}). In order to do so, we will first analyze the topology of the special fiber $LG_r(\Gamma)_0$ of $LG_r(\Gamma)$ via the technique of quiver representations. The latter turns out to be useful, because one can naturally associate a quiver with relations to any convex collection of homothety classes of lattices. Following Notation \ref{nt:morphisms in a configuration}, we define the $\Gamma$-\textit{weight} of a sequence $(i_1,...,i_s)$ in $I$ to be $\sum_{k=1}^{s-1} n_{i_k,i_{k+1}}$. We will simply call it the weight if the context is clear.

\begin{defn}\label{defn:associated quiver}
Following Notation \ref{nt:morphisms in a configuration}, the  quiver associated to a convex configuration $\Gamma$ is defined to be a pair $(Q(\Gamma),J_{\Gamma})$, where $Q(\Gamma)$ is a finite quiver and $J_{\Gamma}$ is an ideal of $\kappa Q(\Gamma)$ such that 
\begin{itemize}
    \item Let $Q(\Gamma)'=(Q(\Gamma)'_0,Q(\Gamma)'_1)$ be the quiver such that $Q(\Gamma)'_0=I$ and $Q(\Gamma)'_1=\{(i,j)\in I^2\mid i\neq j\}$, where $(i,j)$ represents an arrow with source $i$ and target $j$. 
    Then, $Q(\Gamma)$ is obtained from $Q(\Gamma)'$ by removing all arrows $(i,j)$ such that there exists a path $i=i_1,i_2,...,i_s=j$ with length at least 2 in $Q'(\Gamma)$ that has the same weight as $(i,j)$.
    \item $J_{\Gamma}$ is the two-sided ideal of $\kappa Q(\Gamma)$ generated by all paths $(i_1,...,i_s)$ which fail to obtain the minimal weight among all paths with same head and tail, together with the differences of any two paths with same head and tail obtaining the minimal weight.    
\end{itemize}
\end{defn}

\begin{rem}
See Figure \ref{fig: quiver linked grassmannian} for examples of quivers associated to convex lattice configurations. Intuitively, one can think of the arrow $(i,j)$ in $Q(\Gamma)'$ as indicating the injection $F_{i,j}\colon L_{i}\rightarrow L_{j}$. In this way, Definition \ref{defn:associated quiver} can be seen as removing from $Q(\Gamma)'$ such inclusions that can be realized as a composition of two or more other such inclusions. 
\end{rem}

\begin{lem}\label{lem:well-defined}
$(Q(\Gamma),J_{\Gamma})$ is independent of the choices of representatives for lattice classes in $\Gamma$. Moreover, $J_{\Gamma}$ is an admissible ideal of $\kappa Q(\Gamma)$. 
\end{lem}
\begin{proof}
First, we check that $(Q(\Gamma),J_{\Gamma})$ is independent of the choice of representatives. It suffices to check that changing $L_i$ with $\pi^aL_i$ does not change the definition: this change has the effect of changing the weight of any path starting and not ending (resp. ending and not starting) at $i\in I=Q(\Gamma)_0$ by $-a$ (resp. by $a$). Since the removal of arrows from $Q(\Gamma)'$ is made based on weight comparison between paths with same head and tail, $Q(\Gamma)$ is unaltered under a change of representatives. Same can be said for the definition of $J_{\Gamma}$.   

Next, the weight $w$ along an oriented cycle with source $L_i$ is strictly positive, since $\pi^w L_i\subset L_i$ and the vertices correspond to non-homothetic lattices. It follows that $J_{\Gamma}$ contains all paths containing an oriented cycle. 
On the other hand, since there are only $|I|$ vertices in $Q(\Gamma)$, any path consisting of at least $|I|$ arrows must contain an oriented cycle, hence is in $J_{\Gamma}$. In other words, $R^{|I|}_{Q(\Gamma)}\subset J_{\Gamma}$, where $R_{Q(\Gamma)}$ represents the arrow ideal. It is clear that $J_{\Gamma}\subset R^2_{Q(\Gamma)}$, because by the definition of $Q(\Gamma)$, any arrow from $i$ to $j$ in $Q(\Gamma)$ is the unique path in $Q(\Gamma)$ between these two vertices obtaining the minimal weight (hence not contained in $J_\Gamma$) and there are no loops in $Q(\Gamma)$. Thus, $R^{|I|}_{Q}\subset J_{\Gamma}\subset R^2_Q$. 
\end{proof}

Consider all $f_{i,j}$'s as in Notation \ref{nt:morphisms in a configuration} such that $(i,j)\in Q(\Gamma)_1$. This gives a representation of $Q(\Gamma)$ over $\kappa$ of dimension $\mathbf d$ enjoying special relations. This is the motivation behind the definition of $J_{\Gamma}$ and is summarized in the next two propositions.

\begin{prop}\label{prop:quiver representation and special fiber}
The maps $f_{i,j}$ induce a representation $M_\Gamma$ of $Q(\Gamma)$ of dimension $\mathbf d$. The underlying vector spaces of $M_\Gamma$ are $(\ov L_i)_{i\in Q(\Gamma)_0}$.
Furthermore, the set of closed points of $LG_r(\Gamma)_0$ is identified with the set of subrepresentations of $M_\Gamma$ of dimension $\mathbf r$.
\end{prop}
\begin{proof}
Follows directly from construction.
\end{proof}

When verifying properties of $M_\Gamma$, including that it is a representation of $(Q(\Gamma), J_\Gamma)$, the convexity of $\Gamma$ plays an important role. In particular, one can take the convex hull of any two elements:

\begin{lem}\label{lem:convex hull of two points}
Let $\{[L_0],...,[L_a]\}\subset \Gamma$ be the convex hull of $[L_0]$ and $[L_a]$ such that $[L_i]$ is adjacent to $[L_{i+1}]$. Then
\begin{enumerate}
\item $f_{0,a}=f_{a-1,a}\circ\cdots\circ f_{0,1}$ and $f_{a,0}= f_{1,0}\circ\cdots\circ f_{a,a-1}$, in other words, $n_{0,a}=\sum_{i=0}^{a-1}n_{i,i+1}$ and $n_{a,0}=\sum_{i=0}^{a-1}n_{i+1,i}$;
    \item $\ker f_{i,i+1}=\mathrm{Im}f_{i+1,i}$ and $\ker f_ {i+1,i}=\mathrm{Im}f_{i,i+1}$ for $0\leq i\leq a-1$;
    \item $\ker f_{i,i+1}\cap\ker f_{i,i-1}=0$ for $1\leq i\leq a-1$.
\end{enumerate}
\end{lem}
\begin{proof}
Up to a suitable choice of basis of $V$ and scaling the representatives, we may assume that $L_a=\mathrm{span}\{e_1,...,e_d\}$ and $L_0=\mathrm{span}\{\pi^{a_1}e_1,...,\pi^{a_d}e_d\}$ where $a=a_1\geq \cdots \geq a_d=0$. Moreover,   $L_i=\mathrm{span}\{\pi^{a_{i,j}}e_j\}_{1\leq j\leq d}$ and $a_{i.j}=\max(a_j-i,0)$ for $0\leq i\leq a$. 

It follows that $n_{0,a}=n_{a,0}=a$ and $n_{i,i+1}=n_{i+1,i}=1$ for $0\leq i\leq n-1$, this gives (1). On the other hand, straightforward calculation shows that, for all possible $i$, 
$$f_{i+1,i}(\overline L_{i+1})=\mathrm{span}\{\overline {\pi^{a_{i,j}}e_j}|a_j>  i\} \mathrm{\ and\ }f_{i-1,i}(\ov L_{i-1})=\mathrm{span}\{\overline {\pi^{a_{i,j}}e_j}|a_j< i\}.$$
Hence $$\ker f_{i+1,i}=\mathrm{span}\{\overline {\pi^{a_{i+1,j}}e_j}|a_j\leq   i\} \mathrm{\ and\ }\ker f_{i-1,i}=\mathrm{span}\{\overline {\pi^{a_{i-1,j}}e_j}|a_j\geq  i\}.$$
This gives (2) and (3).
\end{proof}

\begin{prop}\label{prop:ambient representation} Let $M_\Gamma$ be as in Proposition \ref{prop:quiver representation and special fiber}. We have
\begin{enumerate}
\item For any $i,j\in Q(\Gamma)_0:=I$, there exists a path $\ell$ in $Q(\Gamma)$ such that $f_\ell=f_{i,j}$;
    \item for any two paths $\ell_1$, $\ell_2$ in $Q(\Gamma)'$ such that $s(\ell_1)=s(\ell_2)$, $t(\ell_1)=t(\ell_2)$ and $f_{\ell_i}\neq 0$ for both $i$, we have $f_{\ell_1}=f_{\ell_2}$;
    \item for any non-trivial path $\ell$ in $Q(\Gamma)'$ not of minimal weight (e.g. a cycle), we have $f_{\ell}=0$;
     \item if $\ell\in Q(\Gamma)_1$, then $[L_{s(\ell)}]$ is adjacent to $[L_{t(\ell)}]$. 
\end{enumerate}
In particular, $M_{\Gamma}$ is a representation of $(Q(\Gamma),J_{\Gamma})$.
\end{prop}
\begin{proof}
(1) Take a longest path $\ell=(i_1,...,i_s)$ in $Q'(\Gamma)$ with the same source, tail and weight as the arrow $(i,j)$. Then $f_\ell\colon \ov L_i\rightarrow\ov L_j$ is induced by $F_{i_{s-1},i_s}\circ\cdots\circ F_{i_1,i_2}=F_{i,j}$. Hence $f_\ell=f_{i,j}$. Note that the weight of $(i,j)$ is the minimal among all paths from $i$ to $j$. It follows that all arrows in $\ell$ are preserved in $Q(\Gamma)$.

(2) Clearly, $f_{\ell_1}=f_{s(\ell_1),t(\ell_1)}=f_{s(\ell_2),t(\ell_2)}=f_{\ell_2}$.

(3) Again, let $\ell=(i_1,...,i_s)$. Then $F_{i_{s-1},i_s}\circ\cdots\circ F_{i_1,i_2}(L_{i_1})\subset \pi F_{i_1,i_s}(L_{i_1})\subset \pi L_{i_s}$. Thus $f_\ell=0$.

(4) According to Lemma \ref{lem:convex hull of two points} (1), the convex hull of $L_{s(\ell)}$ and $L_{t(\ell)}$ gives a path in $Q'(\Gamma)$ with the same weight as $\ell$. Since $\ell$ has length one and $\ell\in Q(\Gamma)_1$, the path given by the convex hull must also have length one. Therefore, $[L_{s(\ell)}]$ is adjacent to $[L_{t(\ell)}]$.
\end{proof}

\begin{rem}\label{rem:pre-linked}
Proposition \ref{prop:quiver representation and special fiber} and \ref{prop:ambient representation} almost imply that $LG_r(\Gamma)_0$ is a pre-linked Grassmannian over $\kappa$ in the sense of  \cite[Definition A.1.2]{Ohrk}, except potentially in the case where there exist minimal paths $\ell_1$ and $\ell_2$ in $Q(\Gamma)$ with same heads and tails such that $f_{\ell_1}=0$ while $f_{\ell_2}\neq 0$. We will see later (in Remark~\ref{rem:prelinked 2}) that $LG_r(\Gamma)_0$ is a pre-linked Grassmannian at least for certain $\Gamma$'s.
\end{rem}
The importance of the above proposition lies within the fact that the bound quiver algebra of $(Q_{\Gamma},J_{\Gamma})$ is a finite-dimensional $\kappa$-algebra. In particular, its finitely-generated modules are completely decomposable, i.e. they can be decomposed into a direct sum of indecomposable modules in a unique way (\cite[Theorem I.4.10]{assem2006elements}). Moreover, we get a complete list of indecomposable projective modules: let $A_{\Gamma}=\kappa Q(\Gamma)/J_{\Gamma}$ and let $\epsilon_i$ be the idempotent element corresponding to $i\in Q(\Gamma)_0$. We have
\begin{lem}{\cite[Lemma I.5.3(b), Corollary II.2.12]{assem2006elements}}\label{lem:indecomp_proj}
$A_{\Gamma}$ (as a left $A_{\Gamma}$-module) admits a decomposition of the form 
$A_\Gamma\cong \bigoplus_{i\in Q(\Gamma)_0} P_i$, where each $P_i:=A_{\Gamma}\cd \epsilon_i$ is a projective indecomposable $A_{\Gamma}$-module. 
Moreover, each $P_i$ corresponds to the dimension-$\mathbf 1$ representation $(\epsilon_j\cd P_i,f_{\ell})_{j\in Q(\Gamma)_0,\ell\in Q(\Gamma)_1}$ of $(Q(\Gamma),J_{\Gamma})$,
and
every projective $A_{\Gamma}$-module is a direct sum of such $P_i$'s. 
\end{lem}

\begin{proof}
Only the last part requires justification. Using the standard equivalence between modules of bound quiver algebras and quiver representations $F:\textrm{Mod} A_{\Gamma}\to \textrm{Rep}_{\kappa}(Q(\Gamma),J_{\Gamma})$ ( \cite[Theorem III.1.6]{assem2006elements} ), $F(P_i)=(\epsilon_j\cd P_i,f_{\ell})_{j\in Q(\Gamma)_0,\ell\in Q(\Gamma)_1}$, where $f_{\ell}(v)= \ell\cd v$. By definition, $\epsilon_j\cd P_i$ is the vector space spanned by residue classes (mod $J_{\Gamma}$) of paths from $i$ to $j$. By the definition of $J_{\Gamma}$, this vector space is 1-dimensional.  
\end{proof}
Hereafter, we shall always use $\{P_i\mid i\in Q(\Gamma)_0\}$ to denote the complete set of indecomposable projective representations of $(Q(\Gamma),J_{\Gamma})$. Note that for each $j\in Q(\Gamma)_0$, the vector space of $P_i$ on $j$ is the image of the vector space of $P_i$ on $i$.
\begin{prop}\label{prop:ambient representation 2}
$A_{\Gamma}$ satisfies the following properties: 
\begin{enumerate}
\item $\dim_{\kappa} A_{\Gamma}=|I|^2$ and can be presented as $\bigoplus_{i,j\in Q(\Gamma)_0}\kappa\cd\ell_{i,j}$, where $\ell_{i,j}$ are defined as follows: when $i=j$, $\ell_{i,j}=\epsilon_i$; when $i\neq j$, we fix a choice of a path $\ell_{i,j}$ in $Q(\Gamma)$ from $i$ to $j$ such that the induced linear map $\ov L_i\to \ov L_j$ along that path is non-zero. 
\item The algebra structure of $A_{\Gamma}$ is determined by the following rule: $\ell_{i',j}\cd \ell_{i,i'}=0$ if and only if $\ell_{i,i'}=\ell_{k,i'}\cd \ell_{i,k}$, where $k$ is the index in $I$ such that $[L_k]=[F_{i',j}(L_{i'})\cap \pi L_{j}]$ is the point in the convex hull of $\{[L_{i'}],[L_j]\}$ that is adjacent to $[L_{i'}]$. 
\item Let $M_\Gamma$ be as in Proposition \ref{prop:quiver representation and special fiber}. It corresponds to a projective $A_{\Gamma}$-module if and only if elements in $\Gamma$ belong to one apartment in $\mathfrak B_d$.
\end{enumerate}
\end{prop}

\begin{proof}
(1) The dimension statement follows directly from the decomposition in Lemma \ref{lem:indecomp_proj} and the fact $I=Q(\Gamma)_0$. The presentation of $A_{\Gamma}$ follows as one can take $\ell_{i,j}$ to be a basis of $\epsilon_j\cd P_i$. 

(2) Notice that the existence of such a vertex $k$ follows from the convexity of $\Gamma$. By Lemma \ref{lem:convex hull of two points} (1), $\ell_{i',j}=\ell_{k,j}\cdot\ell_{i',k}$, thus the ``if" part is clear. The ``only if" part reduces to showing that $\ell_{k,i'}\cd \ell_{i,k}\neq 0$, provided that
$\ell_{i',j} \cdot \ell_{i,i'} = 0$. Suppose $F_{i,k}=\pi^a$, $F_{k,i'}=\pi^b$ and $F_{i,i'}=\pi^c$. We just need to check $a+b=c$. Further set $F_{i',j}=\pi^t$ and without loss of generality take $L_k=F_{i',j}(L_{i'})\cap \pi L_{j}=\pi^tL_{i'}\cap \pi L_j$.

Since $\ell_{i',j}\cdot \ell_{i,i'}=0$, we have $\pi^{c+t}L_i\subset\pi L_j$. Also, $\pi^{c+t}L_i\subset \pi^tL_{i'}$, hence $\pi^{c+t}L_i\subset L_k$. On the other hand, $\pi^{c-1}L_i\not\subset L_{i'}$, hence $\pi^{c+t-1}L_i\not\subset L_k$ and $F_{i,k}=\pi^{c+t}$. It then further reduces to showing $b=-t$. First of all, $\pi^{-t}L_k\subset L_{i'}$ by construction, hence $b\leq -t$. Secondly, since $\pi^tL_{i'}\subset L_j$, we have $L_k=\pi^tL_{i'}\cap \pi L_j\supset\pi^{t+1}L_{i'}$. Therefore, $L_{i'}\subsetneq \pi^{-t-1}L_k$ and $b>-t-1$. 
Thus $b=-t$.



(3) Assume first that $\Gamma$ is contained in an apartment and choose representatives such that
\[L_i=\langle \pi^{a_{i,1}}e_1,...,\pi^{a_{i,d}}e_d\rangle, a_{i,j}\in \ZZ,\ \fa i\in Q(\Gamma)_0=I.\]
Let $q_i\colon L_i\rightarrow \ov L_i$ be the quotient map, and denote $v_{i,j}=q_i(\pi^{a_{i,j}}e_j)\in \ov L_i$. 
We claim that, 
for every $1\leq j\leq d$, there exists exactly one $i_j$ such that $v_{i_j,j}\in \ov L_{i_j}$ is non-zero and the representation $R_j:=(\kappa\cdot f_{i_j,k}(v_{i_j,j}))_{k\in Q(\Gamma)_0}$ is isomorphic to $P_{i_j}$ (hence also to $(\kappa\cdot v_{k,j})_{k\in Q(\Gamma)_0}$).

Let $i_j\in Q(\Gamma)_0$ be an index such that $f_{k,i_j}(v_{k,j})=0$ for all $k\neq i_j$. Such index must exist, otherwise for every $n\in Q(\Gamma)_0$ one can find $n'\neq n$ such that $f_{n',n}(v_{n',j})=v_{n,j}$. Consequently, one gets an infinite sequence $n_1,n_2,...$ such that 
$f_{n_{k+1},n_k}(v_{n_{k+1},j})=v_{n_k,j}\neq 0.$
This would violate (3) in Proposition \ref{prop:ambient representation}. 
It remains to show that $f_{i_j,k}(v_{i_j,j})\neq 0$ for all $k\neq i_j$. Let $[L_h]$ be the point in the convex hull of $[L_{i_j}]$ and $[L_k]$ that is adjacent to $[L_{i_j}]$. Since $f_{h,i_j}(v_{h,j})=0$, by Lemma \ref{lem:convex hull of two points} (2), $v_{h,j}$ is in the image of $f_{i_j,h}$. Hence we must have $f_{i_j,h}(v_{i_j,j})=v_{h,j}\neq 0$. It follows from Lemma \ref{lem:convex hull of two points} (3) that $f_{i_j,k}(v_{i_j,j})\neq 0$, and hence $f_{i_j,k}(v_{i_j.j})=v_{k,j}$. Thus, $\ov L_k=\s\{f_{i_j,k}(v_{i_j,j})\mid j=1,...,d\}$ for all $k$ and $M_{\Gamma}=\bigoplus_{j=1}^d R_j\cong\bigoplus_{j=1}^d P_{i_j}$.

Conversely, suppose $M_{\Gamma}\cong\bigoplus_{j=1}^d P_{i_j}$ is projective. Fix $e_j\in L_{i_j}$ such that $v_j:=q_{i_j}(e_j)\in\ov L_{i_j}$ generates the direct summand $P_{i_j}$. By Nakayama's Lemma, $\ov L_k=\s\{f_{i_j,k}(v_j)\mid j=1,...,d\}$ implies $L_k=\s\langle F_{i_1,k}(e_1),...,F_{i_d,k}(e_d)\rangle=\s\langle \pi^{n_{i_1,k}}e_1,...,\pi^{n_{i_d,k}}e_d\rangle$ (Notation \ref{nt:morphisms in a configuration}). In other words $\Gamma$ is contained in one apartment. This completes the proof.
\end{proof}
\begin{rem}
We will prove in Section \ref{sec:geometry of linked flag schemes} that if $\Gamma$ is a locally linearly independent configuration (Definition \ref{defn:locally linearly independent}), $M_{\Gamma}$ is projective and hence $\Gamma$ lies within one apartment of $\mathfrak B_d$ (Proposition \ref{prop:indecomp of loc linear indep}).   
\end{rem}

\subsection{Local linear independence}\label{subsec:the local linear independence} In this subsection we introduce a special kind of configuration $\Gamma$ in $\mathfrak B^0_d$, namely the locally linearly independent configurations. It turns out not only are the corresponding quivers $Q(\Gamma)$ relatively simple, but also the subrepresentations of $M_\Gamma$ can be completely classified (see Section \ref{sec:geometry of linked flag schemes}).

\begin{defn}\label{defn:locally linearly independent}
Let $\Gamma=\{[L_i]\}_{i\in I}$ be a convex collection of lattice classes. We say that $\Gamma$ is \textit{locally linearly independent} at $[L_0]\in \Gamma$ if, letting $\{[L_i]\}_{i\in I'\subset I}$ be the set of lattice classes in $\Gamma$ that are adjacent to $[L_0]$, then the spaces $\{f_{i,0}(\overline L_i)=\ker f_{0,i}\}_{i\in I'}$ are linearly independent in $\overline L_0$. We say that $\Gamma$ is \textit{locally linearly independent} if it is so at all points. 
\end{defn}

We say a graph $G$ is a \textit{double tree} if it is obtained from a tree $T$ by adding one edge between each pair of adjacent vertices. We call $T$ the \textit{associated tree} of $G$.

\begin{lem}\label{lem:locally linearly independent underlying graph}
Let $\Gamma=\{[L_i]\}_{i\in I}$ be a locally linearly independent configuration. Then 
\begin{enumerate}
\item for any $i,j$, the arrow $(i,j)$ is in $Q(\Gamma)_1$ if and only if $[L_i]$ is adjacent to $[L_j]$;

\item the underlying (un-directed) graph $G$ of $Q(\Gamma)$ is a double tree;

\item let $T$ be the tree associated to $G$. For any $i,j\in Q(\Gamma)_0$ considered as vertices of $T$, let $i=i_1,...,i_s=j$ be the minimal path in $T$ from $i$ to $j$. Then the convex hull of $\{[L_i],[L_j]\}$ is $\{[L_{i_l}]\}_{1\leq l\leq s}$, and $f_{i,j}=f_{i_{s-1},i_s}\circ\cdots\circ f_{i_1,i_2}$;

\item If $[L]\in \Gamma$ corresponds to a leaf of $T$, then $\Gamma\backslash [L]$ is still a locally linearly independent convex configuration. 
\end{enumerate}
\end{lem}
\begin{proof}
(1) By Proposition \ref{prop:ambient representation} (4), it remains to show the ``if" part. 
Suppose now $(i,j)\not\in Q(\Gamma)_1$. Then, by Proposition \ref{prop:ambient representation} (1), we must have $f_{i,j}=f_{i_{s-1},i_s}\circ\cdots\circ f_{i_1,i_2}$ where $i_1=i,...,i_s=j$ is a path in $Q(\Gamma)_1$ and $s\geq 3$. Since $f_{i,j}\neq 0$, by Proposition \ref{prop:ambient representation} (3), $[L_{i_{s-1}}]\neq [L_i]$. Since $[L_{i_{s-1}}]$ is adjacent to $[L_j]$, and $f_{i,j}(\overline L_i)\subset f_{i_{s-1},j}(\overline L_{i_{s-1}})$, $[L_i]$ is not adjacent to $[L_j]$ by local linear independence at $[L_j]$.

(2) According to part (1), it suffices to show that there is no cycle in $G$ with length $\geq 3$. Suppose $(i_1,...,i_s,i_1)$ is a cycle in $G$ with  $s\geq 3$. By local linear independence and  induction, we have $f_{i_{j-1},i_j}\circ\cdots\circ f_{i_1,i_2}\neq 0$, hence $f_{i_1,i_j}=f_{i_{j-1},i_j}\circ\cdots\circ f_{i_1,i_2}$ for all $1\leq j\leq s$. Indeed, if $f_{i_{j-1},i_j}\circ\cdots\circ f_{i_1,i_2}= 0$, then by the inductive hypothesis and Proposition \ref{prop:ambient representation 2} (2) we have $$f_{i_1,i_{j-1}}=f_{i_{j-2},i_{j-1}}\circ\cdots\circ f_{i_1,i_2}= f_{i_j,i_{j-1}}\circ f_{i_1,i_j}.$$
This contradicts the local linear independence at $[L_{i_{j-1}}]$.
Now setting $j=s$ we have $f_{i_1,i_s}=f_{i_{s-1},i_s}\circ f_{i_1,i_{s-1}}$. This again contradicts with the local linear independence at $[L_s]$. 

(3) According to (1), the convex hull gives a path in $T$ (and also in $Q(\Gamma)$) from $i$ to $j$. Since $T$ is a tree, this path must be the same as $i_1,...,i_s$, which proves the first part. The expression of $f_{i,j}$ follows from Proposition \ref{prop:ambient representation} (1) since there is only one path in $Q(\Gamma)$ from $i$ to $j$.

(4) We only need to check that $\Gamma\backslash [L]$ is convex, which follows directly from (3).
\end{proof}
\begin{rem}\label{rem:prelinked 2}
Suppose $\Gamma$ is locally linearly independent. For $[L_1],[L_2]\in \Gamma$, by Lemma \ref{lem:locally linearly independent underlying graph} there is exactly one minimal path $\ell$ in $Q(\Gamma)$ from $[L_1]$ to $[L_2]$, hence $f_{\ell}=f_{1,2}\neq 0$. By Remark \ref{rem:pre-linked}, $LF_r(\Gamma)_0$ is a pre-linked Grassmannian. Note also that in this case the projective subrepresentations of $M_\Gamma$ corresponds to simple points of $LF_r(\Gamma)_0$ in the sense of Osserman (\cite[Definition A.1.4]{Ohrk}).
\end{rem}

\begin{ex}\label{ex:local linear indep}
We give a few examples of locally linearly independent configurations. See also Figure \ref{fig: quiver linked grassmannian} for their associated quiver.

(1) We say that $\Gamma$ is a \textit{convex chain} if it is the convex hull in $\mathcal B^0_d$ of two lattice classes. A convex chain is a locally linearly independent configuration by Lemma \ref{lem:convex hull of two points}.

(2) We say that $\Gamma=\{[L_i]\}_{i\in I}$ is a \textit{star-shaped configuration} if there is an $i_0\in I$ such that $L_{i_0}=\mathrm{span}\{e_1,...,e_d\}$ and  there exists disjoint subsets $J_i\subset \{1,...,d\}$ such that 
$$L_i=\mathrm{span}\{ \pi^{\epsilon_j}e_j|\epsilon_j=-1\mathrm{\ if}\ j\in 
J_i\mathrm{\ and\ }\epsilon_j=0\mathrm{\ otherwise.}\}\ \forall\ i\in I\backslash \{i_0\}.$$

A star-shaped configuration is locally linearly independent: it is straightforward to verify that $[L_i]$ is only adjacent to $[L_{i_0}]$ and
$f_{i,i_0}(\ov L_i)=\ker f_{i_0.i}=\mathrm{span}(\overline e_j)_{j\in J_i}$
for $i\in I\backslash\{i_0\}$.

(3) One can check that for any tree $T$, there is a locally linearly independent configuration $\Gamma\subset\mathfrak B^0_d$ whose associated tree as in Lemma \ref{lem:locally linearly independent underlying graph} is exactly $T$. Indeed, let $V(T)$ be the set of vertices of $T$ and $d=|V(T)|$. Pick a basis $\{e_v\}_{v\in V(T)}$ of $V$. For $u,v\in V(T)$ denote by $p_{u,v}$ the minimal path in $T$ connecting $u$ and $v$. Fix a root $u_0$ of $T$, and denote by $a_{u,v}$ the number of edges in $p_{u,u_0}\cap p_{u,v}$. Let $L_u\subset V$ be the lattice generated by $\{\pi^{a_{u,v}}e_{v}\}_{v}$. Then the configuration $\Gamma=\{[L_u]\}_{u\in V(T)}$ is convex and locally linearly independent, and its associated tree is naturally identified with $T$. This follows from the fact that for any $u,v$, the convex hull of $[L_u]$ and $[L_v]$ is the set of all $[L_w]$ such that $w$ is a vertex in $p_{u,v}$. We leave the details to the reader.
\end{ex}

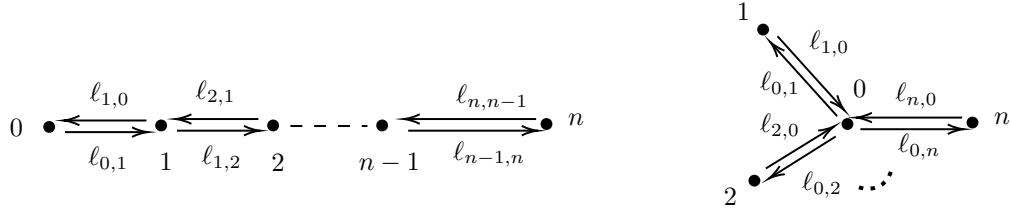
\begin{figure}[H]
\tikzset{every picture/.style={line width=0.75pt}} 

\begin{tikzpicture}[x=0.55pt,y=0.55pt,yscale=-1,xscale=1]
\draw [color={rgb, 255:red, 0; green, 0; blue, 0 }  ,draw opacity=1 ]   (188,119.2) -- (133,119.2) ;
\draw [shift={(131,119.2)}, rotate = 360] [color={rgb, 255:red, 0; green, 0; blue, 0 }  ,draw opacity=1 ][line width=0.75]    (10.93,-3.29) .. controls (6.95,-1.4) and (3.31,-0.3) .. (0,0) .. controls (3.31,0.3) and (6.95,1.4) .. (10.93,3.29)   ;
\draw  [color={rgb, 255:red, 0; green, 0; blue, 0 }  ,draw opacity=1 ][fill={rgb, 255:red, 0; green, 0; blue, 0 }  ,fill opacity=1 ] (124.5,123.61) .. controls (124.5,121.74) and (122.98,120.22) .. (121.11,120.22) .. controls (119.24,120.22) and (117.72,121.74) .. (117.72,123.61) .. controls (117.72,125.48) and (119.24,127) .. (121.11,127) .. controls (122.98,127) and (124.5,125.48) .. (124.5,123.61) -- cycle ;
\draw [color={rgb, 255:red, 0; green, 0; blue, 0 }  ,draw opacity=1 ]   (132,127.2) -- (187,127.2) ;
\draw [shift={(189,127.2)}, rotate = 180] [color={rgb, 255:red, 0; green, 0; blue, 0 }  ,draw opacity=1 ][line width=0.75]    (10.93,-3.29) .. controls (6.95,-1.4) and (3.31,-0.3) .. (0,0) .. controls (3.31,0.3) and (6.95,1.4) .. (10.93,3.29)   ;
\draw [color={rgb, 255:red, 0; green, 0; blue, 0 }  ,draw opacity=1 ]   (265,118.2) -- (210,118.2) ;
\draw [shift={(208,118.2)}, rotate = 360] [color={rgb, 255:red, 0; green, 0; blue, 0 }  ,draw opacity=1 ][line width=0.75]    (10.93,-3.29) .. controls (6.95,-1.4) and (3.31,-0.3) .. (0,0) .. controls (3.31,0.3) and (6.95,1.4) .. (10.93,3.29)   ;
\draw  [color={rgb, 255:red, 0; green, 0; blue, 0 }  ,draw opacity=1 ][fill={rgb, 255:red, 0; green, 0; blue, 0 }  ,fill opacity=1 ] (201.5,122.61) .. controls (201.5,120.74) and (199.98,119.22) .. (198.11,119.22) .. controls (196.24,119.22) and (194.72,120.74) .. (194.72,122.61) .. controls (194.72,124.48) and (196.24,126) .. (198.11,126) .. controls (199.98,126) and (201.5,124.48) .. (201.5,122.61) -- cycle ;
\draw [color={rgb, 255:red, 0; green, 0; blue, 0 }  ,draw opacity=1 ]   (209,126.2) -- (264,126.2) ;
\draw [shift={(266,126.2)}, rotate = 180] [color={rgb, 255:red, 0; green, 0; blue, 0 }  ,draw opacity=1 ][line width=0.75]    (10.93,-3.29) .. controls (6.95,-1.4) and (3.31,-0.3) .. (0,0) .. controls (3.31,0.3) and (6.95,1.4) .. (10.93,3.29)   ;
\draw [color={rgb, 255:red, 0; green, 0; blue, 0 }  ,draw opacity=1 ]   (456,118.2) -- (367,118.2) ;
\draw [shift={(365,118.2)}, rotate = 360] [color={rgb, 255:red, 0; green, 0; blue, 0 }  ,draw opacity=1 ][line width=0.75]    (10.93,-3.29) .. controls (6.95,-1.4) and (3.31,-0.3) .. (0,0) .. controls (3.31,0.3) and (6.95,1.4) .. (10.93,3.29)   ;
\draw  [color={rgb, 255:red, 0; green, 0; blue, 0 }  ,draw opacity=1 ][fill={rgb, 255:red, 0; green, 0; blue, 0 }  ,fill opacity=1 ] (353.5,122.61) .. controls (353.5,120.74) and (351.98,119.22) .. (350.11,119.22) .. controls (348.24,119.22) and (346.72,120.74) .. (346.72,122.61) .. controls (346.72,124.48) and (348.24,126) .. (350.11,126) .. controls (351.98,126) and (353.5,124.48) .. (353.5,122.61) -- cycle ;
\draw [color={rgb, 255:red, 0; green, 0; blue, 0 }  ,draw opacity=1 ]   (368,126.2) -- (455,126.2) ;
\draw [shift={(457,126.2)}, rotate = 180] [color={rgb, 255:red, 0; green, 0; blue, 0 }  ,draw opacity=1 ][line width=0.75]    (10.93,-3.29) .. controls (6.95,-1.4) and (3.31,-0.3) .. (0,0) .. controls (3.31,0.3) and (6.95,1.4) .. (10.93,3.29)   ;
\draw  [color={rgb, 255:red, 0; green, 0; blue, 0 }  ,draw opacity=1 ][fill={rgb, 255:red, 0; green, 0; blue, 0 }  ,fill opacity=1 ] (278.5,122.61) .. controls (278.5,120.74) and (276.98,119.22) .. (275.11,119.22) .. controls (273.24,119.22) and (271.72,120.74) .. (271.72,122.61) .. controls (271.72,124.48) and (273.24,126) .. (275.11,126) .. controls (276.98,126) and (278.5,124.48) .. (278.5,122.61) -- cycle ;
\draw  [color={rgb, 255:red, 0; green, 0; blue, 0 }  ,draw opacity=1 ][fill={rgb, 255:red, 0; green, 0; blue, 0 }  ,fill opacity=1 ] (466.5,121.61) .. controls (466.5,119.74) and (464.98,118.22) .. (463.11,118.22) .. controls (461.24,118.22) and (459.72,119.74) .. (459.72,121.61) .. controls (459.72,123.48) and (461.24,125) .. (463.11,125) .. controls (464.98,125) and (466.5,123.48) .. (466.5,121.61) -- cycle ;
\draw  [dash pattern={on 4.5pt off 4.5pt}]  (287,123) -- (340,123) ;
\draw [color={rgb, 255:red, 0; green, 0; blue, 0 }  ,draw opacity=1 ]   (749,117.2) -- (678,117.2) ;
\draw [shift={(676,117.2)}, rotate = 360] [color={rgb, 255:red, 0; green, 0; blue, 0 }  ,draw opacity=1 ][line width=0.75]    (10.93,-3.29) .. controls (6.95,-1.4) and (3.31,-0.3) .. (0,0) .. controls (3.31,0.3) and (6.95,1.4) .. (10.93,3.29)   ;
\draw [color={rgb, 255:red, 0; green, 0; blue, 0 }  ,draw opacity=1 ]   (679,125.2) -- (748,125.2) ;
\draw [shift={(750,125.2)}, rotate = 180] [color={rgb, 255:red, 0; green, 0; blue, 0 }  ,draw opacity=1 ][line width=0.75]    (10.93,-3.29) .. controls (6.95,-1.4) and (3.31,-0.3) .. (0,0) .. controls (3.31,0.3) and (6.95,1.4) .. (10.93,3.29)   ;
\draw  [color={rgb, 255:red, 0; green, 0; blue, 0 }  ,draw opacity=1 ][fill={rgb, 255:red, 0; green, 0; blue, 0 }  ,fill opacity=1 ] (759.5,120.61) .. controls (759.5,118.74) and (757.98,117.22) .. (756.11,117.22) .. controls (754.24,117.22) and (752.72,118.74) .. (752.72,120.61) .. controls (752.72,122.48) and (754.24,124) .. (756.11,124) .. controls (757.98,124) and (759.5,122.48) .. (759.5,120.61) -- cycle ;
\draw  [color={rgb, 255:red, 0; green, 0; blue, 0 }  ,draw opacity=1 ][fill={rgb, 255:red, 0; green, 0; blue, 0 }  ,fill opacity=1 ] (673.5,121.61) .. controls (673.5,119.74) and (671.98,118.22) .. (670.11,118.22) .. controls (668.24,118.22) and (666.72,119.74) .. (666.72,121.61) .. controls (666.72,123.48) and (668.24,125) .. (670.11,125) .. controls (671.98,125) and (673.5,123.48) .. (673.5,121.61) -- cycle ;
\draw  [color={rgb, 255:red, 0; green, 0; blue, 0 }  ,draw opacity=1 ][fill={rgb, 255:red, 0; green, 0; blue, 0 }  ,fill opacity=1 ] (609.5,160.61) .. controls (609.5,158.74) and (607.98,157.22) .. (606.11,157.22) .. controls (604.24,157.22) and (602.72,158.74) .. (602.72,160.61) .. controls (602.72,162.48) and (604.24,164) .. (606.11,164) .. controls (607.98,164) and (609.5,162.48) .. (609.5,160.61) -- cycle ;
\draw [color={rgb, 255:red, 0; green, 0; blue, 0 }  ,draw opacity=1 ]   (614,152.2) -- (659.3,124.25) ;
\draw [shift={(661,123.2)}, rotate = 508.32] [color={rgb, 255:red, 0; green, 0; blue, 0 }  ,draw opacity=1 ][line width=0.75]    (10.93,-3.29) .. controls (6.95,-1.4) and (3.31,-0.3) .. (0,0) .. controls (3.31,0.3) and (6.95,1.4) .. (10.93,3.29)   ;
\draw [color={rgb, 255:red, 0; green, 0; blue, 0 }  ,draw opacity=1 ]   (662,130.2) -- (617.69,158.13) ;
\draw [shift={(616,159.2)}, rotate = 327.77] [color={rgb, 255:red, 0; green, 0; blue, 0 }  ,draw opacity=1 ][line width=0.75]    (10.93,-3.29) .. controls (6.95,-1.4) and (3.31,-0.3) .. (0,0) .. controls (3.31,0.3) and (6.95,1.4) .. (10.93,3.29)   ;
\draw  [color={rgb, 255:red, 0; green, 0; blue, 0 }  ,draw opacity=1 ][fill={rgb, 255:red, 0; green, 0; blue, 0 }  ,fill opacity=1 ] (615.5,56.61) .. controls (615.5,54.74) and (613.98,53.22) .. (612.11,53.22) .. controls (610.24,53.22) and (608.72,54.74) .. (608.72,56.61) .. controls (608.72,58.48) and (610.24,60) .. (612.11,60) .. controls (613.98,60) and (615.5,58.48) .. (615.5,56.61) -- cycle ;
\draw [color={rgb, 255:red, 0; green, 0; blue, 0 }  ,draw opacity=1 ]   (622,61.2) -- (666.66,110.71) ;
\draw [shift={(668,112.2)}, rotate = 227.95] [color={rgb, 255:red, 0; green, 0; blue, 0 }  ,draw opacity=1 ][line width=0.75]    (10.93,-3.29) .. controls (6.95,-1.4) and (3.31,-0.3) .. (0,0) .. controls (3.31,0.3) and (6.95,1.4) .. (10.93,3.29)   ;
\draw [color={rgb, 255:red, 0; green, 0; blue, 0 }  ,draw opacity=1 ]   (663,116.2) -- (618.34,66.69) ;
\draw [shift={(617,65.2)}, rotate = 407.95] [color={rgb, 255:red, 0; green, 0; blue, 0 }  ,draw opacity=1 ][line width=0.75]    (10.93,-3.29) .. controls (6.95,-1.4) and (3.31,-0.3) .. (0,0) .. controls (3.31,0.3) and (6.95,1.4) .. (10.93,3.29)   ;
\draw [line width=1.5]  [dash pattern={on 1.69pt off 2.76pt}]  (678,165.2) .. controls (684,167.2) and (696,167.2) .. (700,154.2) ;

\draw (92,116.4) node [anchor=north west][inner sep=0.75pt]    {$0$};
\draw (195,139.4) node [anchor=north west][inner sep=0.75pt]    {$1$};
\draw (272,140.4) node [anchor=north west][inner sep=0.75pt]    {$2$};
\draw (332,139.4) node [anchor=north west][inner sep=0.75pt]    {$n-1$};
\draw (476,112.4) node [anchor=north west][inner sep=0.75pt]    {$n$};
\draw (146,135.4) node [anchor=north west][inner sep=0.75pt]    {$\ell _{0,1}$};
\draw (147,91.4) node [anchor=north west][inner sep=0.75pt]    {$\ell _{1,0}$};
\draw (220,89.4) node [anchor=north west][inner sep=0.75pt]    {$\ell _{2,1}$};
\draw (224,134.4) node [anchor=north west][inner sep=0.75pt]    {$\ell _{1,2}$};
\draw (399,92.4) node [anchor=north west][inner sep=0.75pt]    {$\ell _{n,n-1}$};
\draw (397,131.4) node [anchor=north west][inner sep=0.75pt]    {$\ell _{n-1,n}$};
\draw (769,111.4) node [anchor=north west][inner sep=0.75pt]    {$n$};
\draw (701,91.4) node [anchor=north west][inner sep=0.75pt]    {$\ell _{n,0}$};
\draw (703,126.4) node [anchor=north west][inner sep=0.75pt]    {$\ell _{0,n}$};
\draw (672,89.4) node [anchor=north west][inner sep=0.75pt]    {$0$};
\draw (637,152.4) node [anchor=north west][inner sep=0.75pt]    {$\ell _{0,2}$};
\draw (608,112.4) node [anchor=north west][inner sep=0.75pt]    {$\ell _{2,0}$};
\draw (583,163.4) node [anchor=north west][inner sep=0.75pt]    {$2$};
\draw (592,36.4) node [anchor=north west][inner sep=0.75pt]    {$1$};
\draw (642,55.4) node [anchor=north west][inner sep=0.75pt]    {$\ell _{1,0}$};
\draw (608,83.4) node [anchor=north west][inner sep=0.75pt]    {$\ell _{0,1}$};

\end{tikzpicture}
\caption{The quiver associated to a convex chain (left) and a star-shaped configuration (right). The extra relations in the path algebra of the left quiver are given by $\ell_{i,i+1}\cdot \ell_{i+1,i}=0$ for $0\leq i\leq n-1$, while the extra relations of the right quiver are given by $\ell_{i,0}\cdot \ell_{0,i}=0$ for $1\leq i\leq n$.}\label{fig: quiver linked grassmannian}
\end{figure}

\subsection{Examples of linked Grassmannians}
In this part, we show some examples of linked Grassmannians, which motivate the study of the subject. These include Osserman's linked Grassmannians, standard local models of Shimura varieties of PEL-type and moduli of linked linear series. We continue using $R,\pi,K$ as before. 

\subsubsection{Osserman's linked Grassmannian}\label{subsubsec:linked grassmannian}
We recall Osserman's notion of linked Grassmannian, which was first introduced in \cite{Olls} for the construction of a moduli scheme of limit linear series on reducible nodal curves. To have a better comparison with our notion of linked Grassmannian, we also adopt the notion of linked chain from \cite{murray2016linked}.

\begin{defn}
Let $S$ be an integral and Cohen-Macaulay scheme, $E_1,...,E_n$ be vector bundles on $S$, each of rank $d$. Suppose we are given $g_i:E_i\to E_{i+1}$ and $h_i:E_{i+1}\to E_i$ for $1\leq i\leq n-1$. 

(1) (\cite[1.1]{murray2016linked}) Let $s\in \mathscr O_S$, we say that $\mathbf E=(E_\bullet)$ is an $s$-linked chain on $S$ if:
\begin{enumerate}
\item[(i)] $g_i\circ h_i=h_i\circ g_i=s\cd\id$, for all $i$.
\item[(ii)] The closed subscheme of $S$ where $\rk(g_i)+\rk(h_i)<d$ is empty.
\item[(iii)] The subschemes of $S$ where $\rk (g_i)>\rk(g_{i+1}\circ g_i)$ and where $\rk (h_{i+1})>\rk(h_{i}\circ h_{i+1})$ are both empty.  
\end{enumerate}

(2) ({\cite[A.2-A.4]{Olls}}\label{defn:osserman linked grass})
Fix $r<d$. Let $\mathbf E$ be an $s$-linked chain. 
Let $\mathcal{OLG}_r(\mathbf E)$ be the functor associating to each $S$-scheme $T$ the set of subbundles $V_1,...,V_n$ of $E_{1,T},...,E_{n,T}$ of rank $r$ satisfying $g_{i,T}(V_i)\subset V_{i+1}$ and $h_{i,T}(V_{i+1})\subset V_i$ for all $i$. It is represented by a scheme $OLG_r(\mathbf E)$ projective over $S$, which we call an \textit{Osserman's linked Grassmannian} over $S$.
\end{defn}

\begin{prop}\label{prop:linked grass N linked chain}
$\mathrm{(1)}$ An Osserman's linked Grassmannian associated to a $\pi$-linked chain on $\spec(R)$ is a linked Grassmannian associated to a convex chain in $\mathfrak B^0_d$ and vice versa.  $\mathrm{(2)}$ An Osserman's linked Grassmannian associated to a $0$-linked chain on $\spec(\kappa)$ is the special fiber of a linked Grassmannian associated to a convex chain in $\mathfrak B^0_d$ and vice versa. 
\end{prop}
\begin{proof}
(1)
 Given a $\pi$-linked chain $\mathbf E$ on $\spec(R)$ as in Definition \ref{defn:osserman linked grass}, we may assume that none of the $g_i$'s or $h_i$'s is an isomorphism. Then mapping all $E_i$ to $E_1$ gives an identification of $E_i$ with a lattice in $E_1\otimes K$. More  precisely, denote $L_i=h_1\circ\cdots\circ h_{i-1}(E_i)\subset E_1\subset E_1\otimes K$, then condition (i) implies that $\pi L_i\subset  L_{i+1}\subset L_i$. We claim that $\Gamma=\{[L_i]\}_{1\leq i\leq n}$ is the convex hull of $[L_1]$ and $[L_n]$. Indeed, for each $i$, we have $L_i\subset L_1\cap \pi^{-1}L_{i+1}$. If $u\in L_1\cap \pi^{-1}L_{i+1}$ is a vector that does not lie in $L_i$, then $\pi u\in L_{i+1}\subset L_i$. It follows that $h_i(\pi u)\neq 0$ and $h_1\circ\cdots\circ h_{i}(\pi u)=0$ over the closed point of $\spec(R)$, which contradicts condition (iii). Hence $L_i=L_1\cap \pi^{-1}L_{i+1}$, which implies the claim. Hence $OLG_r(\mathbf E)=LG_r(\Gamma)$.

The converse direction is basically covered in \cite[\S 3.2]{hahn2020mustafin}. We include the proof for the sake of completeness. Suppose 
$L_1\subset L_2\subset \cdots\subset L_n$ gives the convex hull of $[L_1]$ and $[L_n]$ in $\mathfrak B^0_d$, hence $[L_i]$ is adjacent to $[L_{i+1}]$. Then the morphisms 
$$F_{i,i+1}\colon L_i\hookrightarrow L_{i+1}\mathrm{\ and\ }F_{i+1,i}\colon L_{i+1}\rightarrow L_i,\mathrm{\ where\ } F_{i+1,i}(z)=\pi z,$$ give a $\pi$-linked chain $\mathbf E$ on $\spec(R)$. Indeed, note that $F_{i,i+1}$ and $F_{i+1,i}$ are
exactly the maps  constructed in Proposition \ref{prop:mustafin degeneration functor representable}.    Condition (i) in Definition \ref{defn:osserman linked grass} (1) is satisfied by construction; conditions (ii) and (iii) are just Lemma \ref{lem:convex hull of two points} (2) and (3) since $F_{i,j}$ is an isomorphism over the generic point of $S$.
Hence $LG_r(\Gamma)=OLG_r(\mathbf E)$ is an Osserman's linked Grassmannian. 

(2) According to part (1), it remains to show that an $OLG_r(\mathbf E)$ of a $0$-linked $\mathbf E$ on $\kappa$ is the special fiber of $LG_r(\Gamma)$ for a convex chain $\Gamma$. Note that the bundles $E_\bullet$ are $\kappa$-vector spaces of dimension $d$. Again, we may assume that none of the $g_i$'s or $h_i$'s is an isomorphism. For $l<i$, set 
$$g_{l,i}:=g_{i-1}\circ g_{i-2}\circ\cdots\circ g_l\mathrm{\ and\ } h_{i,l}:=h_l\circ h_{l+1}\circ\cdots\circ h_{i-1},$$ and $g_{i,i}=h_{i,i}=\mathrm{id}$.

By \cite[Lemma 2.3]{osserman2014linked}, we can find a set of linearly independent vectors $\ov W_j=\{\overline u_{j,1},...,\overline u_{j,k_j}\}\subset E_j$ for each $1\leq j\leq n$ such that $\sum_jk_j=d$ and for each $i$, the set of vectors
$$\Big(\bigcup_{1\leq j\leq i-1}g_{j,i}(\ov W_j)\Big)\cup \Big( \bigcup_{i\leq j\leq n}h_{j,i}(\ov W_j)\Big)$$
generates $E_i$. Now pick lifts $W_j\subset R^d$ of $h_{j,1}(\ov W_j)\subset E_1$, and let $L_i\subset K^d$ be the lattice generated by 
$$\Big(\bigcup_{1\leq  j\leq i-1}\pi^{-j} W_j\Big)\cup \Big( \bigcup_{i\leq j\leq n}\pi^{-i}W_j\Big).$$
It is then easy to verify that $\Gamma:=\{[L_1],...,[L_n]\}$ is the convex hull of $[L_1]$ and $[L_n]$, and the special fiber of $LG_r(\Gamma)$ is isomorphic to $OLG_r(\mathbf E)$.
\end{proof}

\subsubsection{Standard Local models}\label{subsubsec:local model}
(\cite[\S 4.1]{gortz2001flatness}, \cite[Definition 3.27]{rapoport2016period}) 
Let $e_1\hh e_d$ be the standard basis of $K^d$ and $\Gamma$ be the set of the lattices: $$L_i:=\lag\pi^{-1}e_1,...,\pi^{-1}e_i,e_{i+1},...,e_d\rag,\  0\leq i\leq n-1. $$

The standard local model, denoted $M^{\text{loc}}$, of Shimura varieties is the $R$-scheme parametrizing the functor $\mM$ from $(\text{Sch}_B)$ to $(\text{Sets})$ such that for any $B$-scheme $S$, $\mM(S)$ is the set of all isomorphism classes 
\[
\begin{tikzcd}
L_{0, S} \ar[r] & L_{1,S} \ar[r] & \cdots\ar[r]
 & L_{n-1,S} \ar[r,"\pi"] & L_{0,S} \\
\cF_0  \ar[r]\ar[hookrightarrow]{u} & \cF_1  \ar[r]\ar[hookrightarrow]{u} & \cdots\ar[r] & \cF_{n-1}
  \ar[r]\ar[hookrightarrow]{u} & \cF_0\ar[hookrightarrow]{u}
\end{tikzcd}
\]
where $\cF_i$ is a subbundle of rank $r$ of $L_{i,S}$. Notice that $\Gamma=([L_0],...,[L_{n-1}])$ clearly gives a convex collection of lattices as we have $L_0\subset L_1\subset...\subset L_{n-1}\subset \pi^{-1}L_0$. One then checks that $\mathscr{M}$ agrees with the functor $\mathcal{LG}_{r}(\Gamma)$ in Definition \ref{defn:mustafin degeneration functor}.

The geometry of the special fiber $\ov{M}^{\text{loc}}$ of the standard local model can be interpreted from the perspective of quiver representations. In \cite[\S 4.3]{gortz2001flatness}, G\"ortz concluded that $\ov{M}^{\text{loc}}$ has $\binom{d}{r}$ many irreducible components, indexed by the length-$d$ integer sequences $w(\mu)$, where $w$ is any element in $\SS_d$ and $\mu=(\underbrace{1\hh 1}_{r\text{ times}},0\hh 0)$. Moreover, let $\nu=(\nu_1,...,\nu_d)$ be such a vector, and denote $I=\{k_1\hh k_r\}$ to be the index subset such that $v_k=1$ if and only if $k\in I$. The general element in the component $S_\nu$ corresponds to a projective subrepresentation 
of $M_{\Gamma}$ isomorphic to \[P_{k_1+1}\op\cdots\op P_{k_r+1}\] 
where we set $P_{n+1}=P_1$. This representation-theoretic interpretation for irreducible components (that its general points correspond to projective representations) is similar to our conclusion for locally linearly independent configurations (see Section \ref{sec:geometry of linked flag schemes}).  
Moreover, we would also like to point out another connection to our approach, namely, the Kottowitz-Rapoport stratification of $\ov{M}^{\text{loc}}$ can also be interpreted in the context of quiver representation. The main point is that the group-theoretic data for each stratum correspond to ranks of (compositions of) linear maps viewed from the perspective of quiver representation. See \cite[\S 2]{gortz2010supersingular} for further detail.

\subsubsection{Degeneration of linear series}\label{subsubsec:degenerate linear series}
Another important example comes from studying degeneration of moduli spaces of linear series on algebraic curves. 
We only sketch the general idea here as the details will be carried out in Section \ref{sec:connection to lls}. 

We start with a relative curve $X/\spec(R)$ whose special fiber is a reducible nodal curve (Definition \ref{defn:smoothing family}). Then there is a space $\widetilde{\mathcal G}^2$ in which the moduli of linked linear series is cut out. Moreover, $\widetilde{\mathcal G}^2$ is projective over the relative Picard scheme $\mathrm{Pic}(X/B)$ of line bundles with fixed multidegrees on each fiber of $X$. For any section $s\colon B\rightarrow \mathrm{Pic}(X/B)$, the fiber product $\widetilde{\mathcal G}^2\times_{\mathrm{Pic}(X/B)} B$ is a linked Grassmannian (Proposition \ref{prop:moduli of lls and linked flags}).
It turns out that there exists a forgetful map from $\widetilde{\mathcal{G}}^2$ to another $B$-scheme $\mathcal{G}^2$ inside which one can construct the moduli scheme of limit linear series on $X/\spec(R)$. Such a description of the moduli of limit linear series eventually allows us to prove a criterion for smoothing of limit linear series over arbitrary nodal curves (Theorem \ref{thm:smoothing of lls}).

\section{Geometry of Linked Grassmannians in the Locally Linearly Independent Case}\label{sec:geometry of linked flag schemes}
Recall that one interesting feature of linked Grassmannians is that their special fibers become quiver Grassmannians for quivers with extra relations. 
In this section, we study the topological properties of linked Grassmannians in the locally linearly independent case via analysing their points as quiver representations. 


Through out this section, all schemes are assumed to be $\kappa$-schemes. $\Gamma$ will always denote a locally linearly independent lattice configuration in $\mathfrak B^0_d$, $Q(\Gamma)$ the associated quiver as in Definition \ref{defn:associated quiver}, and $T$ the induced tree of $Q(\Gamma)$ as in Lemma \ref{lem:locally linearly independent underlying graph}. As usual, $E(T)$ and $V(T)$ will denote the sets of edges and vertices respectively.
Note that $\Gamma$, as well as $Q(\Gamma)_0$, is identified with $V(T)$ and we will write $\Gamma=\{[L_v]\}_{v\in V(T)}$.  
Let $M_\Gamma=(\ov L_v:=L_v/\pi L_v)_v$ be the representation of $Q(\Gamma)$ induced by $\Gamma$ as in Proposition \ref{prop:quiver representation and special fiber}. Let $\Gr({\underline x},M_\Gamma)$ be the quiver Grassmannian of $M_\Gamma$ with dimension vector ${\underline x}\in\mathbb Z^{V(T)}_{\geq 0}$. Recall from Proposition \ref{prop:quiver representation and special fiber} that $\Gr(\mathbf r,M_\Gamma)$ is isomorphic to the special fiber $LG_r(\Gamma)_0$ of the linked Grassmannian $LG_r(\Gamma)$.

We denote by $\vec E(T)$ the set of directed edges of $T$, hence $\vec E(T)=Q(\Gamma)_1$. For each $e$, let $\vec e$ denote an orientation on $e$ and $\cev e$ the reversed orientation of $\vec e$. As usual, $s(\vec e)=t(\cev e)$ and $t(\vec e)=s(\cev e)$ will denote the source and target of $\vec e$ respectively. 
For each vertex $v$, denote by $\vec E_v$ the set of edges containing $v$ which are oriented outwards from $v$, and $\cev E_v$ the set of oriented edges obtained by reversing the edges in $\vec E_v$. For each edge $\vec e$, let $A_{\vec e}$ be the set of vertices $v$ such that the minimal path from $v$ to $s(\vec e)$ does not pass through $t(\vec e)$. Then $A_{\cev e}=V(T)\backslash A_{\vec e}$. See Figure \ref{fig:graph} for an illustration.

\tikzset{every picture/.style={line width=0.75pt}}
\begin{figure}[ht]	
\begin{tikzpicture}[x=0.45pt,y=0.45pt,yscale=-0.9,xscale=0.9]
\import{./}{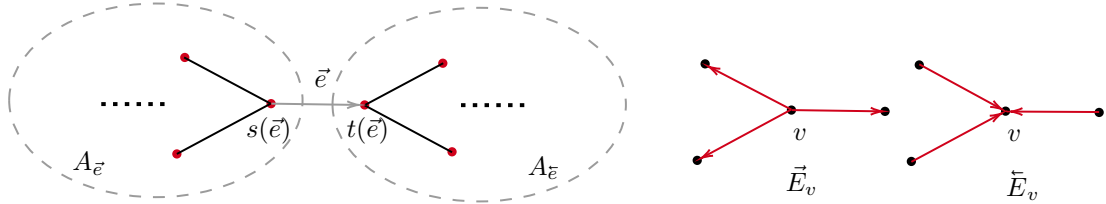}
\end{tikzpicture}		
\caption{The left part is the set of vertices (in red) contained in $A_{\vec e}$ and $A_{\cev e}$. The right part is the set of oriented edges (in red) contained in $\vec E_v$ and $\cev E_v$.}
\label{fig:graph}
\end{figure}

Recall from Notation \ref{nt:morphisms in a configuration} that we have maps $f_{v,v'}\colon \ov L_v\rightarrow \ov L_{v'}$ defining $M_\Gamma$ induced by the inclusion $L_v\xrightarrow{\pi^n} L_{v'}$, where $n$ is the minimal number such that such inclusion exists. According to our convention on quiver representations, the map $f_{s(\vec e),t(\vec e)}$ will also be denoted by $f_{\vec e}$. 

We will consider the following representations of $Q(\Gamma)$: 
for $u\in V(T)$ let $P_u$ be the dimension-$\mathbf 1$ projective representation associated to $u$ as in Lemma \ref{lem:indecomp_proj}. In other words, for $\vec e\in \vec E(T)$, we set $f_{\vec e}=\mathrm{id}$ if $\vec e$ is pointing outwards from $u$ and $f_{\vec e}=0$ otherwise. 
Take a directed edge $\vec \iota$. We construct a representation $R_{\vec \iota}=(U_v)_{v\in V(T)}$ as follows:  we set $U_v=\kappa$ if $v\in A_{\vec \iota}$ and $U_v=0$ if $v\in A_{\cev \iota}$. For $\vec e\in \vec E(T)$, if $s(\vec e)\in A_{\vec \iota}$ and $\vec e$ is pointing outwards from $s(\vec \iota)$, we set $f_{\vec e}=\mathrm{id}$, otherwise $f_{\vec e}=0$. See Figure \ref{fig:local indep indecomp} for an example.

\tikzset{every picture/.style={line width=0.75pt}}
\begin{figure}[ht]	
\begin{tikzpicture}[x=0.5pt,y=0.5pt,yscale=-0.9,xscale=0.9]
\import{./}{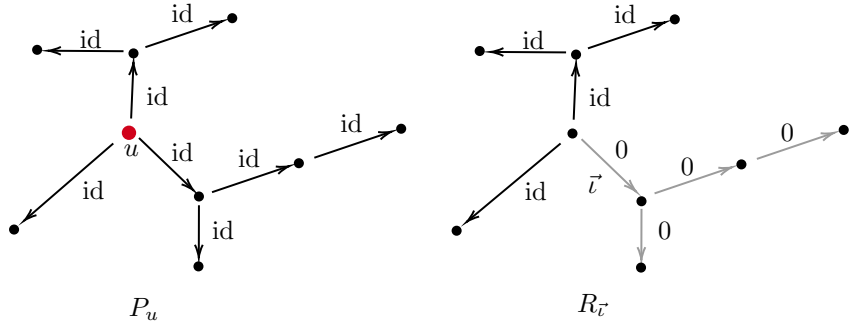}
\end{tikzpicture}		
\caption{Two types of representations of $Q(\Gamma)$, where $\Gamma$ is locally linearly independent. The missing arrows are assumed to be all zero.}
\label{fig:local indep indecomp}
\end{figure}

We would like to comment that, in this section, $R_{\vec \iota}$ will often occur in pairs with $R_{\cev\iota}$. More specifically, if $u$ is a vertex of the edge $\iota$, then $R_{\vec \iota}\oplus R_{\cev\iota}$ can be realized as a specialization of $P_u$. See Proposition~\ref{prop:linked degeneration stratification} and Example~\ref{ex:r=1} for an illustration of the later fact.

Finally, for all $n\in \mathbb Z_{\geq 0}$ recall that we denote by $[n]$ the set $\{1,...,n\}$.

\subsection{Subrepresentations of $M_\Gamma$}
\addtocontents{toc}{\protect\setcounter{tocdepth}{2}}
In this subsection we consider subrepresentations of $M_\Gamma$ (of arbitrary dimension). It turns out that their decompositions are quite simple. This provides a very efficient way via quiver representation for analyzing the geometry of $LG_r(\Gamma)_0=\Gr(\mathbf r,M_\Gamma)$.
\begin{lem}\label{lem:decomp} 
Any subrepresentation of $M_{\Gamma}$ 
decomposes as a direct sum of subrepresentations of dimension $\leq \mathbf 1$. 
\end{lem}
\begin{proof} Let $M=(U_v)_{v\in V(T)}$ be a subrepresentation of $M_\Gamma$ with dimension ${\underline x}=(x_v)_{v\in V(T)}$. 
We prove the theorem by induction on $|\Gamma|$ and $|{\underline x}|=\sum_vx_v$, which we will refer to as the \textit{value} of ${\underline x}$ in this proof. The base case $|{\underline x}|=0$ or $|\Gamma|=1$ is trivial, hence we may assume $|{\underline x}|\geq 1$ and $|\Gamma|\geq 2$. Let $u$ be a leaf of $T$ adjacent to $u'$. Note that $\Gamma\backslash [L_u]$ is still convex and locally linearly independent by Lemma \ref{lem:locally linearly independent underlying graph} (4). If $U_u=0$, then any decomposition of $M$ as a representation of $Q(\Gamma\backslash [L_u])$ extends to a decomposition of $M$ as a representation of $Q(\Gamma)$. Hence by induction on $|\Gamma|$, we are done. We next assume $U_u\neq 0$.

1) If $f_{u,u'}(U_u)\neq 0$, take $a_u\in U_u$ such that $f_{u,u'}(a_u)\neq 0$ and denote $a_v=f_{u,v}(a_u)$ for $v\neq u$. Note that $a_v\neq 0$ by Lemma \ref{lem:locally linearly independent underlying graph} (3) and local linear independence. Take $V_v\subset U_v$ inductively with respect to the distance between $v$ and $u$ for all $v\in V(T)$ such that 

(i) $U_v=V_v\oplus \langle a_v\rangle$; 

(ii) if $\vec \iota\in\cev E_v$ is the (unique) edge such that $u\in A_{\vec \iota}$, then $V_v$ contains $f_{\vec \iota}(V_{s(\vec \iota)})$; and

(iii) for all $\vec e\in \vec E_v\backslash \cev\iota$, $V_{v}$ contains $\ker f_{\vec e}\cap U_v$. 

\tikzset{every picture/.style={line width=0.75pt}}
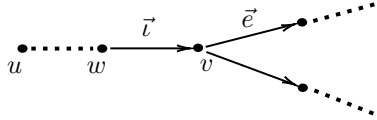
\begin{figure}[ht]	
\begin{tikzpicture}[x=0.5pt,y=0.4pt,yscale=-0.9,xscale=0.9]

\draw  [color={rgb, 255:red, 0; green, 0; blue, 0 }  ,draw opacity=1 ][fill={rgb, 255:red, 0; green, 0; blue, 0 }  ,fill opacity=1 ] (378.89,186.81) .. controls (378.89,184.94) and (377.37,183.42) .. (375.5,183.42) .. controls (373.63,183.42) and (372.11,184.94) .. (372.11,186.81) .. controls (372.11,188.68) and (373.63,190.2) .. (375.5,190.2) .. controls (377.37,190.2) and (378.89,188.68) .. (378.89,186.81) -- cycle ;
\draw  [color={rgb, 255:red, 0; green, 0; blue, 0 }  ,draw opacity=1 ][fill={rgb, 255:red, 0; green, 0; blue, 0 }  ,fill opacity=1 ] (379.5,255.61) .. controls (379.5,253.74) and (377.98,252.22) .. (376.11,252.22) .. controls (374.24,252.22) and (372.72,253.74) .. (372.72,255.61) .. controls (372.72,257.48) and (374.24,259) .. (376.11,259) .. controls (377.98,259) and (379.5,257.48) .. (379.5,255.61) -- cycle ;
\draw  [color={rgb, 255:red, 0; green, 0; blue, 0 }  ,draw opacity=1 ][fill={rgb, 255:red, 0; green, 0; blue, 0 }  ,fill opacity=1 ] (210.89,214.42) .. controls (210.89,212.55) and (209.37,211.04) .. (207.5,211.04) .. controls (205.63,211.04) and (204.11,212.55) .. (204.11,214.42) .. controls (204.11,216.3) and (205.63,217.81) .. (207.5,217.81) .. controls (209.37,217.81) and (210.89,216.3) .. (210.89,214.42) -- cycle ;
\draw  [color={rgb, 255:red, 0; green, 0; blue, 0 }  ,draw opacity=1 ][fill={rgb, 255:red, 0; green, 0; blue, 0 }  ,fill opacity=1 ] (291.5,213.61) .. controls (291.5,211.74) and (289.98,210.22) .. (288.11,210.22) .. controls (286.24,210.22) and (284.72,211.74) .. (284.72,213.61) .. controls (284.72,215.48) and (286.24,217) .. (288.11,217) .. controls (289.98,217) and (291.5,215.48) .. (291.5,213.61) -- cycle ;
\draw [color={rgb, 255:red, 0; green, 0; blue, 0 }  ,draw opacity=1 ]   (294.5,212.61) -- (367.21,189.42) ;
\draw [shift={(369.11,188.81)}, rotate = 522.31] [color={rgb, 255:red, 0; green, 0; blue, 0 }  ,draw opacity=1 ][line width=0.75]    (10.93,-3.29) .. controls (6.95,-1.4) and (3.31,-0.3) .. (0,0) .. controls (3.31,0.3) and (6.95,1.4) .. (10.93,3.29)   ;
\draw [color={rgb, 255:red, 0; green, 0; blue, 0 }  ,draw opacity=1 ]   (214.3,214.2) -- (279.3,214.2) ;
\draw [shift={(281.3,214.2)}, rotate = 180] [color={rgb, 255:red, 0; green, 0; blue, 0 }  ,draw opacity=1 ][line width=0.75]    (10.93,-3.29) .. controls (6.95,-1.4) and (3.31,-0.3) .. (0,0) .. controls (3.31,0.3) and (6.95,1.4) .. (10.93,3.29)   ;
\draw [color={rgb, 255:red, 0; green, 0; blue, 0 }  ,draw opacity=1 ]   (295.11,218) -- (367.91,251.77) ;
\draw [shift={(369.72,252.61)}, rotate = 204.89] [color={rgb, 255:red, 0; green, 0; blue, 0 }  ,draw opacity=1 ][line width=0.75]    (10.93,-3.29) .. controls (6.95,-1.4) and (3.31,-0.3) .. (0,0) .. controls (3.31,0.3) and (6.95,1.4) .. (10.93,3.29)   ;
\draw  [color={rgb, 255:red, 0; green, 0; blue, 0 }  ,draw opacity=1 ][fill={rgb, 255:red, 0; green, 0; blue, 0 }  ,fill opacity=1 ] (143.89,214.42) .. controls (143.89,212.55) and (142.37,211.04) .. (140.5,211.04) .. controls (138.63,211.04) and (137.11,212.55) .. (137.11,214.42) .. controls (137.11,216.3) and (138.63,217.81) .. (140.5,217.81) .. controls (142.37,217.81) and (143.89,216.3) .. (143.89,214.42) -- cycle ;
\draw [line width=1.5]  [dash pattern={on 1.69pt off 2.76pt}]  (147.89,214.42) -- (199.89,214.42) ;
\draw [line width=1.5]  [dash pattern={on 1.69pt off 2.76pt}]  (384.89,185.42) -- (438.3,167.2) ;
\draw [line width=1.5]  [dash pattern={on 1.69pt off 2.76pt}]  (385.89,258.42) -- (438.3,284.2) ;

\draw (287,223.4) node [anchor=north west][inner sep=0.75pt]    {$v$};
\draw (192,225.4) node [anchor=north west][inner sep=0.75pt]    {$w$};
\draw (125,225.4) node [anchor=north west][inner sep=0.75pt]    {$u$};
\draw (238,178.4) node [anchor=north west][inner sep=0.75pt]    {$\vec{\iota }$};
\draw (322,168.4) node [anchor=north west][inner sep=0.75pt]    {$\vec{e}$};

\end{tikzpicture}		
\caption{Case 1 of the proof of Lemma \ref{lem:decomp}}
\label{fig:proof of decomposition}
\end{figure}

We verify the existence of $V_v$, see Figure \ref{fig:proof of decomposition}. Since $u\in A_{\vec\iota}$, we can apply the inductive hypothesis for $w:=s(\vec\iota)$: $V_w$ contains $\ker f_{\vec \iota}\cap U_w$ and does not contain $a_w$. It follows that 
$a_v=f_{\vec \iota}(a_w)\not\in f_{\vec\iota}(V_w)$. By the local linear independence of $\Gamma$ at $[L_v]$, the spaces $f_{\vec \iota}(U_w)$ and  $\{\ker f_{\vec e}\cap U_v\}_{\vec e\in\vec E_v\backslash \cev \iota}$ are linearly independent; this verifies the existence of $V_v$. Since $V_v$ contains $f_{\vec\iota}(V_w)$ as well as $\ker f_{\vec e}\cap U_v$, which contains $f_{\cev e}(U_{s(\cev e)})$ for $\vec e\in \vec E_v\backslash\cev\iota$, we get a subrepresentation $(V_v)_{v\in V(T)}$ of $M_\Gamma$ of dimension ${\underline x}-\mathbf 1$. Therefore, $M=(\langle a_v\rangle )_v\oplus (V_v)_v$ and we apply the inductive hypothesis on the representation $(V_v)_v$, whose dimension vector has value $|{\underline x}|-|V(T)|$.

2) Suppose $f_{u,u'}(U_u)=0$. If $f_{u',u}(U_{u'})\neq U_u$, take $a_u\in U_u\backslash f_{u',u}(U_{u'})$ and $V_u\subset U_u$ such that $U_u=V_u\oplus \langle a_u \rangle$ and $f_{u',u}(U_{u'})\subset V_u$. Set $V_v=U_v$ and $a_v=0$ for $v\neq u$, then $M=(\langle a_v\rangle )_v\oplus (V_v)_v$ and we use the inductive hypothesis on $(V_v)_v$, whose dimension vector has value $|{\underline x}|-1$.

3) Suppose $f_{u,u'}(U_u)=0$ and $f_{u',u}(U_{u'})= U_u$. If $x_u<x_{u'}$, then we can take $0\neq a_{u'}\in U_{u'}$ such that $f_{u',u}(a_{u'})=0$, hence $f_{u',v}(a_{u'})\neq 0$ for all $v\neq u$ by the local linear independence of $\Gamma$ at $[L_{u'}]$. Take $V_{u'}\subset U_{u'}$ such that $U_{u'}=\langle a_{u'}\rangle \oplus V_{u'}$ and $V_{u'}$ contains $\ker f_{\vec e}\cap U_{u'}$ for all $\vec e\in \vec E_{u'}$ such that $t(\vec e)\neq u$.  Argue as in 1) for all branches of $T$ at $u'$ but the one containing $u$, we get subspaces $V_v\subset U_v$ for $v\neq u$ such that $(U_v)_{v\neq u}=(\langle a_v\rangle )_{v\neq u}\oplus (V_v)_{v\neq u}$ as representations of $Q(\Gamma\backslash [L_u])$. Set $a_u=0$ and $V_u=U_u$, then $M=(\langle a_v\rangle )_v\oplus (V_v)_v$ and we apply the inductive hypothesis on $(V_v)_v$, whose dimension vector has value $|{\underline x}|-|V(T)|+1$.

4) Suppose $f_{u,u'}(U_u)=0$, $f_{u',u}(U_{u'})= U_u$ and $x_u=x_{u'}$. Then $f_{u',u}$ is an isomorphism. By induction on $|\Gamma|$, we have $(U_v)_{v\neq u}=\bigoplus_k(U_{v,k})_{v\neq u}$ as representations of $Q(\Gamma\backslash [L_u])$, where each representation  $(U_{v,k})_{v\neq u}$ has dimension $\leq \mathbf 1$. Set $U_{u,k}=f_{u',u}(U_{u',k})$; then $M=(U_v)_{v\in V(T)}=\bigoplus_k(U_{v,k})_v$ where each direct summand $(U_{v,k})_v$ has dimension $\leq \mathbf 1$.
\end{proof}

Knowing that each subrepresentation of $M_\Gamma$ has a decomposition as in Lemma \ref{lem:decomp}, it is now very easy to classify all the indecomposable subrepresentations of $M_\Gamma$, and also calculate the decompositions of the decomposable ones.

\begin{prop}\label{prop:indecomp of loc linear indep}
Suppose $\Gamma$ is a locally linearly independent configuration.
\begin{enumerate}
\item All indecomposable subrepresentations of $M_\Gamma$ are of the form $P_v$ or $R_{\vec e}$.

\item If $M=(U_v)_v$ is a subrepresentation of $M_\Gamma$, then we have the decomposition
$$M\cong \big( \bigoplus_{v\in V(T)} P^{r_v}_v\big)\op \big( \bigoplus_{\vec e\in \vec E(T)} R_{\vec e}^{r_{\vec e}}\big),$$ where
\[r_v=\dim U_v-\sum_{\vec e\in \vec E_v}\dim (\ker f_{\vec e}\cap U_v) \mathrm{\ and\ }r_{\vec e}=\dim (\ker f_{\vec e}\cap U_{s(\vec e)})-\dim f_{\cev e}(U_{t(\vec e)}).\]
In particular, $M$ is projective if and only if $r_{\vec e}=0$ for all $\vec e$, and 
$$M_\Gamma\cong \bigoplus_{v\in V(T)} P^{d_v}_v, \mathrm{\ where\ }d_v=d-\sum_{\vec e\in \vec E_v}\dim \ker f_{\vec e}\mathrm{\ and\ } \sum_{v\in V(T)}d_v=d.$$

\item If $M$ has dimension $\mathbf r$ then $r_{\vec e}=r_{\cev e}$ for all $\vec e$, and $M$ can be decomposed as a direct sum of subrepresentations of dimension $\mathbf 1$.

\item $\Gamma$ is contained in an apartment.
\end{enumerate}
\end{prop}
\medskip

\begin{proof}
(1) It is easy to verify that both $P_v$ and $R_{\vec e}$ are indecomposable. Now let $R=(U_v)_v$ be an indecomposable subrepresentation of $M_\Gamma$ of dimension $\leq \mathbf 1$. We prove by induction on $|\Gamma|$ that $R$ must either be of the form $P_v$ or $R_{\vec e}$. Let $u$ be a leaf of  $T$ and $\Gamma'=\Gamma\backslash \{[L_u]\}$, and $u'$ the unique vertex adjacent to $u$. If $f_{u,u'}(U_u)\neq 0$ then $R=P_u$ and we are done. We next assume $f_{u,u'}(U_u)=0$. Since $R$ is indecomposable, $f_{u',u}(U_{u'})=U_u$.

By Lemma \ref{lem:locally linearly independent underlying graph} (4) we know that $\Gamma'$ is still locally linearly independent.  If the restriction $R'$ of $R$ on $Q(\Gamma')$ can be decomposed as $R_1\oplus R_2$, where $R_1=(U^1_v)_{v\neq u}$ and $R_2=(U^2_v)_{v\neq u}$, such that $\dim U^1_{u'}\geq \dim U^2_{u'}$, then $(U_u,R_1)\oplus (0,R_2)$ gives a decomposition of $R$. Hence we may assume that $R'$ is also indecomposable. By induction, $R'$ is either $P'_{v}$ or $R'_{\vec e}$ as subrepresentations of $M_{\Gamma'}$. Hence so is $R$ as a subrepresentation of $M_\Gamma$.

(2) By part (1), $M$ can be represented as a direct sum of $P_v$'s and $R_{\vec e}$'s. Let us denote by $W_v$ (resp. $W_{\vec e}$) the $v$-th (resp. $s(\vec e)$-th) component of $P_v^{r_v}$ (resp. $R_{\vec e}^{r_{\vec e}}$). It is then easy to verify that $$\ker f_{\vec e}\cap U_v=W_{\vec e}\oplus f_{\cev e}(U_{t(\vec e)})$$ for all $\vec e\in \vec E_v$,
which gives $r_{\vec e}$. It also follows that 
$$U_v=W_v\oplus \big( \bigoplus_{\vec e \in \vec E_v} (\ker f_{\vec e}\cap U_v)\big),$$ which gives $r_v$. The fact that $r_{\vec e}=0$ for all $\vec e$ is equivalent to the projectivity of $M$ follows from Lemma~\ref{lem:indecomp_proj}; and the
decomposition of $M_\Gamma$ follows from Lemma \ref{lem:convex hull of two points} (2).

(3) We have $r_{\vec e}=r-\dim f_{\vec e}(U_{s(\vec e )})-\dim f_{\cev e}(U_{s(\cev e)})=r_{\cev e}$. By (2), we have $$M\cong\Big(\bigoplus_{v\in V(T)}P^{r_v}_v\Big)\op\Big(\bigoplus_{e\in E(T)}(R_{\vec e}\oplus R_{\cev e})^{r_{\vec e}}\Big),$$ which is a direct sum of subrepresentations of dimension $\mathbf 1$.

(4) Follows from Proposition \ref{prop:ambient representation 2} (3) and part (2).
\end{proof}

We end this subsection with a lemma that will be used to analyse the smoothing property of limit linear series (see Theorem \ref{thm:smoothing of lls}).

\begin{lem}\label{lem:lifting limit linear series}
Let $\Gamma$ be locally linearly independent. Given a non-empty subset $I\subset V(T)$ and $r$-dimensional vector spaces $V_v\subset \ov L_v$ for $v\in I$. Suppose for all $u\in \Gamma$, the vector space 
$$W_u:=\{x\in \ov L_u|f_{u,v}(x)\in V_v\mathrm{\ for\ all\ }v\in I\}$$ has dimension at least $r$. Then there is an $\mathbf r$-dimensional subrepresentation $M=(U_v)_v$ of $M_\Gamma$ such that $U_v=V_v$ for all $v\in I$.
\end{lem}
\begin{proof}
We use a similar proof as in \cite[Proposition A.6]{osserman2019limit}. Note that $W=(W_v)_v$ is a subrepresentation of $M_\Gamma$ of dimension ${\underline x}\geq \mathbf r$ and $W_v=V_v$ for all $v\in I$. We proceed by induction on $|{\underline x}|$. The base case ${\underline x}=\mathbf r$ is trivial, so we assume $|{\underline x}|>|\mathbf r|$. Then there is a pair of adjacent vertices $u_1$ and $u_2$ such that $\dim W_{u_1}=r$ and $\dim W_{u_2}>r$. Denote $\widetilde W_{u_2}=\bigoplus\limits_{\vec e\in \cev E_{u_2}}f_{\vec e}(W_{s(\vec e)})$. This is the vector space generated by the images from all $W_u$ for $u\neq u_2$. We claim that $\dim\widetilde W_{u_2}\leq r$. 

Indeed, denote by $\vec \iota$ the directed edge from $u_1$ to $u_2$. By local linear independence at $[L_{u_2}]$ we have an injection
$$\widetilde W'_{u_2}:=\bigoplus_{\vec e\in \cev E_{u_2}\backslash\vec \iota}f_{\vec e}(W_{s(\vec e)})\xhookrightarrow{f_{\cev\iota}} \ker f_{\vec\iota}\cap W_{u_1}\subset W_{u_1}.$$
It follows that 
$$\dim\widetilde W_{u_2}= \dim\widetilde W'_{u_2}+\dim f_{\vec \iota}(W_{u_1})\leq \dim (\ker f_{\vec\iota}\cap W_{u_1})+ \dim f_{\vec\iota}(W_{u_1}) =r.$$

Now we can replace $W_{u_2}$ with any $r$-dimensional subspace that contains $\widetilde W_{u_2}$ while keeping the other $W_i$'s. This gives a subrepresentation $W'=(W'_i)$ of $M_\Gamma$ with dimension no less than $\mathbf r$ and strictly less than $\mathbf r'$. Since $u_2\not\in I$, we still have $W'_{v}=V_v$ for all $v\in I$. Hence by induction we are done.
\end{proof}

\subsection{The stratification of the quiver Grassmannians of $M_\Gamma$.}\label{subsec:stratification} Given a dimension vector $\underline x=(x_v)_v\in\mathbb Z^{V(T)}_{\geq 0}$. Let $M\in \Gr(\underline x,M_\Gamma)$ be a subrepresentation. We denote by $\mathcal S_M$ the set of all dimension-$\underline x$ subrepresentations of $M_\Gamma$ that are isomorphic to $M$ and $\mathcal S^c_M$ its closure. 
This induces a stratification $(\mathcal S_M)_{[M]}$ of $\Gr(\underline x,M_\Gamma)$, where $M$ runs through all isomorphic classes of dimension-$\underline x$ subrepresentations of $M_\Gamma$. We associate a preorder ``$\prec$" on $\Gr(\underline x,M_\Gamma)$ where $M\prec M'$ if $M\in \mathcal S^c_{M'}$.
On the other hand, we define a map $\Phi_{\underline x}\colon \Gr({\underline x},M_\Gamma)\rightarrow\mathbb Z^{\vec E(T)}_{\geq 0}$ such that 
$$M:=(U_v)_{v\in V(T)}\mapsto (\dim f_{\vec e}(U_{s(\vec e)}))_{\vec e\in \vec E(T)}.$$ 
We also associate a partial order on $Z^{\vec E(T)}_{\geq 0}$ where $(d_{\vec e})_{\vec e}\leq (d'_{\vec e})_{\vec e}$ if $d_{\vec e}\leq d'_{\vec e}$ for all $\vec e\in \vec E(T)$. Then it follows from construction that $\Phi_{\underline x}$ is order-preserving: if $M\prec M'$ then $\Phi_{\underline x}(M)\leq \Phi_{\underline x}(M')$. We will see in Proposition~\ref{prop:linked degeneration stratification} (2) that the converse is also true. Moreover, following directly from Proposition \ref{prop:indecomp of loc linear indep}, we have

\begin{lem}\label{lem:stratification map}
The fiber of $\Phi_{\underline x}$ at $(d_{\vec e})_{\vec e}\in\mathbb Z^{\vec E(T)}_{\geq 0}$ is, if non-empty, $\mathcal S_M$, where
\[M\cong\Big(\bigoplus_{v\in V(T)} P^{r_v}_{v}\Big)\op \Big(\bigoplus_{\vec e\in\vec E(T)} R^{r_{\vec e}}_{\vec e}\Big)\mathrm{\ with\ }r_v=x_v-\sum_{\vec e\in \vec E_v}(x_v-d_{\vec e})\mathrm{\ and\ }r_{\vec e}=x_{s(\vec e)}-d_{\vec e}-d_{\cev e}.
\]
\end{lem}

We next show that the stratification $\Gr({\underline x},M_\Gamma)=\cup_{[M]}\mathcal S_M$ is well-behaved, i.e., this is a stratification by locally closed irreducible subsets.  
In principle, at least for characteristic zero, the conclusion should follow from a standard argument of stratifications of quiver Grassmannians (for acyclic quivers in characteristic zero see e.g.  \cite[\S 2]{cerulli2012quiver}). However, to avoid unnecessary reference checking, we include the proof.


\begin{prop}\label{prop:quiver stratification}
$\mathcal{S}_M$ is an irreducible locally closed subset of $\Gr({\underline x},M_\Gamma)$ of dimension $$\dim \h(M,M_\Gamma)-\dim\End(M).$$ 
\end{prop}

\begin{proof}
Since the fibers of $\Phi_{\underline x}$ are locally closed, so is $\mathcal S_M$. For the irreducibility, note that, if $M\cong M'\cong P_u$, then $M$ (resp. $M'$) is generated by a vector $x\in \overline L_u$ (resp. $x'\in \overline L_u$) such that $0\neq f_{u,v}(x)\in \ov L_v$ (resp. $0\neq f_{u,v}(x')\in \ov L_v$) for all $v\in V(T)$. Hence for a general choice of $t\in\kappa$, the vector $tx+(1-t)x'\in \overline L_u$ has non-zero image in $\overline L_v$ for all $v$. This gives rise to a subrepresentation $M_t$ of $M_\Gamma$ isomorphic to $P_u$, and hence a rational map  $\mathbb A^1_\kappa\dashrightarrow \mathcal S_{P_u}$ where $t\mapsto M_t$. Here $\mathcal S_{P_u}$ is considered as a stratum of $\Gr(\mathbf 1,M_\Gamma)$. 
The image of this map contains $M$ and $M'$, hence they are contained in the same irreducible component of $\mathcal S_{P_u}$. Therefore, $\mathcal S_{P_u}$ is irreducible. Similarly, $\mathcal S_{R_{\vec e}}$ is irreducible as a stratum of $\Gr(\dim R_{\vec e},M_\Gamma)$. In general, for $M\cong \bigoplus_vP_v^{r_v}\op\bigoplus_{\vec e}R^{r_{\vec e}}_{\vec e}$ we have a rational dominant map
$$\prod_{v\in V(T)}\mathcal S_{P_v}^{r_v}\times\prod_{\vec e\in\vec E(T)}\mathcal S_{R_{\vec e}}^{r_{\vec e}}\dashrightarrow \mathcal S_M$$
given by taking the direct sum. Since the source is irreducible, so is $\mathcal S_M$.

We now compute $\dim \mathcal S_M$ following the proof of \cite[Lemma 2.4]{cerulli2012quiver}. Let $X$ denote the quasi-affine subvariety of $$\Big(\prod_{\vec e\in \vec E(T)}\h(\kappa^{x_{s(\vec e)}},\kappa^{x_{t(\vec e)}})\Big)\times\Big(\prod_{v\in V(T)}\h(\kappa^{x_v},\ov L_v)\Big) $$
consisting of points $((g_{\vec e})_{\vec e},(F_v)_v)\in Y$ such that $F_{t(\vec e)}\circ g_{\vec e}=f_{\vec e} \circ F_{s(\vec e)}$ for all $\vec e \in \vec E(T)$ and that the $F_v$'s are all injective. Note that $X$ parametrizes all ``embeddings" of ${\underline x}$-dimensional subrepresentations of $M_\Gamma$. We have
\[\begin{tikzcd}
&X\ar[dl,"p_1"']\ar[dr,"p_2"]&\\
Y:=\displaystyle\prod_{\vec e\in \vec E(T)}\h(\kappa^{x_{s(\vec e)}},\kappa^{x_{t(\vec e)}})&&\Gr({\underline x}, M_\Gamma)
\end{tikzcd}\]
where $p_1$ is the forgetful map and $p_2$ sends $((g_{\vec e})_{\vec e},(F_v)_v)\in X$ to $((F_v(\kappa^{x_v})_{v\in V(T)})$.

Moreover, denote $\GL_{{\underline x}}=\prod_{v\in V(T)}\GL_{x_v}(\kappa)$, there exists a $\GL_{{\underline x}}$-action on $Y$ 
and a free $\GL_{{\underline x}}$-action on $X$ turning $p_1$ into a $\GL_{{\underline x}}$-equivariant morphism: let $(\phi_v)_v$ be an element in $\GL_{{\underline x}}$, it sends $(g_{\vec e})_{\vec e}\in Y$ to $(\phi_{t(\vec e)}\circ g_{\vec e}\circ\phi^{-1}_{s(\vec e)})_{\vec e}$, and $((g_{\vec e})_{\vec e},(F_v)_v)\in X$ to $((\phi_{t(\vec e)}\circ g_{\vec e}\circ\phi^{-1}_{s(\vec e)})_{\vec e},(F_v\circ \phi^{-1}_v)_v)$ respectively. We have: (i) 
the $\GL_{{\underline x}}$-orbits in $Y$ precisely correspond to isomorphism classes of representations of $Q(\Gamma)$ of dimension ${\underline x}$, and the orbit $\mathcal O_M$ in $Y$ has dimension $\dim\GL_{{\underline x}}-\dim\aut(M)=\dim\GL_{{\underline x}}-\dim\End(M)$; (ii) the fibers of $p_2$ are the (free) $\Gr_{{\underline x}}$-orbits of $X$, and $p^{-1}_2(\mathcal S_M)=p^{-1}_1(\mathcal O_M)$; and (iii) the fiber of $p_1$ over any point of $\mathcal O_M$ is the set of injections $M\hookrightarrow M_\Gamma$. Thus, one can conclude that $\mathcal{S}_M$ has dimension 
\[\dim\GL_{{\underline x}}-\dim\End(M)+\dim\h(M,M_\Gamma)-\dim\GL_{{\underline x}}=\dim\h(M,M_\Gamma)-\dim\End(M).\]
\end{proof}

By Proposition \ref{prop:indecomp of loc linear indep} and Proposition \ref{prop:quiver stratification}, the computation of $\dim \mathcal S_M$ reduces to the computation of  the dimensions of the Hom spaces between the $P_v$'s and $R_{\vec e}$'s.

\begin{lem}\label{lem:hom_dim}
We have: 
\begin{enumerate}
    \item $\dim\h(P_v,P_{v'})=1$. 
    \item $\dim \h(R_{\vec e},P_v)=\begin{cases}0\text{ if }v \in A_{\vec e}\\
    1\text{ otherwise. }
    \end{cases}$
     \item $\dim \h(P_{v},R_{\vec e})=\begin{cases}0\text{ if }v\in A_{\cev e}\\
    1\text{ otherwise. }
    \end{cases}$
    \item $\dim \h(R_{\vec e_1},R_{\vec e_2})=\begin{cases}0\text{ if }A_{\vec e_1}\subset A_{\vec e_2}.\\
    1\text{ otherwise. }
    \end{cases}$
\end{enumerate}
\end{lem}
\begin{proof}
We only prove part (1), part (2,3,4) is similar and not used in the rest of the paper, so we leave the details to the reader. Denote $P_w=(\langle a_v\rangle)_{v\in V(T)}$, where $0\neq a_v\in \kappa$. Up to scaling, we assume $a_v=f_{w,v}(a_w)$ for all $v$. Hence for any morphism $F$ from $P_w$ to $P_{w'}:=(\langle a'_v\rangle)_v$ we must have $F(a_v)=F(f(w,v)(a_w))=f_{w,v}(F(a_w))$. In other words, $F$ is determined by $F(a_w)$. On the other hand, each choice of $F(a_w)\in\langle a'_w\rangle$ gives rise to a morphism 
$F$ from $P_w$ to $P_{w'}$. Hence we have $\dim \h(P_w,P_{w'})=1$.
\end{proof}

We now give an alternate description of the preorder on $\Gr(\mathbf r,M_\Gamma)$, namely, it is induced by $\Phi_{\underline x}$ and the order on $\mathbb Z^{\vec E(T)}_{\geq 0}$.

\begin{prop}\label{prop:linked degeneration stratification}
Let $M,M'\in \Gr({\underline x},M_\Gamma)$ be two subrepresentations. Write $M:=(U_v)_v\cong \big( \bigoplus_v P^{r_v}_v\big)\op \big( \bigoplus_{\vec e} R_{\vec e}^{r_{\vec e}}\big)$ as in Proposition \ref{prop:indecomp of loc linear indep}.

\begin{enumerate}
    \item If there is an $\vec \iota\in\vec E(T)$ such that $r_{\vec \iota}>0$ and $r_{\cev \iota}>0$,  
then there is a representation $N\in \Gr({\underline x},M_\Gamma)$ such that $\mathcal S_M\subset \mathcal S^c_N$, and 
$$N\cong \Big( \bigoplus_{v\in V(T)} P^{r_v}_v\Big)\op \Big( \bigoplus_{\vec e\in \vec E(T)\backslash \{\vec \iota,\ \cev \iota\}} R_{\vec e}^{r_{\vec e}}\Big)\oplus R_{\vec \iota}^{r_{\vec \iota}-1}\oplus R_{\cev \iota}^{r_{\cev \iota}-1}\oplus P_{s(\vec \iota)}.$$

\item $M\prec M'$ if and only if $\Phi_{\underline x}(M)\leq \Phi_{\underline x}(M')$.
\end{enumerate}

\end{prop} 
\begin{proof}
(1) We may write 
$$M'=(U'_v)_v:= \big( \bigoplus_{v\in V(T)} P^{r_v}_v\big)\op \big( \bigoplus_{\vec e\in \vec E(T)\backslash \{\vec \iota,\ \cev \iota\}} R_{\vec e}^{r_{\vec e}}\big)\oplus R_{\vec \iota}^{r_{\vec \iota}-1}\oplus R_{\cev \iota}^{r_{\cev \iota}-1}.$$
Let $u=s(\vec \iota)$ and $u'=t(\vec \iota)$. Then $M= M'\oplus R_{\vec\iota}\op R_{\cev \iota}$ and we can pick vectors $a\in U_{u}$ and $a'\in U_{u'}$ that generate $R_{\vec \iota}$ and $R_{\cev \iota}$, respectively. It follows that $f_{\cev \iota}(a')=0$ and $f_{\vec \iota}(a)=0$, hence we can find $b\in \ov L_u$ such that $f_{\vec \iota}(b)=a'$ by Lemma \ref{lem:convex hull of two points} (2). Now for a general $t\in \kappa$, we have $f_{\vec \iota}(a+tb)=ta'\neq 0$, and, for $\vec e\in \vec E_u\backslash \vec \iota$, we have $f_{\vec e}(a+tb)=f_{\vec e}(a)+tf_{\vec e}(b)\neq 0$ since $f_{\vec e}(a)\neq 0$. Hence, by local linear independence, $a+tb$ generates a subrepresentation $(U_{t,v})_v$ of $M_\Gamma$ which is isomorphic to $P_u$. Moreover, for $w\in A_{\vec \iota}$, we know that $U_{t,w}$ is independent with $U'_w$ since $U_{0,w}=\kappa\cdot f_{u,w}(a)$ is so; for $w\in A_{\cev \iota}$, we have $U_{t,w}=\kappa\cdot  f_{u',w}(a')$ is also linearly independent with $U'_w$. As a result, let $N_t:=M'\oplus (U_{t,v})_v$, then $N_t\cong N$ and $N_t\xrightarrow{t\rightarrow 0} M$, hence $\mathcal S_M\subset \mathcal S^c_N$.

(2) Since $\Phi_{\underline x}$ is order preserving, it remains to show that $\Phi_{\underline x}(M)\leq \Phi_{\underline x}(M')$ implies $M\prec M'$. We may assume $\Phi_{\underline x}(M)\neq\Phi_{\underline x}(M')$, otherwise $\mathcal S_M=\mathcal S_{M'}$ by Lemma \ref{lem:stratification map} and we are done. Pick a tuple $D\in\mathbb Z^{\vec E(T)}_{\geq 0}$ such that      $\Phi_{\underline x}(M)\leq D\leq \Phi_{\underline x}(M')$, and $|D-\Phi_{\underline x}(M)|=1$.  
We may assume that  
\[D-\Phi(M)=(0\hh 0,1_{\vec \iota},0\hh 0).\]  
Then again, by Lemma \ref{lem:stratification map}, we have $\Phi^{-1}_{\underline x}(D)=\mathcal S_N$, where $N$ is the same as  part (1). Note that the existence of $N$, in other words the positivity of $r_{\vec\iota}$ and $r_{\cev\iota}$ in $M$, is actually ensured by the fact $\Phi_{\underline x}(M')\geq D$. Hence $M\prec N$ by part (1). Now replace $M$ with $N$ and proceed inductively, we have $M\prec N\prec\cdots\prec  M'$.
\end{proof}

\subsection{The geometry of linked Grassmannians.} In this subsection we investigate the geometry of linked Grassmannians through the tools developed in the preceding two subsections. We start from computing all possible strata of the special fiber of $LG_r(\Gamma)$, namely $\Gr(\mathbf r,M_\Gamma)$. This amounts to computing the image of $\Phi:=\Phi_\mathbf r$.

\begin{thm}\label{thm:linked degeneration stratification}
Let $\Gamma$ be a locally linearly independent configuration in $\mathfrak B^0_d$. Suppose $M_\Gamma=\oplus P^{d_v}_v$ where $\sum d_v=d$. Given a tuple $D:=(d_{\vec e})_{\vec E(T)}\in\mathbb Z^{|\vec E(T)|}_{\geq 0}$,
\begin{enumerate}
\item
$D$ is contained in the image of $\Phi$ if and only if
\begin{equation}\label{eq:strata 0}
\begin{cases}
0\leq d_{\vec e}\leq r-d_{\cev e}\leq \sum_{v\in A_{\vec e}}d_v\ \mathrm{for\ all\ }\vec e;\\
r-\sum_{\vec e\in \vec E_v}(r-d_{\vec e})\geq 0 \mathrm{\ for\ all\ } v\in V(T).
\end{cases}
\end{equation}

\item
The irreducible components of $\Gr(\mathbf r,M_{\Gamma})$ are of the form $\mathcal S^c_N$, where $N\cong \oplus P^{r_v}_v$ runs over all isomorphic classes of projective subrepresentations of $M_\Gamma$, which are classified by the conditions $\sum r_v=r$ and
\begin{equation}\label{eq:strata 1}\Phi(N)=\Big(\sum_{v\in A_{\vec e}}r_v\Big)_{\vec e\in \vec E(T)}\leq \Big(\sum_{v\in A_{\vec e}}d_v\Big)_{\vec e\in\vec E(T)}.\end{equation}
As a result, $\mathcal S^c_N$ is the set of subrepresentations $M=(U_v)_v$ such that $\dim f_{\vec e }(U_{s(\vec e)})\leq  \sum_{v\in A_{\vec e}}r_v$ for all $\vec e$, and $\Gr(\mathbf r,M_\Gamma)$ has pure dimension $r(d-r)$.
\end{enumerate}
\end{thm}

To sum up, the stratification $(\mathcal S_M)_{[M]}$ of $\Gr(\mathbf r,M_\Gamma)$ is naturally induced by the tuples $D$ satisfying (\ref{eq:strata 0}). Moreover, the strata contained in an irreducible component $\mathcal S^c_N$ correspond to all $D$ such that, in addition to (\ref{eq:strata 0}), $D\leq \Phi(N)$.

\begin{proof}[Proof of Theorem \ref{thm:linked degeneration stratification}.]
(1) \textit{Step 1. Let $M=(U_v)_v$ be an $\mathbf r$-dimension subrepresentation of $M_\Gamma$.} We have $$0\leq \dim f_{\vec e}(U_{s(\vec e)})\leq \dim (\ker f_{\cev e}\cap {U_{s(\cev e)}})= r-\dim f_{\cev e}(U_{s(\cev e)})\leq \dim \ker f_{\cev e}=\sum_{v\in A_{\vec e}}d_v.$$ This proves the first inequality in (\ref{eq:strata 0}). On the other hand, since $\Gamma$ is locally linearly independent,  $$r\geq \dim\Big(\sum_{\vec e\in \vec E_v}(\ker f_{\vec e}\cap U_v)\Big)=\sum_{\vec e\in \vec E_v}\dim(\ker f_{\vec e}\cap U_v)=\sum_{\vec e\in \vec E_v}(r-\dim f_{\vec e}(U_v)).$$ This gives the second inequality in (\ref{eq:strata 0}).

\textit{Step 2: Suppose we have a tuple $D$ satisfying (\ref{eq:strata 0}).} 
Recall that the projective direct summand $P^{d_v}_v$ of $M_\Gamma$ is defined by a $d_v$-dimensional subspace of $\ov L_v$ whose image in $\ov L_{v'}$ under $f_{v,v'}$ still has dimension $d_v$. Pick a basis $\{\zeta^i_v\}_{i\in [d_v]}$ of this subspace of $\ov L_v$. 
Then for each $u\in V(T)$, the vectors $\{f_{v,u}(\zeta^i_v)\}_{v\in V(T),i\in [d_v]}$ form a basis of $\ov{L}_u$.  Let us call $\zeta^i_v$ the \textit{$i$-th global basis vector on $v$}.

To show the realizability of $D$, by Lemma \ref{lem:stratification map}, we should look for subrepresentations $M:=(U_v)_v$ isomorphic to $(\bigoplus_v P^{s_v}_v)\op (\bigoplus_{\vec e} R^{s_{\vec e}}_{\vec e})$, where 
\[s_v=r-\sum_{\vec e\in \vec E_v}(r-d_{\vec e})\geq 0\mathrm{\ and\ }s_{\vec e}=r-d_{\vec e}-d_{\cev e}=s_{\cev e}\geq 0.
\]
For simplicity, we set $s_e=s_{\vec e}=s_{\cev e}$. Then, looking at the dimension of any $U_u$, we have $\sum_{v\in V(T)}s_v+\sum_{e\in E(T)}s_e=r.$


For each $u\in V(T)$, consider the item $P_u^{s_u}\op(\bigoplus_{\vec e\in \vec E_u}R^{s_e}_{\vec e})$; we see that for each $\vec e\in \vec E_u$, there naively should be $s_e$ basis vectors of $U_u$ that lie in $\ker f_{\vec e}$ (which generate the term $R^{s_e}_{\vec e}$); in other words, these basis vectors ``come from the global basis vectors on $A_{\cev e}$". Similarly, there should be another $s_u$ basis vectors of $U_u$ coming from all global basis vectors, and these vectors should generate the term $P^{r_u}_u$. To make this precise, $U_u$ must contain a subspace $W_u$ generated by, for each $\vec e\in \cev E_u$ (here we are switching to the directed edges towards $u$), $s_e$ vectors ``from $A_{\vec e}$":
$$\xi^j_{\vec e}:=\sum_{v\in A_{\vec e}}\sum_{i\in [d_v]} a^{i,j}_{v,\vec e}f_{v,u}(\zeta^i_v),\mathrm{\ where\ }  j\in [s_e],$$
and $s_u$ vectors ``from all global basis vectors": $$\xi^j_u:=\sum_{v\in V(T)}\sum_{i\in [d_v]} a^{i,j}_{v,u}f_{v,u}(\zeta^i_v), \mathrm{\ where\ } j\in [s_u].$$ Here we consider all $a^\bullet_\bullet$'s as coefficients in $\kappa$ of the basis vectors of $\ov L_u$ which are not determined for now. We now let $M=(U_v)_v$ be the minimal subrepresentation such that $W_u\subset U_u$ for all $u$. We claim that for a general choice of coefficients $a^\bullet_\bullet$, $M$ will have dimension $\mathbf r$ and satisfies that $\dim f_{\vec e}(U_{s(\vec e)})=d_{\vec e}$ for all $\vec e$.

\textit{Step. 2.1. We first show that $M$ has dimension $\mathbf r$.} 

\textit{Step. 2.1.1. A simple example.} 
We start this step with illustrating the idea by an example. Let $T$ be the tree with vertices labeled by 1, 2, 3, and 4 as in the left part of Figure \ref{fig:graphproof}.  Let us use the ordered pair $(i,j)$ to denote the oriented edge of $T$ with source $i$ and target $j$. Consider the case $d=4$ and $r=2$. Assume $d_i=1$ for all $i\in [4]$. Let $D$ be the tuple such that $d_{(1,2)}=1$, $d_{(1,3)}=d_{(1,4)}=2$ and $d_{(2,1)}=d_{(3,1)}=d_{(4,1)}=0$. Straightforward calculation shows that $s_1=1$ and $s_2=s_3=s_4=0$, and $s_{(1,2)}=s_{(2,1)}=1$ and all other $s_{(i,j)}$'s vanish. Hence we are looking for a subrepresentation that is isomorphic to $P_1\op R_{(1,2)}\op R_{(2,1)}$. 

By assumption, there is exactly one global basis vector $\zeta_i$ on each vertex $i$. We have $W_3=0$ and $W_4=0$. Moreover, $W_1$ is generated by $\xi_1=\sum_{i=1}^4a_{i,1}f_{i,1}(\zeta_i)$ and $\xi_{(2,1)}=a_{2,(2,1)}f_{2,1}(\zeta_2)$, and $W_2$ is generated by $\xi_{(1,2)}=\sum_{i\neq 2}a_{i,(1,2)}f_{i,2}(\zeta_i)$.

For each $i$, by construction, $U_i$ is generated by $f_{1,i}(W_1)$ and $f_{2,i}(W_2)$. Accordingly, $U_1$ is generated by $\xi_1,\xi_{(2,1)}$ and $f_{2,1}(\xi_{(1,2)})=0$. Since $\xi_1$ and $\xi_{(2,1)}$ are not proportional, for a general choice of coefficients $a_{\bullet}$, $U_1$ has dimension $r=2$. Similarly, $U_2$ is generated by $\xi_{(1,2)}$, $f_{1,2}(\xi_1)=\sum_{i\neq 2}a_{i,1}f_{i,2}(\zeta_i)$ and $f_{1,2}(\xi_{(2,1)})=0$, hence it has dimension $2$. For $j=3,4$, we have $U_j$ generated by $f_{1,j}(\xi_1)=\sum_{i\neq j}a_{i,j}\zeta_i$ and  $f_{1,j}(\xi_{(2,1)})=a_{2,(2,1)}f_{2,j}(\zeta_2)$, and $f_{2,j}(\xi_{(1,2)})=0$, hence also has dimension $2$.

\textit{Step 2.1.2. The proof.} 
Fix $w\in V(T)$; then $U_w$ is generated by $(f_{u,w}(W_u))_{u\in V(\Gamma)}$. Let $\vec I_w\subset \vec E(T)$ be the the set of all directed edges pointing towards $w$. For each $u\neq w$, let $\vec e_u\in \vec E_u$ be the unique edge that lies in $\vec I_w$. We denote $A_u=A_{\vec e_u}$ and $A_w=V(T)$ for convenience. See the right part of Figure \ref{fig:graphproof}.

\tikzset{every picture/.style={line width=0.75pt}}
\begin{figure}[ht]	
\begin{tikzpicture}[x=0.5pt,y=0.5pt,yscale=-0.9,xscale=0.9]
\import{./}{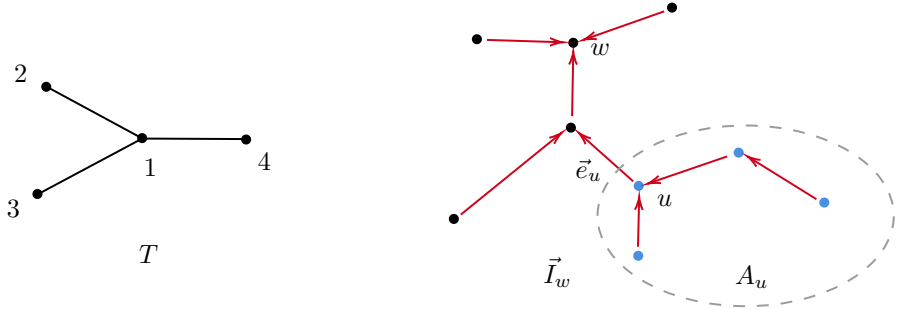}
\end{tikzpicture}		
\caption{The left part is the tree in Step 2.1.1. The right part illustrates the oriented edges (in red) contained in $\vec I_w$ and vertices (in blue) contained in $A_u$.}
\label{fig:graphproof}
\end{figure}

Note that $f_{u,w}\circ f_{v,u}=f_{v,w}$ if $v\in A_u$ and $0$ otherwise,
and if $\vec e\not\in\vec I_w$, then $f_{t(\vec e),w}(\xi^j_{\vec e})=0$ for all $j\in[s_e]$. 
It follows that $U_w$ is actually generated by the following \textit{candidate generators} 
$$\Big\{f_{u,w}(\xi^j_u)=\sum_{v\in A_u}\sum_{i\in [d_v]} a^{i,j}_{v,u}f_{v,w}(\zeta^i_v)\Big\}_{u\in V(T),\ j\in[s_u]}, \mathrm{\ and\ }$$
$$\Big\{f_{t(\vec e),w}(\xi^j_{\vec e})=\sum_{v\in A_{\vec e}}\sum_{i\in [d_v]} a^{i,j}_{v,\vec e}f_{v,w}(\zeta^i_v)\Big\}_{\vec e\in \vec I_w,\ j\in [s_e]}.$$
Since the number of candidate generators above equals $\sum_{u\in V(T)}s_u+\sum_{e\in E(T)}s_e=r$, it suffices to show that these vectors are linearly independent.


Consider the $r\times d$ matrix $\mathfrak C$ of coefficients whose rows are labeled by the set of candidate generators, namely, the set of tuples 
$$\{(u,j)\}_{u\in V(T),\ j\in[s_u]}\mathrm{\ and\ }\{(\vec e,j)\}_{\vec e\in \vec I_w,j\in [s_e]},$$
and columns labeled by the set of global basis vectors, namely the set of tuples $\{(v,i)\}_{v\in V(T),\ i\in [d_v]}$. Let $*$ represent either $u$ or $\vec e$. Then the entry of $\mathfrak C$ on the $(*,j)$-th row and $(v,i)$-th column is $a^{i,j}_{v,*}$ if $v\in A_*$, and 0 otherwise. It is enough to show that $\mathfrak C$ has a non-trivial $r\times r$ minor considered as a polynomial in $a^\bullet_\bullet$.

To see this, we associate each row $(*,j)$ a distinct global basis vector on $A_*$. This is possible inductively: suppose we have picked distinct global basis vectors for all $(u,j)$ where $u$ is in a subset $S\subset V(T)$ and $j\in [s_u]$, and all $(\vec e,j)$  such that $\vec e\in \vec I_w\cap \cev E_u=\cev E_u\backslash \cev e_u$ for some $u\in S$ and $j\in [s_e]$. We can further assume that $w\not\in S$ and the maximal subgraph $T_S$ of $T$ with vertices in $V(T)\backslash S$ is connected. If $T_S=\{w\}$ the inductive step is trivial since there are $d>r$ global basis vectors  on $A_w=V(T)$ to choose. In the following we assume $T_S\neq \{w\}$ and continue the process for a leaf $z\neq w$ of $T_S$ and all edges in $\cev E_z\backslash \cev e_z$ (see Figure \ref{fig:graphproof2}). 

\tikzset{every picture/.style={line width=0.75pt}}
\begin{figure}[ht]	
\begin{tikzpicture}[x=0.5pt,y=0.5pt,yscale=-0.9,xscale=0.9]
\import{./}{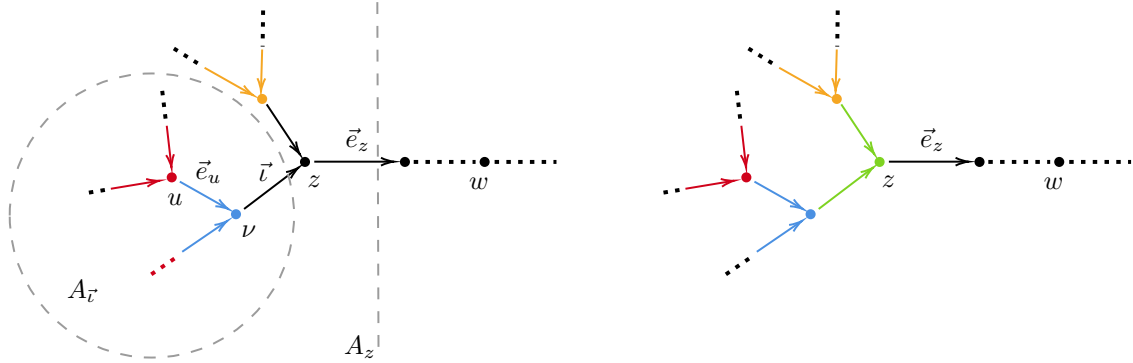}
\end{tikzpicture}		
\caption{The set of vertices and directed edges $*$ (the ones that are not black) such that the distinct global basis vectors are assigned for $(*,j)$ for all $j\in [s_*]$. The left part is before the inductive step and the right part is after the inductive step. Note that the set of edges with a same color (again, not black) is the set $\cev E_v\backslash \cev e_v$ for some vertex $v$.}
\label{fig:graphproof2}
\end{figure}

Given $\vec \iota\in \cev E_z\backslash \cev e_z$, denote $\nu=s(\vec \iota)$. By applying (\ref{eq:strata 0}) for $\vec e=\vec \iota$, we found that the number of global basis vectors on $A_{\vec \iota}$ that are not picked yet is  
$$\sum_{u\in A_{\vec \iota}}d_u-\sum_{u\in A_{\vec \iota}}(s_u+\sum_{\vec e\in \cev E_ u\backslash \cev e_ u}s_{e})=\sum_{u\in A_{\vec \iota}}d_u-\Big(s_\nu+\sum_{\vec e\in \vec E_\nu\backslash \vec\iota}(r-d_{\vec e})\Big)\geq r-d_{\cev \iota}-d_{\vec\iota} = s_{\iota}.$$
Hence we are able to pick distinct global basis vectors for  $(\vec \iota,j)$, where $j\in[s_\iota]$. Suppose now the global basis vectors are picked for all $\vec\iota\in\cev E_z\backslash \cev e_z$. 
Applying (\ref{eq:strata 0}) again for $\vec e=\vec e_z$ and $v=z$, the number of global basis vectors on $A_z$ that are not picked yet is $$\sum_{u\in A_z}d_u-\Big(\sum_{u\in A_z\backslash z}s_u+\sum_{u\in A_z}\sum_{\vec e\in \cev E_u\backslash \cev e_u}s_e\Big)=\sum_{u\in A_z}d_u-\sum_{\vec e\in\vec E_z\backslash \vec e_z}(r-d_{\vec e})\geq r-d_{\cev e_z}-\sum_{\vec e\in\vec E_z\backslash \vec e_z}(r-d_{\vec e}) \geq s_z.$$
Hence we are able to pick the distinct global basis vector for  $(z,j)$, where $j\in[s_z]$. This completes the induction.

Recall that we let $*$ represent either $u\in V(T)$ or $\vec e\in \vec I_w$. Suppose the distinct global basis vector associated to $(*,j)$ is the $i_{*,j}$-th global basis vector on $v_{*,j}$. Let $\mathfrak M$ be the $r\times r$ sub-matrix of $\mathfrak C$ whose columns are labeled by all $(v_{*,j},i_{*,j})$s. Then the determinant of $\mathfrak M$ contains the monomial term 
$$\prod_{u\in V(T),\ j\in [s_u]}a^{i_{u,j},j}_{v_{u,j},u} \prod_{\vec e\in \vec I_w, \  j\in [s_e]}a^{i_{\vec e,j},j}_{v_{\vec e,j},\vec e}.$$  
In particular, we have $\det\mathfrak M\neq 0$, hence $U_w$ has dimension $r$.

\textit{Step 2.2. It remains to show that $\dim f_{\vec \tau}(U_w)=d_{\vec \tau}$ for all $\vec \tau\in \vec E_w$.} Among all candidate generators of $U_w$, the ones with non-trivial image in $U_{t(\vec \tau)}$ are all possible $f_{u,w}(\xi^j_u)$s and $f_{u,w}(\xi^j_{\vec e})$s such that $u\in A_{\vec \tau}$ and $\vec e\in \cev E_u\backslash \cev e_u$ . The number of such candidate generators is
$$\sum_{u\in A_{\vec \tau}}(s_u+\sum_{\vec e\in \cev E_u\backslash \cev e_u}s_{e})=s_w+\sum_{\vec e\in \vec E_w\backslash \vec \tau}(r-d_{\vec e})=d_{\vec \tau}.$$
Moreover, the image of these candidate generators in $U_{t(\vec \tau)}$ gives rise to $d_{\vec \tau}$ candidate generators of $U_{t(\vec \tau)}$. By Step 2.1.2, these candidate generators in $U_{t(\vec \tau)}$ are linearly independent, hence $\dim f_{\vec \tau}(U_w)=d_{\vec \tau}$ and we are done.


(2) The irreducible components of $\Gr(\mathbf r,M_\Gamma)$ are of the form $\mathcal S^c_M$, where $M$ is maximal with respect to the preorder. Suppose $\mathcal S^c_N$ is an irreducible component of $\Gr(\mathbf r,M_\Gamma)$. If there is a direct summand $R_{\vec e}$ in the decomposition of $N$, then by Proposition \ref{prop:indecomp of loc linear indep} (3), there must also be an $R_{\cev e}$. Hence $\mathcal S_N$ is contained in the closure of another stratum in $\Gr(\mathbf r,M_\Gamma)$ by Proposition \ref{prop:linked degeneration stratification}, which provides a contradiction. 
Therefore $N$ must be projective. On the other hand, if $N$ is not maximal, then by the proof of Proposition \ref{prop:linked degeneration stratification} (2), $N$ must contain an $R_{\vec e}$ in its decomposition, hence can not be projective. Therefore, the irreducible components of $\Gr(\mathbf r,M_\Gamma)$ are parametrized by all isomorphic classes of projective subrepresentations of $M_\Gamma$.

We now write $N:=(V_v)_v\cong \oplus P^{r_v}_v$. Then (\ref{eq:strata 1}) is equivalent to (\ref{eq:strata 0}) since $d_{\vec e}:=\dim f_{\vec e}(V_{s(\vec e)})=\sum_{v\in A_{\vec e}}r_v$. 
Moreover, By Lemma \ref{lem:hom_dim} and the dimension formula Proposition \ref{prop:quiver stratification}, we have 
$$\dim\mathcal S_N=\dim\mathrm{Hom}(\oplus P^{r_v}_v,\oplus P^{d_v}_v)-\dim \mathrm{End}(\oplus P^{r_v}_v)=rd-r^2=r(d-r).$$
Thus $\Gr(\mathbf r,M_\Gamma)$ has pure dimension $r(d-r)$.
\end{proof}
\begin{rem}
(1) The idea in Theorem \ref{thm:linked degeneration stratification} can also be used to describe the geometry of $\Gr({\underline x},M_\Gamma)$ for general ${\underline x}$. Although we may not have $r_{\vec e}=r_{\cev e}$ as in Proposition \ref{prop:indecomp of loc linear indep} (3) for decompositions of dimension-${\underline x}$ subrepresentations of $M_\Gamma$ (or, as in the proof of Theorem \ref{thm:linked degeneration stratification}, $s_{\vec e}=s_{\cev e}$), one can still compute all possible strata and hence describe its irreducible components.

(2) We can also describe the intersection of irreducible components of $\Gr(\mathbf r,M_\Gamma)$ (codimensions, strata in the intersection, etc.), as well as count the number of irreducible components with the information in Theorem \ref{thm:linked degeneration stratification}, which will extend the results for $r=1$ in \cite{cartwright2011mustafin}. See also Example \ref{ex:r=1} below. Moreover, as we will see later in Theorem \ref{thm:flatness},  $\Gr(\mathbf r,M_\Gamma)$ is isomorphic to the special fiber of the Mustafin degeneration $M_r(\Gamma)$. Hence it is also possible to classify all primary/secondary components of the special fiber of $M_r(\Gamma)$ (\cite[Definition 3.2]{habich2014mustafin}). 

We leave all the details to the interested readers. 
\end{rem}
\begin{ex}\label{ex:olink_irr}
Let $\Gamma=\{[L_1],[L_2]\}$ be a two-point configuration. The irreducible components of $\Gr(\mathbf r,M_\Gamma)=LG_r(\Gamma)_0$ correspond to non-negative numbers $r_1,r_2$ such that $r_1+r_2=r$ and $r_i\leq d_i$ for $i=1,2$, where $d_1=\mathrm{rank}(f_{1,2})$ and $d_2=\mathrm{rank}(f_{2,1})$ are positive intergers such that $d_1+d_2=d$. This is exactly the description carried out in \cite[Example A.17]{Olls}; recall from Proposition \ref{prop:linked grass N linked chain} (2) that an Osserman's linked Grassmannian over $\kappa$ is isomorphic to $LG_r(\Gamma)_0$ for a convex chain $\Gamma$. Note that we completely described the points in each component of $\Gr(\mathbf r,M_\Gamma)$ while  \cite[Example A.17]{Olls} only identified the irreducible components. Moreover, our conclusion of Theorem \ref{thm:linked degeneration stratification} (2) for convex chains completely answers Question A.19 of \textit{loc.cit.}
\end{ex}

\begin{ex}\label{ex:r=1}
Note that by Proposition \ref{prop:indecomp of loc linear indep} (2), $d_v>0$ if $v$ is a leaf of $T$. Suppose $r=1$, then the set of irreducible components of $\Gr(\mathbf r,M_\Gamma)$ is identified with $V(T)$: for each $v\in V(T)$, $Z_v:=\mathcal S^c_{P_v}$ gives an irreducible component and vice versa. Moreover, for each $v$, the strata contained in $Z_v$ are exactly all $\mathcal S_{R_e}$ where $e$ is an edge containing $v$ and $R_e=R_{\vec e}\op R_{\cev e}$. As a result, $Z_v\cap Z_{v'}$ is non-empty if and only if $v$ is adjacent to $v'$, in which case the intersection is $\mathcal S_{R_e}$ where $e$ is the edge connecting $v$ and $v'$. This agrees with the results in \cite[\S 2]{cartwright2011mustafin}.
\end{ex}

As a consequence of Theorem \ref{thm:linked degeneration stratification}, we now prove the  theorem of the global geometry of a linked  Grassmanian associated to a locally linearly independent configuration $\Gamma$. Note that the case when $\Gamma$ is a convex chain is proved in \cite{helm2008flatness} via a local computation.



\begin{thm}\label{thm:flatness}
Let $\Gamma$ be a locally linearly independent configuration. Then $LG_r(\Gamma)$ is irreducible and flat over $R$. Moreover, both $LG_r(\Gamma)$ and its special fiber $LG_r(\Gamma)_0=\Gr(\mathbf r, M_\Gamma)$ are reduced and Cohen-Macaulay. As a result, $LG_r(\Gamma)=M_r(\Gamma)$ as a scheme.
\end{thm}
\begin{proof}
By Proposition \ref{prop:indecomp of loc linear indep} (3), we have a rational and dominant morphism  $$\prod_{1\leq i\leq r}LG_1(\Gamma)\dashrightarrow LG_r(\Gamma)$$ induced by taking the direct sum of the dimension-$1$ subspaces, where the product on the left is over $R$. Since $LG_1(\Gamma)$ is irreducible by Theorem \ref{thm:Faltings}, so is the product. Hence $LG_r(\Gamma)$ is irreducible.

The rest of the proof is similar to \cite[Theorem 4.1]{helm2008flatness}. We first show the Cohen-Macaulayness of $\Gr(\mathbf r,M_\Gamma)$ by induction on $|\Gamma|$. The base case $|\Gamma|=1$ trivial since $LG_r(\Gamma)$ is a Grassmannian, and the case $|\Gamma|=2$ is covered in \cite[Theorem 4.1]{helm2008flatness}. We now assume $|\Gamma|\geq 3$ and let $[L]$ be a leaf of $T$ and $\Gamma'=\Gamma\backslash\{[L]\}$. Let $[L']$ be the lattice class adjacent to $[L]$ in $\Gamma$ and $\Gamma''=\{[L],[L']\}$. According to Theorem \ref{thm:linked degeneration stratification} (2), $\Gr(\mathbf r,M_\Gamma)$ has pure dimension $r(d-r)$, and so are $\Gr(\mathbf r,M_{\Gamma'})$ and $\Gr(\mathbf r, M_{\Gamma''})$. It follows that $\Gr(\mathbf r, M_\Gamma)$ is a local complete intersection in $\Gr(\mathbf r,M_{\Gamma'})\times \Gr(\mathbf r,M_{\Gamma''})$. By the inductive hypothesis and \cite[\href{https://stacks.math.columbia.edu/tag/045Q}{Tag 045Q}]{stacks-project},  $\Gr(\mathbf r,M_{\Gamma'})\times \Gr(\mathbf r, M_{\Gamma''})$ is Cohen-Macaulay, hence so is $\Gr(\mathbf r,M_{\Gamma})$. 

By Remark \ref{rem:prelinked 2} and Theorem \ref{thm:linked degeneration stratification} (2), the simple points of $\Gr(\mathbf r,M_\Gamma)$ as a prelinked Grassmannian are dense. Thus, according to \cite[Proposition A.2.2]{Ohrk}, $\Gr(\mathbf r,M_\Gamma)$ is generically smooth and hence 
generically reduced. Therefore, $\Gr(\mathbf r,M_\Gamma)$ is reduced by Cohen-Macaulayness. The reducedness and flatness of $LG_r(\Gamma)$ now follows from the irreducibility and \cite[Lemma 6.13]{Olls}. The Cohen-Macaulayness of $LG_r(\Gamma)$ is a consequence of \cite[Cor., page 181]{matsumura_1987}.
\end{proof}

Note that we proved in addition to the main theorem of \cite{habich2014mustafin} that the Mustafin degeneration of a locally linearly independent configuration is Cohen-Macaulay with reduced special fiber.

\section{Application to Limit Linear Series}\label{sec:connection to lls}
In this section we investigate the connection between linked  Grassmannians and moduli spaces of limit linear series on nodal curves. As we shall see, the moduli space of limit linear series admits a natural map from the space of \textit{linked linear series} (Definition \ref{defn:linked linear series}), which, up to twisting by an sufficient ample line bundle, can be written as an union of linked  Grassmannians. Consequently, we derive a criterion for the smoothing of limit linear series.  

Throughout this section we assume that $R$ is a complete discrete valuation ring with fraction field $K$ and algebraically closed residue field $\kappa$. All curves we consider are assumed proper, (geometrically) reduced and connected, and at worst nodal. Furthermore, all irreducible components of a curve are smooth.

\addtocontents{toc}{\protect\setcounter{tocdepth}{2}}
\subsection{Definition of limit linear series}\label{subsec: limit linear series}
We recall the notion of limit linear series on nodal curves. We will use Osserman's notion and focus on curves with \textit{trivial chain structure}, which is much easier to phrase than the non-trivial case. See the precise definition of the later in \cite{osserman2019limit}. 
Meanwhile, we would like to mention that, when dealing with degeneration of linear series, it is possible to replace limit linear series of non-trivial chain structures (when they appear) with the ones of trivial chain structures; however, the underlying curve will be more complicated: it is obtained from the curve of the former by inserting chains of rational curves.

Unless otherwise stated, all definitions in this subsection are from \cite{osserman2019limit}. Let $X_0$ be a nodal curve over $\kappa$. Let $G$ be the dual graph of $X_0$ and $Z_v$ the irreducible component of $X_0$ corresponding to $v\in V(G)$. Let $Z^c_v$ be the closure of $X_0\backslash Z_v$.

The set of multidegrees on $X_0$ is in one-to-one correspondence with the set of divisors on $G$ in a natural way. We say that a multidegree $w$ is obtained from $w'$ by a \textit{twist} at $v\in V(G)$ if the divisor associated to $w$ is obtained from $w'$ as follows: if $v'$ is adjacent to $v$, we increase the degree of $w'$ at $v'$ by one; we decrease the degree of $w'$ at $v$ by the number of vertices adjacent to $v$. In this case we also say that $w'$ is obtained from $w$ by a \textit{negative twist} at $v$.

\begin{defn}\label{defn:concentrated multidegree}
A multidegree $w$ is \textbf{concentrated} on $v$ if there is an ordering on $V(G)$ starting at $v$, and such that for each subsequent vertex $v'$, we have that $w$ becomes negative in vertex $v'$ after taking the composition of the negative twists at all previous vertices. 
\end{defn}

We relate the combinatorial notions to algebraic operations, starting from enriched structures.

\begin{defn}\label{enriched structure}
An \textbf{enriched structure} on a nodal curve $X_0$ consists of the data, for each $v\in V(G)$, of a line bundle $\mathscr O_v$ on $X_0$ and a section $s_v\in \Gamma (X_0,\mathscr O_v)$, satisfying: 

(1) for any $v\in V(G)$ we have $\mathscr O_v|_{Z_v}\cong \mathscr O_{Z_v}(-(Z_v^c\cap Z_v))$ and $ \mathscr O_v|_{Z_v^c}\cong \mathscr O_{Z_v^c}(Z_v^c\cap Z_v)$; 

(2) $\bigotimes_{v\in V(G)}\mathscr O_v\cong \mathscr O_{X_0}$.

(3) $s_v$ vanishes precisely along $Z_v$.
\end{defn}

Now let $(\mathscr O_v,s_v)_{v\in V(G)}$ be an enriched structure on $X_0$. 

\begin{nt}\label{nt:twisting graph}
Fix a multidegree $w_0$ on $X_0$. Let $G(w_0)$ be the directed graph with vertex set 
$V(G(w_0))\subset \mathbb Z^{V(G)}$
 consisting of all multidegrees obtained from $w_0$ by a sequence of twists, and an edge from $w$ to $w'$ if $w'$ is obtained from $w$ by twisting at any vertex of $G$.
Given $w,w'\in V(G(w_0))$, let $P=(w;v_1,...,v_m)$ be a minimal path from $w$ to $w'$ in $G(w_0)$, where the vertex $v_i$ indicates the edge in $G(w_0)$ corresponding to twisting at $v_i$, we set $$\mathscr O_{w,w'}=\bigotimes_{i=1}^m\mathscr O_{v_i}\mathrm{\ and\ } s_{w,w'}=\bigotimes_{i=1}^m s_{v_i}. $$
\end{nt}  

The following proposition ensures that the notations $\mathscr O_{w,w'}$ and $s_{w,w'}$ are well-defined.

\begin{prop}\label{prop:twisting minimal path}\cite[Proposition 2.12]{osserman2019limit}
In the minimal path $P(w;v_1,...,v_m)$ from $w$ to $w'$, the number $m$ and vertices $v_i$ are uniquely determined up to reordering. More generally, paths $P(w,v'_1,...,v'_{m'})$ and $P(w,v''_1,...,v''_{m''})$ starting from $w$ have the same endpoint if and only if the multisets of the $v'_i$ and $v''_i$ differ by a multiple of $V(G)$.
\end{prop}


\begin{nt}\label{nt:twisting map}
Suppose $\mathscr L$ is a line bundle on $ X_0$ of multidegree $w_0$. For any $w\in V(G(w_0))$ set $\mathscr L_w=\mathscr L\otimes \mathscr O_{w_0,w}$. Take also $w'\in V(G(w_0))$. Let $P=(v_1,...,v_m)$ be a minimal path from $w$ to $w'$ as in Notation \ref{nt:twisting graph}. We have a natural map $f_{w,w'}\colon \mathscr L_{w}\rightarrow \mathscr L_{w'}$ induced by multiplying with $s_{w,w'}$. 
\end{nt}  
 
We now have all the ingredients to define limit linear series.

\begin{defn}\label{defn:limit linear series}
Let $X_0$ and $G$ be as above. Fix a multidegree $w_0$ with total degree $d$, and fix a number $r<d$. Choose an enriched structure $(\mathscr O_v,s_v)_v$ on $X_0$, and a tuple $(w_v)_{v\in V(G)}\subset V(G(w_0))$ of multidegrees on $X_0$ such that $w_v$ is concentrated on $v$. Let $\ov G(w_0)$ be the subgraph of $G(w_0)$ consisting of multidegrees $w$ in $V(G(w_0))$ such that, for all $v\in V(G)$, $w_v$ can be obtained from $w$ by twisting vertices other than $v$. A \textit{limit linear series} on $X_0$ consists of a line bundle $\mathscr L$ of multidegree $w_0$ on $X_0$ together with subspaces $V_v\subset \Gamma(X_0,\mathscr L_{w_v})$ of dimension $(r+1)$  such that for all $w\in V(\overline G(w_0))$, the kernel of the linear map 
\begin{equation}\label{eq:limit linear series}  \Gamma(X_0,\mathscr L_w)\to\displaystyle\bigoplus_{v\in V(G)}\Gamma (X_0,\mathscr L_{w_v})/V_v 
\end{equation}
induced by $\op_vf_{w,w_v}$  has dimension at least $r+1$.
\end{defn}

According to \cite[Corollary 2.23]{osserman2017limit} and \cite[Proposition 3.8]{osserman2019limit}, the definition of limit linear series above is equivalent to the one defined in \cite{osserman2019limit}, which is independent of the choice of multidegrees $(w_v)_v$. 

We next introduce the notion of linked linear series, which is closely related to limit linear series.

\begin{defn}{\cite{Ohrk}}\label{defn:linked linear series}
Use the same notation as in Definition \ref{defn:limit linear series}. A \textit{linked linear series} on $X_0$ consists of a line bundle $\mathscr L$ on $X_0$ of multidegree $w_0$ together with subspaces $V_w\subset H^0(X_0,\mathscr L_w)$ of dimension $(r+1)$ for all $w\in V(\overline G(w_0))$ such that 
\begin{equation}\label{eq:linked linear series}
   f_{w,w'}(V_w)\subset V_{w'}\mathrm{\ for\ all\ } w,w'\in V(\overline G(w_0)). 
\end{equation}
\end{defn}

\begin{rem}\label{rem:linked to limit}
Suppose $w_v\in V(\ov G(w_0))$ for all $v$, which is possible according to Remark \ref{rem:equally tropical convex} later. Given a linked linear series $(V_w)_{w\in V(\overline G(w_0))}$, we get immediately a limit linear series by setting $(V_v)_{v\in V(G)}=(V_{w_v})_{v\in V(G)}$. This actually gives a forgetful map from the moduli space of linked linear series to the moduli space of limit linear series.
\end{rem}

By convention, we also denote a limit/linked linear series by a limit/linked $\mathfrak g^r_d$ when the degree and rank are specified.

\subsection{Tropical convexity of the set of multidegrees of limit linear series}\label{subsec:tropical complexity of vgw0}
Let us label the vertex of $G$ as $v_0,v_1,...,v_n$. Recall from \cite{develin2004tropical} that we have the tropical projective space $\mathbf{TP}^n:=\mathbb R^{n+1}/\mathbb R\cdot\mathbf 1$, and a subset $S$ of $\mathbf{TP}^n$ is \textit{tropically convex} if for any $(x_0,...,x_n)$ and $(x'_0,...,x'_n)$ in $S$, we have $$(\min(a+x_0,b+x'_0),...,\min(a+x_n,b+x'_n)\in S
\mathrm{\ for\ all\ }a,b\in\mathbb R.$$

We can identify $V(G(w_0))$ with the integral points in $\mathbf{TP}^n$ as follows.  If $w\in V(G(w_0))$ is obtained from $w_0$ by subsequently twisting $x_{w,j}$ times at $v_j$, then $w$ is identified with $(x_{w,0},...,x_{w,n})$. This is well-defined by Proposition \ref{prop:twisting minimal path}.

\begin{defn}\label{defn:integral tropical convexity}
We call a set $S$ of lattice points in $\mathbf{TP}^n$ \textit{integrally tropically convex} if it is the set of all lattice points in a tropically convex set. The \textit{integral tropical convex hull} of a lattice set $S$ is the smallest integrally tropically convex set that contains $S$.
\end{defn}

It is straightforward to verify that integral tropical convexity of a subset of $V(G(w_0))$ is independent of the choice of $w_0$. 

\begin{nt}\label{nt:twist coefficients}
For $0\leq i\leq n$, suppose  $w_{v_i}$ is obtained from $w_0$ by twisting $a_{i,j}\geq 0$ times at $v_j$ successively for $0\leq j\leq n$. Then $w_{v_i}=(a_{i,0},...,a_{i,n})$. Let $V(G)^{\mathrm{conv}}\subset \mathbf{TP}^n$ be the integral tropical convex hull of all $w_v$.
\end{nt}

\begin{prop}\label{prop:tropical convexity of vgw0}
$V(\overline G(w_0))$ is integrally tropically convex. Moreover, we have $V(\overline G(w_0))\subset V(G)^{\mathrm{conv}}$, and $V(\overline G(w_0))= V(G)^{\mathrm{conv}}$ if and only if
\begin{equation} \label{eq:moduli of lls and linked flags 1}
a_{k,i}-a_{k,j}\geq a_{i,i}-a_{i,j} \mathrm{\ for\ all\ } 0\leq i,j,k\leq n
.\end{equation}
\end{prop}
\begin{proof}
For each $w\in V(G(w_0))$, fix a path in $V(G(w_0))$ from $w_0$ to $w$ and let $x_{w,j}\geq 0$ be the number of twists of $v_j$ in the path.  
Then for each $0\leq i\leq n$, $w_{v_i}$ is obtained from $w$ by twisting $a_{i,j}-x_{w,j}$ times at $v_j$ for each $0\leq j\leq n$. We have $w\in V(\overline G(w_0))$ if and only if $a_{i,i}-x_{w,i}\leq a_{i,j}-x_{w,j}$ for all $0\leq i\leq n$ and $0\leq j\leq n$ by Proposition \ref{prop:twisting minimal path}. In other words, we must have 
\begin{equation}\label{eq:moduli of lls and linked flags 2}
x_{w,i}-x_{w,j}\geq a_{i,i} -a_{i,j}\mathrm{\ for\ all\ }0\leq i,j\leq n. \end{equation} 

Note that $w$ is identified with $(x_{w,0},...,x_{w,n})$ in $\mathbf{TP}^n$. It is easy to see that $V(\overline G(w_0))$ is integrally tropically convex, since each single inequality in (\ref{eq:moduli of lls and linked flags 2}) defines a tropically convex set, and so is their intersection. On the other hand, for $(y_0,y_1,...,y_n)\in V(\overline G(w_0))$, by (\ref{eq:moduli of lls and linked flags 2}) we have 
$$(y_0,y_1,...,y_n)=\min_{0\leq i\leq n}\Big( (a_{i,0},a_{i,1},...,a_{i,n})+(y_i-a_{i,i})\cdot(1,1,...,1)\Big),$$ where by minimum we mean taking the coordinate-wise minimum.
Hence $V(\overline G(w_0))\subset V(G)^{\mathrm{conv}}$. 

If (\ref{eq:moduli of lls and linked flags 1}) is satisfied then we immediately have $w_{v_i}\in V(\overline G(w_0))$. Hence $V(\overline G(w_0))= V(G)^{\mathrm{conv}}$.
\end{proof}

\begin{rem}\label{rem:equally tropical convex}
Condition (\ref{eq:moduli of lls and linked flags 1}) can be satisfied if we choose $w_v$ ``sufficiently concentrated" on $v$. More precisely, given a tuple $(w_v)_v$ of concentrated multidegrees, replace each $w_v$ with $w'_{v}$ obtained from $w_v$ by negatively twist sufficiently many times at $v$; then we get a tuple $(w'_v)_v$ of concentrated multidegrees that satisfies Condition (\ref{eq:moduli of lls and linked flags 1}). In particular, the new $w'_v$ is contained in the new $V(\ov G(w_0))$, which is the integral tropical convex hull of all $w'_v$s.
\end{rem}

\subsection{The moduli space of limit linear series and smoothing property}\label{subsec:smoothing of limit linear series}We first recall the notion of a regular smoothing family as in \cite{osserman2019limit}.

\begin{defn}\label{defn:smoothing family} 
We say that a flat and proper family $\pi\colon X\rightarrow B=\spec(R)$ of curves is a \textit{regular smoothing family} if (1) $X$ is regular and the generic fiber $X_\eta$ is smooth; (2) the special fiber $X_0$ of $\pi$ is a (split) nodal curve; and (3) $\pi$ admits sections through every component of $X_0$.
\end{defn}

Since $R$ is complete, the reduction map from $X_\eta(K)$ to the smooth locus of $X_0$ is surjective according to \cite[Proposition 10.1.40(a)]{liu2002algebraic} (This is used in Proposition \ref{prop:moduli of lls and linked flags} and Theorem \ref{thm:smoothing of lls}). Fix $w_0,d$ and $r$ as in Definition \ref{defn:limit linear series}, and recall that we denoted the dual graph of $X_0$ by $G$ and components by $(Z_v)_{v\in V(G)}$. Fix also concentrated multidegrees $(w_v)_v$ that satisfy (\ref{eq:moduli of lls and linked flags 1}), namely that $w_v\in V(\ov G(w_0))$ for all $v\in V(G)$, which is possible by Remark \ref{rem:equally tropical convex}.  We recall the construction of the moduli space of (limit) $\mathfrak g^r_d$s on $X/B$. The enriched structure on $X_0$ is naturally chosen to be $\mathscr O_v=\mathscr O_X(Z_v)|_{X_0}$ and $s_v=1|_{X_0}$. For a multidegree $w$ on $G$ of total degree $d$ denote by $\mathrm{Pic}^w(X/B)$ the moduli scheme of line bundles of relative degree $d$ over $B$ which have multidegree $w$ on $X_0$.

Let $\widetilde{\mathscr L}_w$ be a universal bundle over $\mathrm{Pic}^w(X/B)\times_B X$.
Take an effective divisor $D=\sum_{v\in V(G)}D_v$ on $X$ such that $D_v$ is a union of sections of $X/B$ that pass through $Z_v$ and avoid the nodes of $X_0$.
Assume $D$ is ``\textit{sufficiently ample}," in other words, $d_v=\deg D_v$ is big enough relative to all $w$ in $V(\overline G(w_0))$ and the genus $g_v$ of $Z_v$. In fact, we will see later that it is enough for us if for all $w\in V(\overline G(w_0))$ and $\mathscr L$ a line bundle on $X_0$ with multidegree $w$ (resp. for all $\mathscr L$ a line bundle on $X_\eta$ with degree $d$), we have $h^1(X_0,\mathscr L(D_0))=0$ (resp. $h^1(X_\eta,\mathscr L(D_\eta))$=0), where $D_0$ (resp. $D_\eta$) is the special fiber (resp. generic fiber) of $D$. Denote  $\tilde d=\sum_v d_v$. 
  
Let $\mathcal P:=\mathrm{Pic}^{w_0}(X/B)$. Consider the diagram:
$$
\begin{tikzcd}
\mathcal P\times_B X\rar\dar&\mathrm{Pic}^{w}(X/B)\times_B X\dar{p_w}\rar{\pi_w}&X\\
\mathcal P\rar{q_w} & \mathrm{Pic}^{w}(X/B).
\end{tikzcd}
$$
Here $q_w$ is induced by tensoring with $\mathscr O_X(Z_v)$ each time $v$ appears in the minimal path in $G(w_0)$ from $w_0$ to $w$ (when restricted to $X_0$, this is just tensoring with $\mathscr O_{w_0,w}$ in Notation \ref{nt:twisting graph}). For simplicity we denote the pullbacks of the denoted maps above by themselves, if there's no confusion. 
Let $\mathcal L_w=q_w^*(\widetilde{\mathscr L}_{w}\otimes\pi_w^*\mathcal O_X(D))$ and $\mathcal E_w=p_{w*}\mathcal L_w$. For $v\in V(G)$ let $\mathcal L_v=q_{w_v}^*(\widetilde{\mathscr L}_{w_v}\otimes \pi_{w_v}^*\mathcal O_X(D)|_{D_v})$ and $\mathcal E_v=p_{w_v*}\mathcal L_v$ over $\mathcal P$. Then $\mathcal E_w$ (resp. $\mathcal E_v$) is a rank-$(d+\tilde d-g+1)$ (resp. rank-$ d_v$) vector bundle by the choice of $d_v$ and \cite[\S 0.5]{mumford1994geometric}. Let $\Gr(r+1,\mathcal E_w)$ be the relative Grassmannian over $\mathcal P$, and $\mathcal G^1$ be the product of all $\Gr(r+1,\mathcal E_{w_v})$s over $\mathcal P$, where $v$ runs over $V(G)$. Similarly, let $\widetilde{\mathcal G}^1$ be the product of all $\Gr(r+1,\mathcal E_w)$s over $\mathcal P$, where $w$ runs over $V(\overline G(w_0))$. Then $\mathcal G^1$ (resp. $\widetilde{\mathcal G}^1$) is the ambient space inside which we will define the moduli space of limit (resp. linked) linear series. To reduce the notation, in the rest of construction, for the pullback of a vector bundle, we will not mention the morphism of the pullback but only specify the scheme that the vector bundle lies on.

Let $\mathcal V_w$ be the universal subbundle on $\Gr(r+1,\mathcal E_w)$ and $\mathcal G^2$ be the locus in $\mathcal G^1$ where 
\begin{equation}\label{eq:determinantal condition of limit linear series}
\mathcal E_w\rightarrow\bigoplus_{v\in V(G)} \mathcal E_{w_v}/\mathcal V_{w_v}
\end{equation}
has rank at most $d+\tilde d-g-r$ for any $w\in V(\overline G(w_0))$, where the map $\mathcal E_w\rightarrow\mathcal E_{w_v}$ is induced by multiplying with $1\in \mathscr O_X(Z_u)$ each time $u$ appears in the minimal path in $G(w_0)$ from $w$ to $w_v$. (Again, when restricted to $X_0$, this is just $f_{w,w_v}$ in Notation \ref{nt:twisting map} up to tensoring with the special fiber of $D$.) 
Accordingly, let $\widetilde{\mathcal G}^2$ be the locus in $\widetilde{\mathcal G}^1$ over which the composition of the morphisms 
\begin{equation}\label{eq:determinantal condition of linked linear series}
\mathcal V_w\hookrightarrow\mathcal E_w\rightarrow \mathcal E_{w'}/\mathcal V_{w'}
\end{equation}
vanishes for all $w,w'\in V(\overline G(w_0))$. Now conditions (\ref{eq:determinantal condition of limit linear series}) and (\ref{eq:determinantal condition of linked linear series}) match with conditions (\ref{eq:limit linear series}) and (\ref{eq:linked linear series}) respectively. The only issue now is that we are tensoring everything with $\mathscr O_X(D)$ in the beginning. Hence,
let $\mathcal G$ (resp. $\widetilde{\mathcal G}$) be the locus in $\mathcal G^2$ (resp. $\widetilde {\mathcal G}^2$) where the map $\mathcal V_{w_v}\rightarrow \mathcal E_v$ vanishes identically for each $v\in V(G)$. Then $\mathcal G$ and $\widetilde{\mathcal G}$ are the desired moduli spaces. Namely, the generic fiber $\mathcal G_\eta$ (resp. $\widetilde {\mathcal G}_\eta$) is the moduli space of $\mathfrak g^r_d$s on $X_\eta$ and the special fiber $\mathcal G_0$ (resp. $\widetilde {\mathcal G}_0$) parametrizes limit (resp. linked) $\mathfrak g^r_d$s on $X_0$ of multidegree $w_0$.

\begin{rem}\label{rem:different scheme structure of moduli of limit linear series}
It is unclear whether the scheme structure of $\mathcal G$ agrees with the moduli space constructed in \cite{osserman2019limit}, although they are the same as topological spaces. The main subtlety is that the determinantal condition in (\ref{eq:determinantal condition of limit linear series}) for $\mathcal G$ is imposed for all $w\in V(\overline G(w_0))$, whereas in \cite{osserman2019limit} it is imposed for all $w\in V(G(w_0))$. See for example the proof of \cite[Proposition 3.2.7]{lieblich2019universal}. Nevertheless, the proof of our smoothing theorem will only involve dimension estimation, hence  the scheme-structure of the moduli space is irrelevant. 
\end{rem}

We next prove a smoothing property of limit linear series on $X_0$ under certain technical assumptions. This is essentially a consequence of dimension estimation of $\mathcal G$. Since it is an intersection of determinantal loci in $\mathcal G^2$, we need to first examine the dimension of $\mathcal G^2$. To do this, note that there is a natural forgetful map $\tilde \pi\colon \widetilde {\mathcal G}^2\rightarrow \mathcal G^2$ as explained in Remark \ref{rem:linked to limit}, and recall that the notion of limit linear series is independent of the choice of concentrated multidegrees.

\begin{prop}\label{prop:moduli of lls and linked flags}
Let $(w_v)_v$ be a set of concentrated multidegrees that satisfy (\ref{eq:moduli of lls and linked flags 1}). Then $\widetilde{\mathcal G}^2$ is covered by linked  Grassmannians. More precisely, let $s\colon B\rightarrow \mathcal P$ be any section of $\mathcal P\rightarrow B$, then the fiber product $\widetilde{\mathcal G}^2\times_\mathcal PB$ is isomorphic to the linked Grassmannian $LG_{r+1}(\Gamma_s)$ associated to a convex configuration $\Gamma_s$  of lattice (classes) in $\Gamma (X_\eta, L(D_\eta))$, where $ L$ is the line bundle on $X_\eta$ corresponding to the generic point of $s$. 
Moreover, for $w\in V(\overline G(w_0))$ let $  L'_w$ be the extension of $L$ to $X$ with multidegree $w$ on $X_0$ and $  L_w= L'_w(D)$, then $\Gamma_s$ is the convex hull of $\{\Gamma(X, L_{w_v})\}_{v\in V(G)}$.
\end{prop}
\begin{proof}
 By construction, $\widetilde{\mathcal G}^1\times_\mathcal PB$ is the product over $B$ of the Grassmannians $\Gr(r+1,\Gamma(X, L_w))$ for $w\in  V(\ov G(w_0))$. For $w'\in V(\overline G(w_0))$, suppose the minimal path in $V(G(w_0))$ from $w$ to $w'$ contains $a_v$ twists at $v$, then the twisting map from $ L_w$ to $L_{w'}$ is just the inclusion $  L_w\hookrightarrow    L_w(\sum a_vZ_v)\simeq  L_{w'}$. Hence, by the definition of $\widetilde {\mathcal G}^2$, it remains to show that the configuration $\{\Gamma(X, L_w)\}_{w\in V(\overline G(w_0))}$ is the convex hull of $\{\Gamma(X, L_{w_v})\}_{v\in V(G)}$.

Note that the intersection of global sections 
$$\Gamma(X, L_{w_0}(\sum a_vZ_v))\cap \Gamma(X, L_{w_0}(\sum b_vZ_v))=\Gamma(X, L_{w_0}(\sum \min\{a_v,b_v\}Z_v))$$
is compatible with taking the minimum of the coefficients of each $Z_v$. Also, we have 
$$\pi\cdot \Gamma(X, L_{w_0}(\sum a_vZ_v))=\Gamma(X, L_{w_0}(\sum (a_v+1)Z_v)).$$
Thus the conclusion reduces to the integral tropical convexity of $V(\ov G(w_0))$ as a set in $\mathbf{TP}^{|V(G)|-1}$, which follows from Proposition \ref{prop:tropical convexity of vgw0} and the choice of $(w_v)_v$.
\end{proof}

\begin{rem}\label{rem:multidegrees to lattices}
We warn the reader that in the proof of Proposition \ref{prop:moduli of lls and linked flags} different multidegrees $w$ may give homothetic lattices $\Gamma(X,L_w)$. For instance, see Proposition \ref{prop:cycle curve example} (2). However, this won't affect the proof. 
\end{rem}


\begin{thm}\label{thm:smoothing of lls}
Let $X/B$ be a smoothing family with special fiber $X_0$. 
Let $w_0$ be a multidegree on $G$ of total degree $d$, and choose concentrated multidegrees $(w_v)_v$ satisfying (\ref{eq:moduli of lls and linked flags 1}). Suppose 
\begin{enumerate}
    \item the map $\tilde \pi\colon \widetilde {\mathcal G}^2\rightarrow \mathcal G^2$ is surjective;
    \item the linked Grassmannians $LG_{r+1}(\Gamma_s)$ in Proposition \ref{prop:moduli of lls and linked flags} are irreducible for all sections $s$.
\end{enumerate}
If the moduli space $\mathcal G_0$ of limit $\mathfrak g^r_d$s of multidegree $w_0$ on $X_0$ has dimension $\rho=g-(r+1)(g-d+r)$ at a given point, then the corresponding limit linear series arises as the limit of a linear series on the geometric generic fiber of $X$. More precisely, $\mathcal G$ has universal relative dimension at least $\rho$ over $B$; and if $\mathcal G_0$ has dimension exactly $\rho $ at a point, then $\mathcal G$ is universal open at that point; if furthermore $\mathcal G_0$ is reduced at a point, then $\mathcal G$ is flat at that point.

Moreover, if $\Gamma_s$ is locally linearly independent for all sections $s$, then (1) and (2) are satisfied.
\end{thm}
\begin{proof}
The fact that local linear independence implies (1) and (2) follows from Theorem \ref{thm:flatness} and Lemma \ref{lem:lifting limit linear series}. We now mimic the proof given in \cite[\S 6]{osserman2019limit}. By \cite[Proposition 3.7]{osserman2015relative} it is enough to show that the map $\mathcal G \rightarrow B$ has universal relative dimension at least $\rho$ over $B$. By \cite[Corollary 5.1]{osserman2015relative} it remains to check that each component of $\mathcal G$, as a closed subscheme of $\mathcal G^1$, has dimension at least $\rho+1$.
 
 Since $\tilde \pi$ 
 is surjective, the fiber product $\mathcal G^2\times_\mathcal P B$ is irreducible for any section $s\colon B\rightarrow \mathcal P$ of $\mathcal P/B$ by condition (2). Hence its special fiber $\mathcal G^2\times_\mathcal P\kappa$ is contained in the closure of its generic fiber $\mathcal G^2\times_\mathcal PK$. Since $K$ is complete, the reduction map from $X_\eta(K)$ to the smooth locus of $X_0$ is surjective. Therefore,
 each point of the special fiber of $\mathcal P= \mathrm{Pic}^{w_0}(X/B)$ is contained in the closure of a $K$-point of the generic fiber $\mathcal P_\eta=\mathrm{Pic}^d(X_\eta)$. 
 It follows that the special fiber $\mathcal G^2_0$ of $\mathcal G^2$ is contained in the union of special fibers of the product $\mathcal G^2\times_\mathcal PB$ with respect to all sections of $\mathcal P/B$, hence contained in the closure of the generic fiber $\mathcal G^2_\eta$. Obviously $\mathcal G^2_\eta$ is a relative Grassmannian over $\mathcal P_\eta$, hence it is irreducible of dimension $d'=g+(r+1)(d+\tilde d-g-r)$. Thus $\mathcal G^2$, as an irreducible closed subscheme of $\mathcal G^1$, has dimension $d'+1$ by \cite[Proposition 6.6]{osserman2015relative}.
 It follows that $\mathcal G$, as an intersection of  determinantal loci in $\mathcal G^2$, has component-wise dimension at least 
$$d'+1-\sum_{v\in V(G)}(r+1)d_v=\rho+1.$$
 \end{proof}

\subsection{Examples of smoothing limit linear series.}\label{subsec:example}
In this subsection we give two examples of a  reducible curve for which the two conditions in Theorem \ref{thm:smoothing of lls} are satisfied. Hence we get the smoothing theorem provided that the moduli space of limit linear series has the expected dimension. Moreover, in the first case the configuration $\Gamma_s$, as in Proposition \ref{prop:moduli of lls and linked flags}, of the induced linked Grassmannian is a convex chain, while in the second case $\Gamma_s$ is star-shaped (see Example \ref{ex:local linear indep} for definitions). As before, let $X/B$ be a regular smoothing family with special fiber $X_0$.

\subsubsection{The two-component case}
\label{subsubsec:two component case}
Suppose $X_0$ only has two components $Z_u$ and $Z_v$. Then all $\Gamma_s$'s in Proposition \ref{prop:moduli of lls and linked flags} are the convex hull of $\Gamma(X,L_{w_u})$ and $\Gamma(X,L_{w_v})$. Alternatively, one can write the convex hull as $\{\Gamma(X,L_{w_u}), \Gamma(X,L_{w_u}(Z_u)),\dotsc,\Gamma(X,L_{w_u}(aZ_u))=\Gamma(X,L_{w_v})\}$ for some integer $a\geq 0$. This is a convex chain, hence locally linearly independent.
In this case the conditions of Theorem \ref{thm:smoothing of lls} are satisfied, and we get the smoothing theorem for limit linear series. Moreover, one can derive from this case a smoothing theorem for limit linear series on curves of pseudo-compact type. The idea is to realize the space of limit linear series on a pseudo-compact curve as a closed subscheme of a product of spaces of limit linear series on curves with two components, where each two-component curve corresponds to a pair of adjacent vertices of the dual graph of the original curve, hence the dimension estimation of the former follows from the estimation of the later, which leads to the smoothing property. See the proof of \cite[Theorem 6.1]{osserman2019limit} for further details.

Note that it is unclear to us whether the lattice configuration associated to a (pseudo)compact type curve is locally linearly independent in general. Nevertheless, there are evidences showing that at least when the multidegree of the limit linear series is relatively small, the corresponding lattice configuration is locally linearly independent, and potentially one can also realize the associated tree of this configuration as a subdivision of the tree induced by the dual graph of the original curve. See Proposition \ref{prop:cycle curve example} (4) for an example.

\subsubsection{The three-rational-component case}\label{subsubsec:cyclic curve}
Let $X_0$ be a curve consisting of three (smooth) rational components $Z_1$, $Z_2$ and $Z_3$. For $i\neq j$, suppose $Z_i$ intersect $Z_j$ at $n_{i,j}$ points $P^k_{i,j}$, where $1\leq k\leq n_{i,j}$. Then $X_0$ is not of pseudo-compact type if $n_{i,j}>0$ for all $i,j$. Let $v_i$ be the vertex in the dual graph of $X_0$ corresponding to $Z_i$.

\tikzset{every picture/.style={line width=0.75pt}}
\begin{figure}[ht]	
\begin{tikzpicture}[x=0.5pt,y=0.5pt,yscale=-0.9,xscale=0.9]
\import{./}{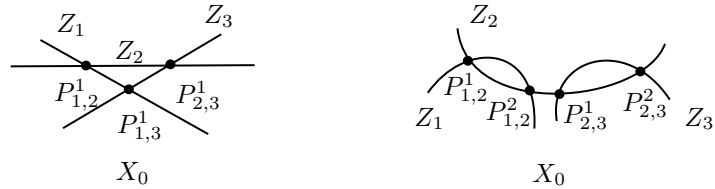}
\end{tikzpicture}	
\caption{Two curves with three rational components: on the left each pair of the components intersect at $n_{i,j}=1$ point; on the right we have $n_{1,3}=0$ while $n_{1,2}=n_{2,3}=2$, and the curve is of pseudo-compact type.}\label{fig: cyclic curve}
\end{figure}

Fix a multidegree $w_0=(a_1,a_2,a_3)$ such that $a_i< 2\min\limits_{1\leq j\leq 3, j\neq i}n_{i,j}$. Let $e_i$ be the multidegree that vanishes at $(v_j)_{j\neq i}$ and has degree 1 on $v_i$. Set 
$$w_{v_i}=w_0+\sum_{j\neq i}n_{i,j}(e_i-e_j)\mathrm{\ and\ }w_i=w_0-\sum_{j\neq i}n_{i,j}(e_i-e_j).$$
Then $w_{v_i}$ (resp. $w_i$) is obtained from $w_0$ by negative twisting (resp. twisting) at $v_i$. 
It is easy to check that $w_{v_i}$ is concentrated on $v_i$, and that $$V(\overline G(w_0))=\{w_0,w_{v_1},w_{v_2},w_{v_3},w_1,w_2,w_3\}$$
is the integral tropical convex hull of $w_{v_1},w_{v_2},w_{v_3}$. See $\overline G(w_0)$ on the left of Figure \ref{fig: twist graph} for an example. 
In the sequel we denote $V(\partial\overline G(w_0))=\{w_1,w_2,w_3\}$.

We choose $D\subset X$ with relative multidegree  $\sum_i(\sum_{j\neq i}n_{i,j}-a_i-1)e_i$ in the construction of the moduli space $\mathcal G$ of limit linear series. Let $L$ be a line bundle on $X_\eta$ induced by a section $s\colon B\rightarrow \mathcal P$ and $L_w$ its extension to $X$ as in Proposition \ref{prop:moduli of lls and linked flags}. Let $\overline L_w$ be the restriction of $L_w$ on $X_0$ and
recall that we have map $f_{w,w'}\colon \ov L_w\rightarrow \ov L_{w'}$ defined similarly as in Notation \ref{nt:twisting map} up to tensoring with the special fiber of $D$. By the first part of the following proposition, the divisor $D\subset X$ is ``sufficiently ample," hence is an appropriate choice for the construction of moduli of limit linear series on $X_0$ with multidegree $w_0$.


\begin{prop}\label{prop:cycle curve example}
Let $L$, $L_w$ and $\overline L_w$ be as above. 
\begin{enumerate}
\item For all $w\in V(\overline G(w_0))$, we have $h^1(X_0,\overline L_w)=0$ and $h^0(X_0,\overline L_w)=\sum_{i<j}n_{i,j}$.

\item For $w\in V(\partial\overline G(w_0))$ we have that $f_{w_0,w}$ induces an isomorphism  $\Gamma(X_0,\overline L_{w_0})\simeq \Gamma(X_0,\overline L_w)$, and $\dim f_{w_0,w_{v_i}}(\Gamma(X_0,\overline L_{w_0}))=\sum_{j\neq i}n_{i,j}$ for $1\leq i\leq 3$. In particular, 
$\Gamma(X, L_{w_0})$ is homothetic to $\Gamma(X, L_w)$ as lattices in $\Gamma(X_\eta,L)$ for $w\in V(\partial\overline G(w_0))$.  

\item The subspaces $f_{w_{v_i},w_0}(\Gamma(X_0,\overline L_{w_{v_i}}))$, which has dimension $n_{i_1,i_2}$ where $\{i_1,i_2\}=\{1,2,3\}\backslash \{i\}$, are linearly independent, hence generates $\Gamma(X_0,\overline L_{w_0})$.

\item The convex configuration $\Gamma_s$ is star-shaped, in particular, locally linearly independent. Moreover, if one of the $n_{i,j}$'s is zero (namely, $X_0$ is of pseudo-compact type), then $\Gamma_s$ is also a convex chain.
\end{enumerate}
\end{prop}
\begin{proof}
(1) The conclusion follows since $\overline L_w$ has multidegree $w+\sum_i(\sum_{j\neq i}n_{i,j}-a_i-1)e_i$.

(2) The first claim follows from the fact that $f_{w_0,w}$ is zero along one component and injective along the other two, hence has trivial kernel on $\Gamma(X_0,\overline L_{w_0})$. Similarly, $f_{w_0,w_{v_i}}$ is injective along $Z_i$ and zero along the other two components, hence its kernel along $\Gamma(X_0,\overline L_{w_0})$ has dimension $n_{i_1,i_2}$ where $\{i_1,i_2\}=\{1,2,3\}\backslash \{i\}$.

(3) By (1) and (2) we have 
$$\dim f_{w_{v_i},w_0}(\Gamma(X_0,\overline L_{w_{v_i}}))=\dim\ker f_{w_0,w_{v_i}}|_{\Gamma(X_0,\overline L_{w_0})}
=n_{i_1,i_2}.$$

For the second claim,
take three vectors $(0,g_1,h_1),(f_2,0,h_2),(f_3,g_3,0)$ in the image of $\Gamma(X_0,\overline L_{w_{v_1}})$, $\Gamma(X_0,\overline L_{w_{v_2}})$ and $\Gamma(X_0,\overline L_{w_{v_3}})$ respectively, where the $i$-th component denote the restriction of the vector to $Z_i$, and suppose they sum to zero. Then $h_1+h_2=0$. As $h_1$ vanishes at $P^k_{1,3}$ and $h_2$ vanishes at $P^k_{2,3}$, both of them vanish at $P^k_{1,3}$ and $P^k_{2,3}$. Hence $h_1=h_2=0$ as they are both of degree $n_{1,3}+n_{2,3}-1$ and there are $n_{1,3}+n_{2,3}$ zero conditions. Similarly $f_1=f_2=0$ and $g_1=g_2=0$. 

(4) Take $\bar e_{i,j}\in \Gamma(X_0,\overline L_{w_{v_i}})$, where $1\leq j\leq n_{i_1,i_2}$, such that $f_{w_{v_i},w_0}(\overline e_{i,j})$ generates $f_{w_{v_i},w_0}(\Gamma(X_0,\overline L_{w_{v_i}}))$.
Lift $\overline e_{i,j}$ to $$e_{i,j}\in \Gamma(X,L_{w_{v_i}})=\Gamma(X,L_{w_0}(-Z_i))\subset \Gamma(X,L_{w_0})\subset \Gamma(X_\eta,L).$$ By (3) and Nakayama's Lemma, $\Gamma(X,L_{w_0})$ is generated by $\{ e_{i,j}\}_{i,j}$. Since $f_{w_{v_i,w_0}}(\overline e_{i,j})$ generates the kernel of $f_{w_0,w_{v_i}}$ on $\Gamma(X_0,\overline L_{w_0})$,  $f_{w_0,w_{v_i}}(\Gamma(X_0,\overline L_{w_0}))$ is generated by the images of all $\overline e_{k,j}$'s, where $k\neq i$. Hence $\Gamma(X_0,\overline L_{w_{v_i}})$ is generated by $\{\bar e_{i,j}\}_j$ and $\{f_{w_k,w_i}(\bar e_{k,j})\}_{k\neq i;j}$. As a result $\Gamma(X,L_{w_0}(-Z_i))$ is generated by $\{e_{i,j}\}_j$ and $\{\pi e_{k.j}\}_{k\neq i;j}$. This proves the first part. Moreover, if $n_{i,j}=0$, we can take $l\in \{1,2,3\}\backslash \{i,j\}$, then $\Gamma(X,L_{w_0}(-Z_l))=\pi\Gamma(X,L_{w_0})$, i.e., the two lattices are homothetic. Therefore, $\Gamma_s$ is a convex chain which consists of three lattice classes.
\end{proof}

The configuration $\Gamma_s$
and the associated quiver $Q(\Gamma_s)$ are illustrated in the right part of Figure \ref{fig: twist graph}. As a direct consequence of Proposition \ref{prop:cycle curve example} (4) and Theorem \ref{thm:smoothing of lls}, we have:

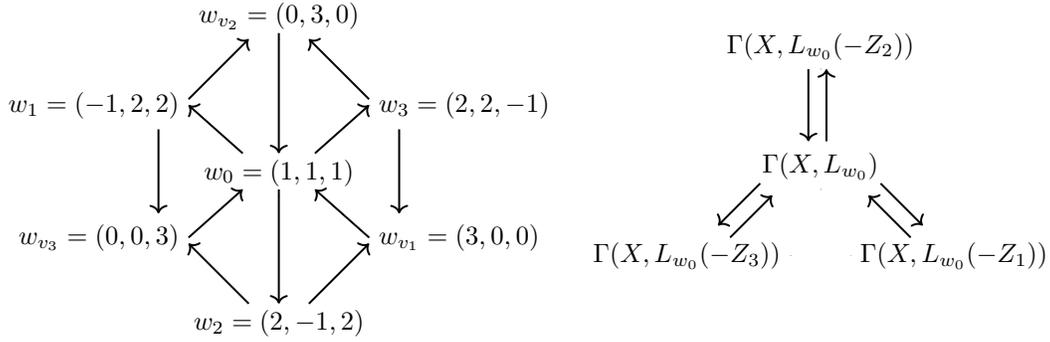
\begin{figure}[H]
\begin{tikzpicture}[scale=0.8]

   \draw[->] (1.1,1.2)--(2,2);
   \draw (0.5,0.5) node[circle, fill=black, scale=0, label=above:{$w_0=(1,1,1)$}]{};
   \draw[->] (-0.1,1.2)--(-1,2);
   \draw[->] (-1,-0.2)--(-0.1,0.6);
   \draw[->] (2,-0.2)--(1.1,0.6);
   \draw[->] (0.5,3.2)--(0.5,1.2);
   \draw[->] (0.5,0.6)--(0.5,-1.3);
   \draw (0.5,3.1) node[circle, fill=black, scale=0, label=above:{$w_{v_2}=(0,3,0)$}]{};
   \draw (0.5,-2) node[circle, fill=black, scale=0, label=above:{$w_2=(2,-1,2)$}]{};
   \draw (-1,2) node[circle, fill=black, scale=0, label=left:{$w_1=(-1,2,2)$}]{};
   \draw (-1,-0.2) node[circle, fill=black, scale=0, label=left:{$w_{v_3}=(0,0,3)$}]{};
   \draw (2,2) node[circle, fill=black, scale=0, label=right:{$w_3=(2,2,-1)$}]{};
   \draw (2,-0.2) node[circle, fill=black, scale=0, label=right:{$w_{v_1}=(3,0,0)$}]{};
   \draw[->] (-1,2.1)--(0,3.1);
   \draw[->] (-1.5,1.6)--(-1.5,0.2);
   \draw[->] (0,-1.3)--(-1,-0.3);
   \draw[->] (2,2.1)--(1,3.1);
   \draw[->] (2.5,1.6)--(2.5,0.2);
   \draw[->] (1,-1.3)--(2,-0.3);
   
   \draw[shift={(9,0)}] (0.5,0.6) node[circle, fill=black, scale=0, label=above:{$\Gamma( X,L_{w_0})$}]{};

   \draw[shift={(9,0)}][->] (-1,-0.2)--(-0.3,0.5);
   \draw[shift={(9,0)}][<-] (-1.2,0)--(-0.5,0.7);
   \draw[shift={(9,0)}][->] (2,-0.2)--(1.3,0.5);
    \draw[shift={(9,0)}][<-] (2.2,0)--(1.5,0.7);
   \draw[shift={(9,0)}][->] (0.3,2.6)--(0.3,1.4);
   \draw[shift={(9,0)}][<-] (0.6,2.6)--(0.6,1.4);
   \draw[shift={(9,0)}] (0.5,2.6) node[circle, fill=black, scale=0, label=above:{$\Gamma( X,L_{w_0}(-Z_2))$}]{};
   \draw[shift={(9,0)}] (0,-0.5) node[circle, fill=black, scale=0, label=left:{$\Gamma(X,L_{w_0}(-Z_3))$}]{};
   \draw[shift={(9,0)}] (1,-0.5) node[circle, fill=black, scale=0, label=right:{$\Gamma(X,L_{w_0}(-Z_1))$}]{};
\end{tikzpicture}
\caption{The left is $\ov G(w_0)$ in the case $n_{i,j}=1$ and $w_0=(1,1,1)$, the arrows represent twisting a vertex of the dual graph of $X_0$. The right is the quiver of $\Gamma_s$.}\label{fig: twist graph}
\end{figure}



\begin{cor}\label{cor:smoothing for cyclic curves}
Let $X_0$ and $w_0$ be as above. Let $X/B$ be a smoothing family with special fiber $X_0$.  Then 
any limit linear series on $X_0$ with multidegree $w_0$ arises as the limit of a linear series on the geometric generic fiber of $X$, if the moduli space of limit linear series is of expected dimension.
\end{cor}

\appendix

\section{A counter-example by G\"ortz}\label{app:counter example}
In \cite{habich2014mustafin}, H\"abich stated an equational description of the Mustafin degeneration $ M_{\underline d} (\Gamma)$ as a subscheme of a product of projective spaces without proof. This was the key ingredient in the proof of the main theorem in the first version of the present paper, namely that linked flag schemes always agree with Mustafin degenerations as schemes. However, this claim is not true in general due to a counter-example communicated to us by Ulrich G\"ortz, which results in a gap in our previous proof. We now illustrate the (simplified) counter-example by G\"ortz.

We first recall the description of $ M_{\underline d} (\Gamma)$ by equations in \cite{habich2014mustafin}, where $\Gamma=\{[L_i]\}_{i\in I}\subset\mathfrak B^0_d$ is a convex collection of homothety classes of lattices in a $d$ dimensional $K$-vector space $V$, and $\underline d=(d_1,...,d_m)$ where $0<d_m<\cdots<d_1<d$ are positive integers. For each lattice $L$, the flag scheme $\mathrm{Flag}_{\underline d}(L)$ is embedded into the product of projective spaces $P=\prod_{j=1}^m\mathbb P^{{d\choose d_j}-1}_R$ by the Pl\"ucker embedding. For each $i$ pick a basis $e^i_1,...,e^i_d$ of $L_i$. Then the respective multihomogeneous coordinates on $P$ are 
$$\{p^{(i)}_{l_1,...,l_{d_j}}=e^i_{l_1}\wedge\cdots\wedge e^i_{l_{d_j}}:1\leq l_1<\cdots <l_{d_j}\leq d\}_{1\leq j\leq m}.$$
Fix a reference lattice $L$ and basis $e_1,...,e_d$. We define $p_{l_1,...,l_{d_j}}$ similarly as above and let $A^i_j$ be the matrix such that $A^i_j\mathbf p_{l_1,...,l_{d_j}}=\mathbf p^{(i)}_{l_1,...,l_{d_j}}$. Then $M_{\underline d} (\Gamma)$ is cut out in $\prod_{[L]\in \Gamma}P$ by the ideal $I_M=\alpha \cap R[...,p^{(i)}_{l_1,...,l_{d_j}},...]$, where $\alpha $ is is the ideal generated over $K$ by all $2\times 2$-minors of the matrices 
$$\begin{bmatrix}
&\vdots&\\\cdots &A^i_jp^{(i)}_{l_1,...,l_{d_j}} &\cdots\\&\vdots&
\end{bmatrix}$$
whose rows are parametrized by $I$ and columns are parametrized by the coordinates of $\mathbb P^{{d\choose d_j}-1}_R$ for all $1\leq j\leq m$, and the ideal generated by the equations cutting out the product of flag schemes $\prod_{[L]\in \Gamma}\mathrm{Flag}_{\underline d}(\Gamma)$.

The gap in this description is that, intuitively, $M_{\underline d}(\Gamma)$ should be by definition a scheme over $R$ that is the closure of its generic fiber; however, the description above only guarantees that $M_{\underline d}(\Gamma)$ is a closed subscheme of such a scheme (cut out by the equations of those flag schemes). In other words, instead of looking at the ideal generated by $I_M$ and the equations cutting out $\prod_{[L]\in \Gamma}\mathrm{Flag}_{\underline d}(\Gamma)$, as above, one should look at the ideal $\alpha'\cap R[...,p^{(i)}_{l_1,...,l_{d_j}},...]$, where $\alpha'$ is the ideal generated by $\alpha$ and the equations (over $K$) cutting out $\prod_{[L]\in \Gamma}\mathrm{Flag}_{\underline d}(V)$.
We illustrate by an example that these two ideals do not always cut out the same subscheme of $\prod_{[L]\in \Gamma}P$.

\begin{ex}\label{ex:counter-example of grotz} (The same example is also discussed in \cite[Remark 2.25]{gora2019local}.)

Consider $d=4$ and  $I=\{1,2\}$. Take a basis $e_1,e_2,e_3,e_4$ of $V$ and let 
$$L_1=\bigoplus_{i=1}^4Re_i\mathrm{\ and\ }L_2=\pi^{-1}Re_1\oplus\bigoplus_{i=2}^4Re_i.$$
Set $m=1$ and $d_1=2$. We have (see Notation \ref{nt:morphisms in a configuration})
$$F_{1,2}=\begin{bmatrix}
\pi&&&\\ &1 &&\\&&1&\\&&&1
\end{bmatrix}\mathrm{\ and\ } 
F_{2,1}=\begin{bmatrix}
1&&&\\ &\pi &&\\&&\pi&\\&&&\pi
\end{bmatrix}.$$
Straightforward calculation shows that the point $(x_1,x_2)\in \Gr(2,4)_\kappa\times \Gr(2,4)_\kappa$ where 
$$x_1=\begin{bmatrix}
1&\\ &1 \\&\\&
\end{bmatrix}
\mathrm{\ and\ } 
x_2=
\begin{bmatrix}
&\\ 1 &\\&1\\1&
\end{bmatrix}$$
is not contained in $LG_2(\Gamma)$, let alone $M_2(\Gamma)$.

On the other hand, after passing to the Pl\"ucker embedding we have $$x_1=(1,0,0,0,0,0)\mathrm{\ and\ } x_2=(0,0,0,1,0,1)$$ and, setting $L=L_2$, we have $A^1_1=\mathrm{diag}(\pi,\pi,\pi,1,1,1)$ and $A^2_1=\mathrm{Id}.$ It is now easy to check that $(x_1,x_2)\in \mathbb P^6_\kappa\times \mathbb P^6_\kappa$ is contained in the subscheme defined by $I_M$ since $A^1_1x_1=(0,...,0)$ on $\kappa$, which provides a contradiction to H\"abich's claim.
\end{ex}

\bibliographystyle{amsalpha}
\bibliography{myrefs}

\providecommand{\bysame}{\leavevmode\hbox to3em{\hrulefill}\thinspace}
\providecommand{\MR}{\relax\ifhmode\unskip\space\fi MR }
\providecommand{\MRhref}[2]{%
  \href{http://www.ams.org/mathscinet-getitem?mr=#1}{#2}
}
\providecommand{\href}[2]{#2}
\begin{thebibliography}{LOTiBZ18}

\bibitem[AB08]{abramenko2008buildings}
Peter Abramenko and Kenneth Brown, \emph{Buildings. theory and applications},
  Springer, 2008.

\bibitem[AF11]{FSMRC}
Marian. Aprodu and Gavril. Farkas, \emph{Koszul cohomology and applications to
  moduli}, {G}rassmannians, moduli spaces and vector bundles \textbf{14}
  (2011), 25--50.

\bibitem[ASS06]{assem2006elements}
Ibrahim Assem, Andrzej Skowronski, and Daniel Simson, \emph{Elements of the
  representation theory of associative algebras: Volume 1: Techniques of
  representation theory}, vol.~65, Cambridge University Press, 2006.

\bibitem[BJ16]{baker2016degeneration}
Matthew Baker and David Jensen, \emph{Degeneration of linear series from the
  tropical point of view and applications}, Nonarchimedean and Tropical
  Geometry, Springer, 2016, pp.~365--433.

\bibitem[CHSW11]{cartwright2011mustafin}
Dustin Cartwright, Mathias H{\"a}bich, Bernd Sturmfels, and Annette Werner,
  \emph{Mustafin varieties}, Selecta Mathematica \textbf{17} (2011), no.~4,
  757--793.

\bibitem[CHZ20]{cotterill2020secant}
Ethan Cotterill, Xiang He, and Naizhen Zhang, \emph{Secant planes of a general
  curve via degenerations}, Geometriae Dedicata \textbf{211} (2020), 165--261.

\bibitem[CIFR12]{cerulli2012quiver}
Giovanni Cerulli~Irelli, Evgeny Feigin, and Markus Reineke, \emph{Quiver
  {G}rassmannians and degenerate flag varieties}, Algebra \& Number Theory
  \textbf{6} (2012), no.~1, 165--194.

\bibitem[DS04]{develin2004tropical}
Mike Develin and Bernd Sturmfels, \emph{Tropical convexity}, Documenta
  Mathematica \textbf{9} (2004), no.~1-27, 7--8.

\bibitem[EH86]{EHL}
David Eisenbud and Joe Harris, \emph{Limit {L}inear {S}eries: {B}asic
  {T}heory}, Inventiones Mathematicae \textbf{85} (1986), 337--371.

\bibitem[EH87]{EHgen}
\bysame, \emph{The {K}odaira {D}imension of the {M}oduli {S}pace of {C}urves of
  {G}enus $\geq$23}, Inventiones Mathematicae \textbf{90, no.2} (1987),
  359--387.

\bibitem[Fal01]{faltings2001toroidal}
Gerd Faltings, \emph{Toroidal resolutions for some matrix singularities},
  Moduli of Abelian Varieties, Springer, 2001, pp.~157--184.

\bibitem[FJP20]{farkas2020kodaira}
Gavril Farkas, David Jensen, and Sam Payne, \emph{The {K}odaira dimensions of
  ${M}_{22}$ and ${M}_{23}$}, arXiv preprint arXiv:2005.00622 (2020).

\bibitem[FKM94]{mumford1994geometric}
John Fogarty, Frances Kirwan, and David Mumford, \emph{Geometric invariant
  theory}, 3rd ed., Springer-Verlag, 1994.

\bibitem[G{\"o}r01]{gortz2001flatness}
Ulrich G{\"o}rtz, \emph{On the flatness of models of certain {S}himura
  varieties of {P}{E}{L}-type}, Mathematische Annalen \textbf{321} (2001),
  no.~3, 689--727.

\bibitem[Gor19]{gora2019local}
Felix Gora, \emph{Local models, {M}ustafin varieties and semi-stable
  resolutions}, arXiv preprint arXiv:1912.12492 (2019).

\bibitem[GY10]{gortz2010supersingular}
Ulrich G{\"o}rtz and Chia-Fu Yu, \emph{Supersingular {K}ottwitz-{R}apoport
  strata and {D}eligne-{L}usztig varieties}, Journal of the Institute of
  Mathematics of Jussieu \textbf{9} (2010), no.~2, 357--390.

\bibitem[H{\"a}b14]{habich2014mustafin}
Mathias H{\"a}bich, \emph{Mustafin degenerations}, Beitr{\"a}ge zur Algebra und
  Geometrie/Contributions to Algebra and Geometry \textbf{55} (2014), no.~1,
  243--252.

\bibitem[He18a]{he2018brill}
Xiang He, \emph{Brill-{N}oether generality of binary curves}, Canadian
  Mathematical Bulletin (2018), 1--21.

\bibitem[He18b]{he2018lifting}
\bysame, \emph{Lifting divisors with imposed ramifications on a generic chain
  of loops}, Proceedings of the American Mathematical Society \textbf{146}
  (2018), no.~11, 4591--4604.

\bibitem[He19]{he2019smoothing}
\bysame, \emph{Smoothing of limit linear series on curves and metrized
  complexes of pseudocompact type}, Canadian Journal of Mathematics \textbf{71}
  (2019), no.~3, 629--658.

\bibitem[HL20]{hahn2020mustafin}
Marvin~Anas Hahn and Binglin Li, \emph{Mustafin varieties, moduli spaces and
  tropical geometry}, manuscripta mathematica (2020), 1--31.

\bibitem[HM82]{harris1982kodaira}
Joe Harris and David Mumford, \emph{On the {K}odaira dimension of the moduli
  space of curves}, Inventiones mathematicae \textbf{67} (1982), no.~1, 23--86.

\bibitem[HO08]{helm2008flatness}
David Helm and Brian Osserman, \emph{Flatness of the linked {G}rassmannian},
  Proceedings of the American Mathematical Society \textbf{136} (2008), no.~10,
  3383--3390.

\bibitem[Liu02]{liu2002algebraic}
Qing Liu, \emph{Algebraic geometry and arithmetic curves}, Oxford graduate
  texts in mathematics, Oxford University Press, 2002.

\bibitem[Liu18]{liu2018limit}
Wen Liu, \emph{Limit linear series on cycle curves}, University of California,
  Davis, 2018.

\bibitem[LO19]{lieblich2019universal}
Max Lieblich and Brian Osserman, \emph{Universal limit linear series and
  descent of moduli spaces}, manuscripta mathematica \textbf{159} (2019),
  no.~1, 13--38.

\bibitem[LOTiBZ18]{liu2018strong}
Fu~Liu, Brian Osserman, Montserrat Teixidor~i Bigas, and Naizhen Zhang,
  \emph{The strong maximal rank conjecture and moduli spaces of curves}, arXiv
  preprint arXiv:1808.01290 (2018).

\bibitem[Mat87]{matsumura_1987}
Hideyuki Matsumura, \emph{Commutative ring theory}, Cambridge Studies in
  Advanced Mathematics, Cambridge University Press, 1987.

\bibitem[MO16]{murray2016linked}
John Murray and Brian Osserman, \emph{Linked determinantal loci and limit
  linear series}, Proceedings of the American Mathematical Society \textbf{144}
  (2016), no.~6, 2399--2410.

\bibitem[Mum72]{mumford1972analytic}
David Mumford, \emph{An analytic construction of degenerating curves over
  complete local rings}, Compositio Mathematica \textbf{24} (1972), no.~2,
  129--174.

\bibitem[Mus78]{mustafin1978}
G.A. Mustafin, \emph{Nonarchimedean uniformization}, Mathematics of the
  USSR-Sbornik \textbf{34} (1978), no.~2, 187.

\bibitem[OiB14]{osserman2014linked}
Brian Osserman and Montserrat~Teixidor i~Bigas, \emph{Linked alternating forms
  and linked symplectic {G}rassmannians}, International Mathematics Research
  Notices \textbf{2014} (2014), no.~3, 720--744.

\bibitem[Oss06]{Olls}
Brian Osserman, \emph{A limit linear series moduli scheme}, Annales de
  l'Institut Fourier \textbf{56, no.4} (2006), 1165--1205.

\bibitem[Oss14]{Ohrk}
\bysame, \emph{Limit linear series moduli stacks in higher rank}, arxiv:
  1405.2937v1 (2014).

\bibitem[Oss15]{osserman2015relative}
\bysame, \emph{Relative dimension of morphisms and dimension for algebraic
  stacks}, Journal of Algebra \textbf{437} (2015), 52--78.

\bibitem[Oss16]{osserman2016dimension}
\bysame, \emph{Dimension counts for limit linear series on curves not of
  compact type}, Mathematische Zeitschrift \textbf{284} (2016), no.~1-2,
  69--93.

\bibitem[Oss19a]{osserman2017limit}
\bysame, \emph{Limit linear series and the {A}mini-{B}aker construction},
  Mathematische Zeitschrift \textbf{293} (2019), 339--369.

\bibitem[Oss19b]{osserman2019limit}
\bysame, \emph{Limit linear series for curves not of compact type}, Journal
  f{\"u}r die reine und angewandte Mathematik \textbf{2019} (2019), no.~753,
  57--88.

\bibitem[RZ16]{rapoport2016period}
Michael Rapoport and Thomas Zink, \emph{Period spaces for p-divisible groups
  (am-141)}, vol. 141, Princeton University Press, 2016.

\bibitem[{Sta}20]{stacks-project}
The {Stacks project authors}, \emph{The stacks project},
  \url{https://stacks.math.columbia.edu}, 2020.

\end{thebibliography}

\end{document}